\documentclass[a4paper,11pt]{article}
\usepackage[utf8x]{inputenc}
\usepackage{amsmath,amsxtra}
\usepackage[varg]{txfonts}
\usepackage[unicode]{hyperref}
\usepackage{framed}
\usepackage[amsmath,hyperref,framed,thmmarks]{ntheorem}
\usepackage{mathrsfs}
\usepackage{graphicx}
\usepackage[oldcommands]{algorithm2e}
\usepackage{enumitem}
\usepackage{nicefrac}

\theoremseparator{.}
\makeatletter
\renewcommand{\newframedtheorem}[1]{%
\theoremprework{\framed\vspace{-1.5ex}}%
\theorempostwork{\vspace{-2ex}\endframed}%
\newtheorem@i{#1}%
}
\makeatother
\newtheorem{theorem}{Theorem}[section]
\newtheorem{proposition}[theorem]{Proposition}
\newtheorem{lemma}[theorem]{Lemma}
\newtheorem{corollary}[theorem]{Corollary}
\newtheorem{definition}[theorem]{Definition}

\theorembodyfont{\upshape}
\theoremsymbol{$\Diamond$}
\newtheorem{remark}[theorem]{Remark}

\newtheorem{conjecture}[theorem]{Conjecture}
\makeatletter
\newtheoremstyle{proof}%
{\item[\theorem@headerfont\hskip\labelsep ##1\theorem@separator]}%
{\item[\ifx\@empty##3\else\theorem@headerfont\hskip \labelsep ##1\ of\ ##3\theorem@separator\fi]}
\makeatother
\theoremheaderfont{\itshape}
\theoremstyle{proof}
\theoremseparator{.}
\theoremsymbol{\ensuremath{\Box}}
\newtheorem{proof}{Proof}
\theoremsymbol{``\ensuremath{\Box}''}
\newtheorem{sketch}{Sketch of proof}
\theoremheaderfont{\sffamily\bfseries}\theorembodyfont{\sffamily\upshape}
\theoremstyle{plain}
\theoremnumbering{alph}
\theoremsymbol{$\Diamond$}
\theoremprework{\small}
\theorempostwork{\normalsize}
\def\bfseries{\fontseries \bfdefault \selectfont \boldmath}

\hypersetup{
breaklinks=true,
colorlinks=true,
linkcolor=blue,
citecolor=blue,
urlcolor=blue,
}


\allowdisplaybreaks

\renewcommand{\theenumi}{\textup{(\roman{enumi})}}
\renewcommand{\labelenumi}{\theenumi}

\makeatletter
\def\@fnsymbol#1{\ensuremath{\ifcase#1\or *\or **\or {**}* \or {**}{**}\else\@ctrerr\fi}}
\makeatother

\DeclareMathSizes{10}{10}{7}{7}

\newcommand{\nfrac}[2]{\nicefrac{#1}{#2}}

\newcommand{\abs}[1]{\ensuremath{\left|#1\right|}}
\newcommand{\br}[1]{\ensuremath{\left(#1\right)}}

\newcommand{\lnorm}[2][2]{\ensuremath{\left\|#2\right\|_{L^{{#1}}}}}
\newcommand{\norm}[1]{\ensuremath{\left\|#1\right\|}}

\newcommand{\seq}[2][\eps>0]{\ensuremath{\br{#2}_{{#1}}}}
\newcommand{\seqn}[2][k]{\ensuremath{\br{{#2}_{#1}}_{#1\in\N}}}

\newcommand{\set}[1]{\ensuremath{\left\{#1\right\}}}
\newcommand{\sett}[2]{\ensuremath{\left\{#1\,\middle|\,#2\right\}}}
\renewcommand{\sp}[1]{\ensuremath{\left\langle#1\right\rangle}}
\newcommand{\sq}[1]{\ensuremath{\left[#1\right]}}

\numberwithin{equation}{section}
\numberwithin{figure}{section}
\numberwithin{table}{section}
\numberwithin{algocf}{section}
\makeatletter
\let\c@table\c@figure
\let\c@algocf\c@figure
\newcommand\captionof[1]{\def\@captype{#1}\caption}
\makeatother

\DeclareFontFamily{U}{dbnsymb}{}
\DeclareSymbolFont{dbnsymb}{U}{dbnsymb}{m}{n}
\DeclareFontShape{U}{dbnsymb}{m}{n}{<1-100>dbnsymb}{}
\DeclareMathSymbol{\overcrossing}{\mathord}{dbnsymb}{033}
\DeclareMathSymbol{\undercrossing}{\mathord}{dbnsymb}{034}

\renewcommand{\a}{\ensuremath{\alpha}}
\renewcommand\angle{\mathop{\mbox{$<\!\!\!)\,$}}\nolimits}
\DeclareMathOperator{\area}{\sH^2}

\DeclareMathOperator{\bra}{bra}
\DeclareMathOperator{\bri}{bri}
\newcommand{\Cn}[1][\omega]{\ensuremath{\mathscr C_{#1}}}

\renewcommand{\d}{\ensuremath{\,\mathrm{d}}}

\newcommand{\ddg}{\g''}
\newcommand{\ddgth}{\ensuremath{\ddg_\th}}
\DeclareMathOperator{\dist}{dist}
\newcommand{\dg}{\g'}

\newcommand{\Eb}{\ensuremath{E_{\mathrm{bend}}}}

\newcommand{\Er}{\ensuremath{\mathcal R}}
\newcommand{\Eth}{\ensuremath{E_\th}}
\newcommand{\eps}{\ensuremath{\varepsilon}}
\newcommand{\g}{\gamma}
\newcommand{\gth}{\ensuremath{\g_\th}}

\renewcommand{\kappa}{\ensuremath{\varkappa}}
\newcommand{\length}{\ensuremath{\mathscr{L}}}

\newcommand{\N}{\ensuremath{\mathbb{N}}}

\renewcommand{\omega}{\mathfrak{k}}
\renewcommand{\P}[1][\psi]{\ensuremath{\mathbb P_{#1}}}

\newcommand{\R}{\ensuremath{\mathbb{R}}}
\newcommand{\K}{\ensuremath{\mathcal{K}}}
\newcommand{\cC}{\ensuremath{\mathscr{C}}}
\newcommand{\sH}{\ensuremath{\mathscr{H}}}
\newcommand{\cT}{\ensuremath{\mathcal{T}}}
\newcommand{\cU}{\ensuremath{\mathcal{U}}}
\newcommand{\refeq}[2]{\ensuremath{\stackrel{\text{\makebox[0cm][c]{\eqref{eq:#1}}}}{#2}}}
\renewcommand{\rho}{\ensuremath{\varrho}}

\newcommand{\rzd}{\ensuremath{(\R/\Z,\R^3)}}
\renewcommand{\S}{\ensuremath{\mathbb{S}}}
\newcommand{\scl}[1][\omega]{\ensuremath{\overline{\Cn[#1]}}}
\DeclareMathOperator{\sign}{sign}

\DeclareMathOperator{\TC}{TC}
\newcommand{\tg}{\tilde\gamma}
\DeclareMathOperator{\tge}{tpc}
\newcommand{\Tge}{\widetilde{\tge}}
\renewcommand{\th}{\ensuremath{\vartheta}}
\newcommand{\tkc}[1][2,b]{\mathcal{T}({#1})}
\newcommand{\uni}[1]{\mathbf{e}_{#1}}
\newcommand{\uvector}[1]{\ensuremath{\frac{\overrightarrow{#1}}{\abs{#1}}}}

\newcommand{\Z}{\ensuremath{\mathbb{Z}}}

\newcommand{\ON}{{\textnormal{on\,\,}}} %domain
\newcommand{\Foa}{{\textnormal{for all\,\,}}} %domain
\newcommand{\Span}{{\textnormal{span\,}}}

\title{The elastic trefoil is the twice covered circle}
\author{Henryk Gerlach\thanks{EPF Lausanne, Switzerland,
\url{henryk.gerlach@gmail.com}}
\and Philipp Reiter\thanks{Fakult\"at f\"ur Mathematik,
Universit\"at Duisburg--Essen, 45117 Essen, Germany, \url{philipp.reiter@uni-due.de}}
\and
Heiko von der Mosel\thanks{Institut f\"ur Mathematik, RWTH Aachen University, Templergraben 55, 52062 Aachen, Germany, \url{heiko@instmath.rwth-aachen.de}}
}

\begin{document}
\maketitle

\begin{abstract}
We investigate  the elastic behavior of knotted loops of springy wire.
To this end we 
minimize the classic bending energy~$\Eb=\int\kappa^2$ together with
a small multiple  of ropelength~$\Er=\textnormal{length}/\textnormal{thickness}$
in order to penalize selfintersection. Our main objective is to characterize
{\it elastic knots}, i.e., 
all limit configurations of
energy minimizers of the total energy 
 $\Eth:=\Eb+\th\Er$ as $\vartheta$ tends to zero.  The elastic unknot
 turns out to be the round circle
 with bending energy $(2\pi)^2$. For all 
 (non-trivial) knot
 classes for which the natural lower bound $(4\pi)^2$ for the bending energy
 is sharp, the respective elastic knot is the twice covered circle. The
 knot classes for which $(4\pi)^2$ is sharp are precisely the 
 $(2,b)$-torus knots for odd $b$ with $|b|\ge 3$
 (containing the trefoil). In particular,
 the elastic trefoil is the twice covered circle.
\end{abstract}



% !TEX root = regelast.tex
\section{Introduction}

The central issue addressed in this paper is the following:
{\rm Knotted loops made of elastic wire spring into some 
(not necessarily unique) stable
configurations when released. 
Can one characterize these configurations?}

There are (at least) three beautiful toy models of such
springy  knots designed by J.~Langer;
see the images in
Figure \ref{fig:springy-knots}.
And one may
ask: why isn't there the springy trefoil? Simply experimenting with
an elastic wire with a hinge reveals the answer: the final shape of
the elastic trefoil would  simply be too boring to play with,
forming two circular
flat loops stacked on top of each other; 
see the image on the bottom right of Figure \ref{fig:springy-knots}.

\begin{figure}
\begin{center}
\begin{tabular}{cc}
\includegraphics[width=0.5\textwidth]{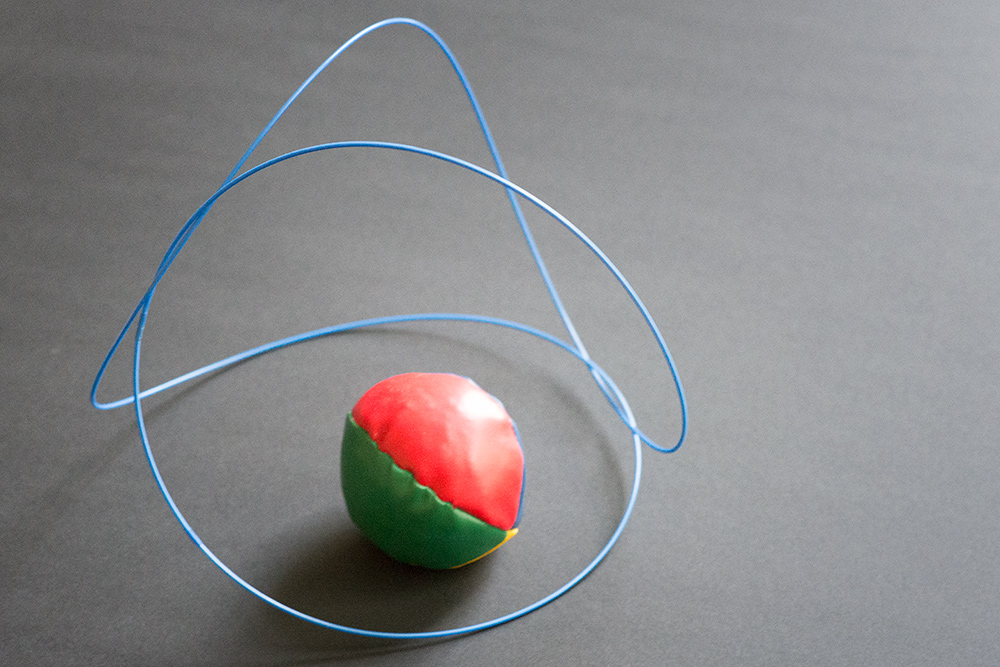} \\
\includegraphics[width=0.5\textwidth]{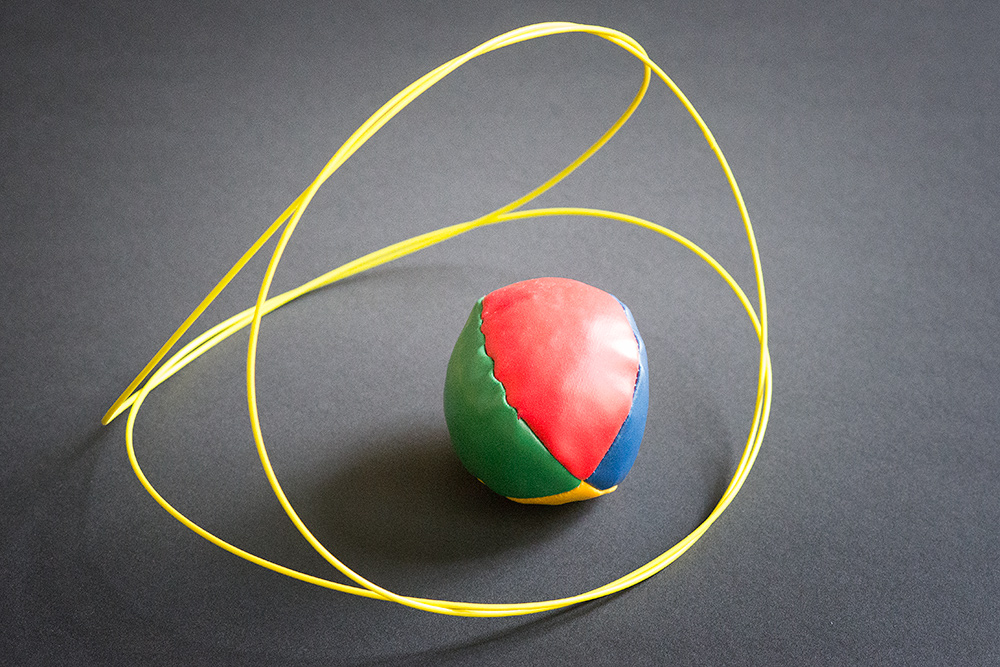} \\
\includegraphics[width=0.5\textwidth]{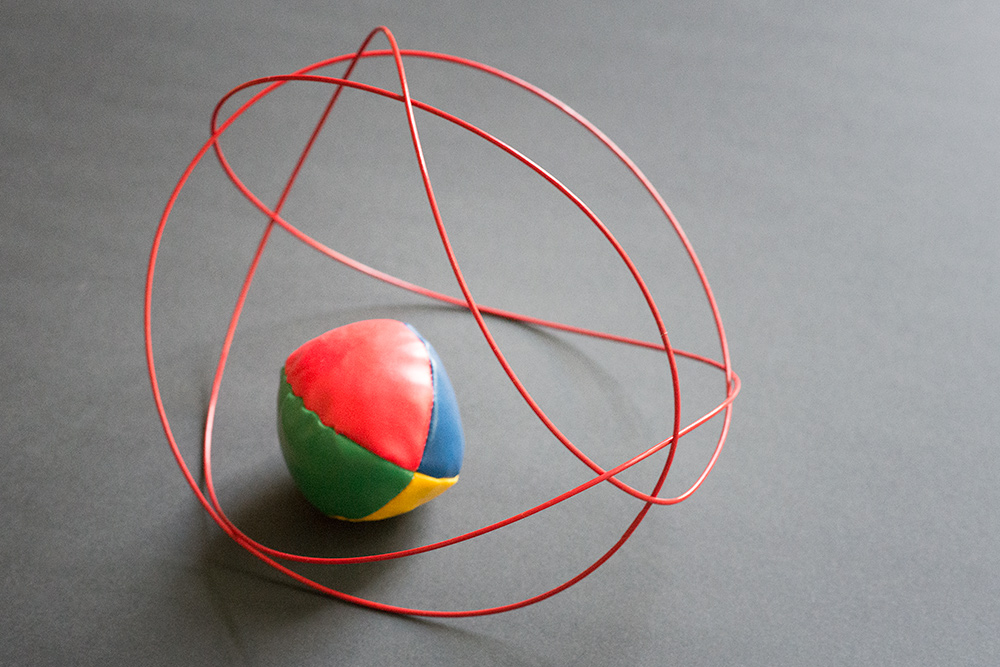}
    \end{tabular}
    \end{center}
    \caption{Springy knots: figure-eight knot, mathematician's loop, and Chinese
    button knot.
    Wire models manufactured by {\sc why knots}, Box 635, Aptos, CA 95003,
    in 1980;
    coloured photographs by B. Bollwerk, Aachen.}
    \label{fig:springy-knots}
    \end{figure}

Mathematically, the classification of elastic knots 
is a fascinating problem, and our aim is to justify the 
behaviour of the trefoil and of more general torus knots
by means of the simplest possible model in elasticity. Ignoring
all effects of extension and   shear the wire
is represented by a sufficiently smooth closed curve $\gamma:\R/\Z\to\R^3$
of unit length and parametrized by arclength (referred to as {\it unit loop}).
We follow 
Bernoulli's approach to consider the {\it bending energy}
\begin{equation}\label{bending}
\Eb(\gamma):=\int_\gamma\kappa^2\d s
\end{equation}
as  the only intrinsic elastic energy---neglecting any additional torsional
effects, and we also exclude external forces and
friction that might be present  in Langer's toy models.
Here, $\kappa=|\gamma''|$ is the classic local curvature of the curve.
To respect a given knot class when minimizing the bending energy
we have to preclude self-crossings. In principle we could add
any self-repulsive {\it knot energy} for that matter, imposing infinite
energy barriers between different knot classes; see, e.g., 
the recent surveys \cite{blatt-reiter-proc,blatt-reiter-mbmb,randy-project_2013,isaac_2014} on such energies and their impact
on geometric knot theory. But  a solid (albeit thin) wire motivates a steric
constraint in form of a fixed (small) thickness of all curves in competition. 
This, and the
geometric rigidity it imposes on the curves lead us to adding a
small amount of 
{\it ropelength} $\Er$ to form the \emph{total energy}
\begin{equation}\label{total-energy}
\Eth := \Eb+\th\Er,\qquad\th >0,
\end{equation}
to be minimized within a prescribed  tame\footnote{A
 knot class is called \emph{tame}
if it contains polygons, i.e., piecewise (affine) linear loops.
Any knot class containing smooth curves is tame, see
Crowell and Fox~\cite[App.~I]{crowell-fox}, and vice versa, any tame knot class
contains smooth representatives. Consequently,
$\cC(\K)\ne\emptyset$ if and only if $\K$ is tame.}
knot class $\K$, that is,
on the class $\cC(\K)$ of all unit loops representing $\K$.
As ropelength is defined as the quotient of length and thickness 
it boils
down for unit loops 
to $\Er(\gamma)=1/\triangle[\gamma]$. 
Following
Gonzalez and Maddocks \cite{gm} the thickness  $\triangle[\cdot]$ may be
expressed as
\begin{equation}\label{thickness}
\triangle[\gamma]:=\inf_{u,v,w\in\R/\Z\atop u\not= v\not=w\not= u}R(\gamma(u),
\gamma(v),\gamma(w)),
\end{equation}
where $R(x,y,z)$ denotes the unique radius of the (possibly degenerate) circle passing
through $x,y,z\in\R^3.$

By means of the direct method in the calculus of variations we show that
in every given (tame) knot class $\K$ and
for every $\th>0$ there is indeed a unit loop $\gamma_\th\in\cC(\K)$
minimizing the total energy $\Eth$ within $\K$;
see Theorem \ref{thm:existence-total} in Section \ref{sect:mini}.

To understand the behaviour of very thin springy knots we investigate
the limit $\th\to 0$. More precisely, we consider arbitrary sequences
$(\gamma_\th)_\th $ of minimizers
in a fixed knot class $\K$ and look at their possible
limit curves $\gamma_0$  as $\th\to 0$. We call any such limit curve
an {\it elastic knot} for $\K$. None
 of these elastic knots is embedded (as we would expect in view of the
 self-contact present in the 
 wire models in Figure \ref{fig:springy-knots})---unless
$\K$ is the unknot, in which case $\gamma_0$ is the once-covered
circle; see Proposition \ref{prop:nonembedded}.
However, it turns out that each elastic knot $\gamma_0$  lies in
the $C^1$-closure of
unit loops representing ${\K}$,
and non-trivial elastic knots can be shown to have strictly smaller
bending energy $\Eb$ than any unit loop in $\cC(\K)$  
(Theorem \ref{thm:existence-limit}). 
This minimizing property  of elastic knots is particularly interesting
for  those
 non-trivial knot classes
$\K$ permitting representatives with bending energy arbitrarily close
to the smallest possible lower bound $(4\pi)^2$ (due to F\'ary's and Milnor's  lower bound $4\pi$ on total curvature
 \cite{fary,milnor}): We can show that for those knot classes
 the {\it only} possible shape of
any elastic knot is that of the twice-covered circle. 
This naturally leads to the question: 

{\it For which knot classes $\K$ do we have $\inf_{\cC(\K)}\Eb=(4\pi)^2$?}

We are going to show that this is true {\it exactly} for the class 
$\cT(2,b)$
of $(2,b)$-torus knots for any odd integer $b$ with $|b|\ge 3.$ Any
other non-trivial knot class has a strictly larger infimum of bending energy.
These facts 
and several other characterizations of $\cT(2,b)$ are
contained in our Main Theorem \ref{thm:elastic-shapes}, from which we can
extract the following complete description of elastic $(2,b)$-torus knots
(including the trefoil):

\begin{theorem}[Elastic $(2,b)$-torus knots]\label{thm:main}
For any odd integer $\abs b\ge 3$
the unique elastic $(2,b)$-torus knot is the twice-covered circle.
\end{theorem}

This result confirms our mechanical and numerical experiments
(see Figure~\ref{fig:low_trefoil_tgA} on the left
and Figure~\ref{fig:trefoil_sim}), as well as the heuristics and
the Metropolis Monte Carlo simulations 
of Gallotti and Pierre-Louis~\cite{gallotti-pierre-louis_2006},
and the numerical gradient-descent results by Avvakumov and
Sossinsky, see~\cite{sossinsky}
and references therein.

Our results especially affect knot classes with
bridge number two (see below for the precise definition)
which in the majority of cases appear in applications,
e.g., DNA knots,
see Sumners~\cite[p.~338]{sumners:dna}. The Main 
Theorem \ref{thm:elastic-shapes}, however, 
implies also that for knots \emph{different} from
the $(2,b)$-torus knots, the respective elastic knot 
is definitely \emph{not} the twice-covered circle.
Similar shapes as in Figure \ref{fig:trefoil_sim}  have been obtained numerically
by Buck and Rawdon~\cite{BR} for a related but different variational problem:
they minimize ropelength with a prescribed curvature bound (using a variant of
the Metropolis Monte Carlo procedure),
see~\cite[Fig.~8]{BR}.

\begin{figure}
\begin{center}
\begin{tabular}{cc}
\includegraphics[width=0.4\textwidth]{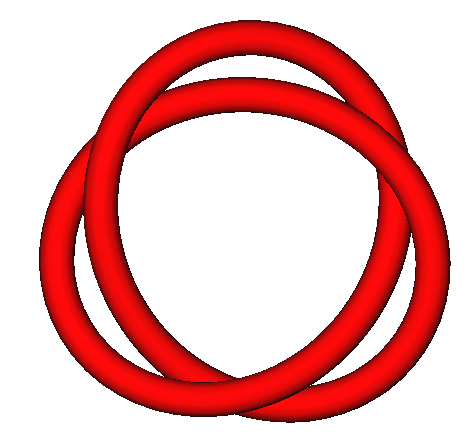} &
   \includegraphics[width=0.4\textwidth]{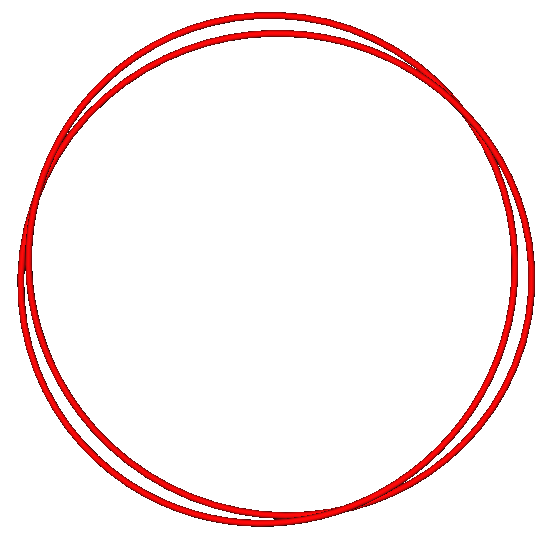} \\
    $\th = 0.1$ & $\th = 0.001$
    \end{tabular}
    \end{center}
    \caption{Minimizers
    of the total energy $\Eth$ in the class of trefoils
    approaching the twice covered circle as $\vartheta$ tends to zero.
    For numerical reasons the ropelength $\Er$ was substituted
    by a repulsive potential, the M\"obius energy introduced by O'Hara~\cite{oha:en}.}
    \label{fig:trefoil_sim}
    \end{figure}

The idea of studying $\g_0$ as a limit configuration  of minimizers
of
the mildly penalized bending energy  goes
back to earlier work of the third author \cite{vdM:meek}, only that
there ropelength in \eqref{total-energy}
is replaced by a self-repulsive potential, like the M\"obius energy introduced
by O'Hara \cite{oha:en}. By means of a Li-Yau-type inequality  for general
loops
\cite[Theorem 4.4]{vdM:meek} it was shown there that for
elastic $(2,b)$-torus knots,
the
maximal multiplicity of double points is three \cite[p.~51]{vdM:meek}.
Theorem \ref{thm:main} clearly shows that this multiplicity
bound is not sharp: the twice-covered circle has infinitely many double points
all of which have multiplicity two.
Lin and Schwetlick~\cite{lin-schwetlick_2010} consider
the gradient flow of
 the elastic energy plus the M\"obius energy scaled by
a certain parameter. However, there is no analysis of the equilibrium 
shapes, and they do not consider the limit case
of sending the prefactor of the M\"obius term to zero.

Directly analyzing
the shape  or even only the regularity of the $\Eth$-minimizers $\gamma_\th$
for positive $\th$  without going to the limit $\th\to 0$
seems much harder because of a~priorily unknown (and possibly
complicated) regions of self-contact that are determined by the minimizers
$\g_\th$ themselves. Necessary conditions were derived by a Clarke-gradient
approach in
 \cite{heiko2} for nonlinearly elastic rods, 
 and for 
 an alternative elastic self-obstacle formulation 
 regularity results were established  in
 \cite{vdM:eke3} depending on the geometry of contact.
If one replaces ropelength in \eqref{total-energy}
by a self-repulsive potential like in \cite{vdM:meek}
 one can prove $C^\infty$-smoothness of $\gamma_\th$
with the deep analytical methods developed
by He~\cite{he:elghf}, and the  second author in various cooperations
\cite{reiter:atme,reiter:rkepdc,blatt-reiter2,blatt-reiter-schikorra_2012}. 
But the corresponding Euler-Lagrange equations for $\Eth$ involving
complicated non-local terms 
do not seem
to give immediate access to determining the shape of $\gamma_\th.$
Notice that directly
minimizing the bending energy $\Eb$ in the $C^1$-closure of $\cC(\K)$
generally leads to a much larger number of minimizers of which the majority 
seems to correspond to 
quite unstable configurations in physical experiments. Our approach of 
penalizing the
bending energy by $\th$ times ropelength and approximating zero thickness by
letting $\th\to 0$ may be viewed as selecting those $\Eb$-minimizers
that
correspond to physically reasonable springy knots with sufficiently small
thickness.

Recall that our simple model neglects any effects of torsion. Twisting
the wire in the experiments before closing it at the hinge (without releasing
the twist) leads to completely different stable configurations; see
Figure \ref{fig:low_trefoil_tgA} on the right. 
So, in that case a more general Lagrangian
taking into account also these torsional effects would need to be
considered, and the question of classifying elastic knots with torsion
is wide open.

\begin{figure}
\begin{center}
   \includegraphics[width=0.3\textwidth]{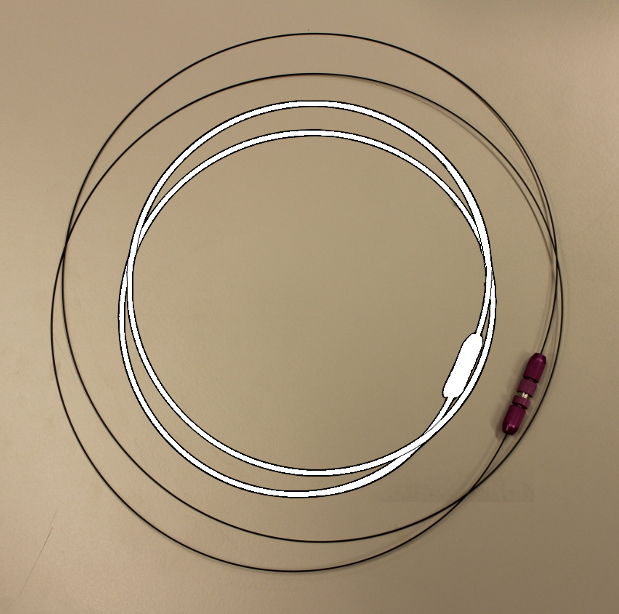}\quad
\includegraphics[width=.4\textwidth]{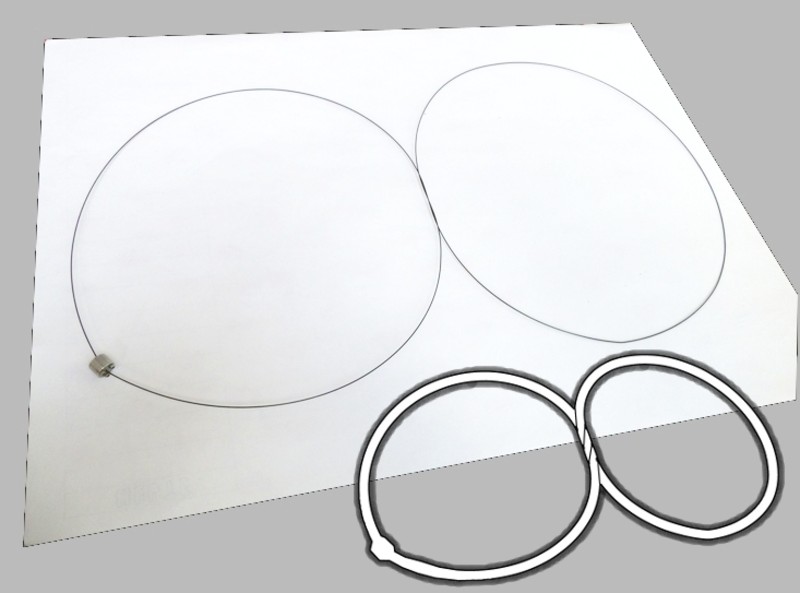}
\caption{Mechanical experiments. Left: The springy trefoil knot is close to the twice-covered
circle. Right: Adding a twist leads to a stable flat trefoil
configuration close to a planar figure-eight.
Wire models by courtesy of John~H.~Maddocks, Lausanne.}
\label{fig:low_trefoil_tgA}
\end{center}
\end{figure}

This is our strategy:

In Section~\ref{sect:mini} we establish first the existence of minimizers $\g_\th$
of the total energy $\Eth$ for each positive $\th$ (Theorem 
\ref{thm:existence-total}). Then we
pass to the limit $\th\to 0$ to obtain a limit configuration $\g_0$
in the $C^1$-closure of $\cC(\K)$ whose bending energy $\Eb(\g_0)$ serves
as a lower bound on $\Eb$ in $\cC(\K)$; see 
Theorem \ref{thm:existence-limit}.
By means of the classic uniqueness result of 
Langer and Singer \cite{LS:cs} on stable elasticae in $\R^3$ we identify
in Proposition \ref{prop:nonembedded} the round circle as the unique
elastic unknot. Elastic knots for knot classes with $(4\pi)^2$ as
sharp lower bound on the bending energy 
turn out to have constant curvature $4\pi$; see Proposition \ref{prop:constcurv}.
By a Schur-type result (Proposition  \ref{prop:minarc}) we can use
this curvature information to establish a  preliminary classification of such elastic knots as {\it tangential pairs of circles}, i.e., as pairs of round circles
each with radius $1/(4\pi)$ with (at least) one point of tangential intersection;
see Figure \ref{fig:intermediate} and Corollary \ref{cor:tg8}.
The key here to proving constant curvature
is an extension of the classic F\'ary--Milnor theorem on total curvature
to the $C^1$-closure
of $\cC(\K)$;
see  Theorem \ref{thm:fm} in the Appendix.
Our argument for that extension crucially relies on Denne's
result on the existence of alternating quadrisecants~\cite{denne}.
The fact that the elastic knot $\gamma_0$ for $\mathcal{K}$ is a tangential
pair of circles implies by means of Proposition \ref{prop:braids} that
$\mathcal{K}$ is actually the class $\cT(2,b)$ for odd $b$ with $|b|\ge 3$.

In order to extract the doubly-covered circle from the one-parameter
family of tangential pairs of circles as the only possible elastic knot
for $\cT(2,b)$
 we use in Section~\ref{sect:torus} 
explicit $(2,b)$-torus knots as suitable
  comparison curves. Estimating their
 bending energies and thickness values allows us to establish improved
 growth estimates for the total energy and ropelength of the $\Eth$-minimizers
 $\g_\th$; see Proposition \ref{prop:estimate-torus}.     
 In his seminal article~\cite{milnor} 
 Milnor
 derived
 the lower bound for the total curvature 
 by studying the \emph{crookedness} of a curve and relating it to
 the total curvature.
 For some regular curve $\g:\R/\Z\to\R^3$
 its crookedness  is the infimum over all $\nu\in\S^{2}$ of
  \begin{equation}\label{mu}
  \mu(\g,\nu) := \#\sett{t_{0}\in\R/\Z}{t_{0} \textnormal{ is a local maximizer of }
 t\mapsto\sp{\g(t),\nu}_{\R^{3}}}.
 \end{equation}
For any curve $\g$ close to a tangential pair of circles that is \emph{not} the
doubly covered circle we can show in Lemma \ref{lem:crook}  that
the set of directions $\nu\in\S^2$ for which $\mu(\gamma,\nu)\ge 3$
is bounded in measure from below by some
multiple of thickness $\triangle[\g]$. Assuming finally that $\g_\th $ converges
for $\th\to 0$ to such a limiting tangential pair of circles different from
the doubly covered circle we use this crookedness estimate
to obtain a contradiction against the total energy growth rate proved
in  Proposition~\ref{prop:estimate-torus}. Therefore, the only possible
limit configuration $\gamma_0$, i.e., elastic knot in the class of
$(2,b)$-torus knots, is the doubly covered circle.

As pointed out above,
the heart of our argument consists
of two bounds on the bending energy, the lower one, $(4\pi)^2$,
imposed by the F\'ary--Milnor inequality, the upper one
given by comparison curves.
The latter ones are constructed by considering
a suitable $(2,b)$-torus knot lying on a (standard) torus and
then shrinking the  width of the torus to zero, such
that the bending energy of the torus knot tends to the lower bound $(4\pi)^2$.
This indicates that a more general result should be valid
if these bounds can be extended to other knot classes.

Milnor~\cite{milnor} proved that the lower bound
on the total curvature is in fact $2\pi\bri\K$
where $\bri\K$ denotes the \emph{bridge index},
i.e., the minimum of crookedness\footnote{In fact, the bridge
index is defined as the minimum over the bridge number.
The latter coincides with crookedness for tame loops,
see Rolfsen~\cite[p.~115]{rolfsen}.} over the knot class~$\K$.
So we should ask which knot classes $\K$ can be represented by a
curve made of a number of strands, say $a$ strands, passing
inside a (full) torus, virtually in the direction of its central core.
The minimum value for $a$ with respect to the knot class~$\K$
is referred to as \emph{braid index}, $\bra\K$.
Thus we are led to believing that the following  assertion holds true
which has already been stated by
Gallotti and Pierre-Louis~\cite{gallotti-pierre-louis_2006}.

\begin{conjecture}[{Circular elastic knots}]
 The $a$-times covered circle is the (unique) elastic knot for the 
 (tame) knot class $\K$ if 
 $\bra\K=\bri\K=a$.
\end{conjecture}

The shape of elastic knots for more general knot classes
is one of the topics of
 ongoing research.
Here we only mention a conjecture that
has personally been communicated to the third author
by  Urs Lang in 1997.

\begin{conjecture}[Spherical elastic knots]\label{conj:lang}
 Any elastic prime knot is a spherical or planar curve.
\end{conjecture}

Our numerical experiments (see Figures~\ref{fig:trefoil_sim} and
\ref{fig:general_knot_classes})
as well as the simulations performed by Gallotti and Pierre-Louis~\cite[Figs.\@ 6 \& 7]{gallotti-pierre-louis_2006}
seem
to support this conjecture.

\begin{figure}
\begin{center}
\begin{tabular}{cc}
\includegraphics[width=0.4\textwidth]{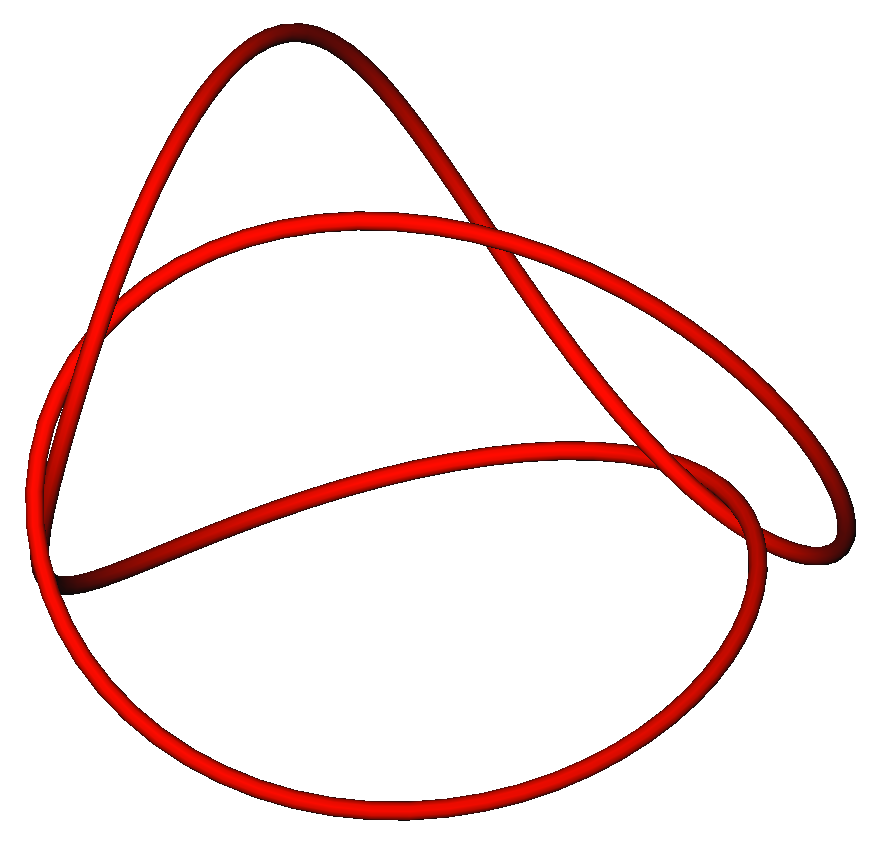} &
   \includegraphics[width=0.4\textwidth]{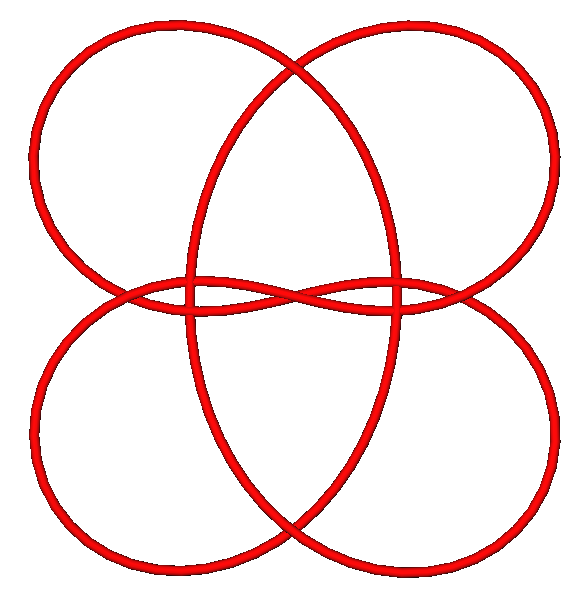} 
   \end{tabular}
   \end{center}
  \caption{Simulated annealing experiments for other knot classes.
  Left: The figure-eight knot ($4_{1}$) (resembling 
  the blue springy figure-eight in
   Figure \ref{fig:springy-knots}) seems to lie on a sphere. Right:
   A planar
    Chinese button knot ($9_{40}$) (in contrast to the spherical red springy wire
    in Figure \ref{fig:springy-knots} indicating that tying a more
    complex wire knot may impose some amount of physical
    torsion in addition to pure bending).}
 \label{fig:general_knot_classes}
\end{figure}

\paragraph{Acknowledgements.}

The second author was partially supported by DFG Trans\-regional Collaborative Research Centre SFB TR 71.
We gratefully acknowledge stimulating discussions with
Elizabeth Denne, Sebastian Scholtes, and John Sullivan.
We would like to thank 
Thomas~El~Khatib for bringing reference~\cite{gallotti-pierre-louis_2006}
to our attention.

% !TEX root = regelast.tex
\section{Existence of elastic knots}\label{sect:mini}

To ease notation we shall simply identify the intrinsic distance on $\R/\Z$ with
$|x-y|$, that is,
$$
|x-y|\equiv |x-y|_{\R/\Z}:=\min\{|x-y|,1-|x-y|\}.
$$
For any knot class $\K$ we define the class $\cC(\K)$ of {\it 
unit loops representing
$\K$} as
$$
\cC(\K):=\{\g\in H^2(\R/\Z,\R^3):\g(0)=0, |\g'|\equiv 1\,\ON\,\R/\Z,\textnormal{\,$\g$ is of knot type $\K$}\},
$$
where $H^2(\R/\Z,\R^3)$ denotes the class of $1$-periodic Sobolev functions whose
second weak derivatives are square-integrable. The bending energy \eqref{bending}
on the space of  curves
$\g\in H^{2}(\R/L\Z,\R^{3})$, $L>0$,  reads as
\begin{equation}\label{bending2}
\Eb(\gamma)=\int_{\g}\kappa^{2}\d s
=\int_{\R/L\Z}\kappa^{2}\abs{\dg}\d t
=\int_0^L\frac{\abs{\ddg\wedge\dg}^2}{\abs\dg^{5}}\d t,
\end{equation}
which reduces to the squared $L^2$-norm $\|\gamma''\|^2_{L^2}$ of the (weak) second derivative $\gamma''$ if
$\gamma$ is parametrized by arclength.
The ropelength functional defined as the quotient of length
and thickness simplifies on $\cC(\K)$ to
\begin{equation}\label{rope}
\Er(\gamma)=\frac{1}{\triangle[\g]},
\end{equation}
where thickness $\triangle[\g]$ can be expressed as in \eqref{thickness}.
For $\th >0$ we want to minimize  the \emph{total energy} $\Eth$
as given in \eqref{total-energy} on the class $\cC(\K)$ of unit loops
representing $\K$.
Note that, in contrast to the bending energy,
a  unit loop in $\cC(\K)$ has finite total energy $\Eth$ if and only if
it belongs to $C^{1,1}$, see~\cite[Lemma~2]{gmsm} and \cite[Theorem 1 (iii)]{SvdM}.
\begin{theorem}[Minimizing the total energy]\label{thm:existence-total}
For any fixed  tame knot class $\K$ and for
each $\th >0$ there exists a unit loop $\g_\th\in\cC(\K)$
such that 
\begin{equation}\label{mini-total}
\Eth(\g_\th)=\inf_{\cC(\K)}\Eth(\cdot).
\end{equation}
\end{theorem}

\begin{proof}
The total energy is obviously nonnegative, and $\cC(\K)$ is not empty
since one may scale any smooth representative of $\K$ 
(which exists due to tameness) down to length
one and  reparametrize to arclength, so that the infimum in \eqref{mini-total}
is finite. Taking a minimal sequence $(\g_k)_k\subset\cC(\K)$ with
$\Eth(\gamma_k)\to\inf_{\cC(\K)}\Eth$ as $k\to\infty$ we get the uniform
bound
$$
\|\g_k''\|_{L^2}^{2}=\Eb(\g_k)\le\Eth(\g_k)\le 1+\inf_{\cC(\K)}\Eth<\infty\quad\Foa
k\gg 1,
$$
so that with $\g_k(0)=0$ and $|\g_k'|\equiv 1$ for all $k\in\N$ we have
a uniform bound on the full $H^2$-norm of the $\g_k$ independent of $k$ for
$k$ sufficiently large.
Since $H^2(\R/\Z,\R^3)$ as a Hilbert space is reflexive and $H^2(\R/\Z,\R^3)$ is compactly embedded in $C^1(\R/\Z,\R^3)$
this implies the existence
of some $\g_\th\in H^2(\R/\Z,\R^3)$ and
a subsequence $(\g_{k_i})_i\subset (\g_k)_k$ converging weakly in $H^2$ and
strongly in $C^1$ to $\g_\th$ as $i\to\infty.$
Thus we 
obtain $\g_\th(0)=0$ and $|\g_\th'|\equiv 1$ on $\R/\Z.$ Since thickness
is upper semicontinuous \cite[Lemma 4]{gmsm} and the bending energy lower semicontinuous with respect
to this type of  convergence  we arrive at
\begin{equation}\label{direct-method}
\Eth(\g_\th)\le\liminf_{i\to\infty}\Eth(\gamma_{k_i})=\inf_{\cC(\K)}\Eth(\cdot)<\infty.
\end{equation}
In particular by definition of $\Eth$, one has $\Er(\g_\th)<\infty$ or 
$\triangle[\g_\th]>0$, which implies by \cite[Lemma 1]{gmsm} that $\g_\th$
is embedded. 
As all closed curves in a $C^1$-neighbourhood of a given embedded curve
are isotopic, as shown by
Diao, Ernst, and Janse van Rensburg~\cite[Lemma~3.2]{dej},
see also~\cite{blatt:isot,reiter:isot},
we find that $\g_\th$ is of knot
type $\K$;  hence $\g_\th\in\cC(\K).$
This gives $\inf_{\cC(\K)}\Eth(\cdot)\le\Eth(\g_\th)$,
which in combination with \eqref{direct-method} concludes the proof.
\end{proof}

Since finite ropelength, i.e., positive thickness, implies
$C^{1,1}$-regularity we know that $\g_\th\in C^{1,1}(\R/\Z,\R^3)$.
However, we are not going to exploit this improved regularity, 
since we investigate the limit $\th\to 0$ and the 
corresponding limit configurations, the elastic knots $\g_0$ for
the given knot class $\K$.

\begin{theorem}[Existence of elastic knots]\label{thm:existence-limit}
Let $\K$ be any fixed  
tame knot class, $\th_i\to 0$ and $(\g_{\th_i})_{i}\subset\cC(\K)$, such that 
$E_{\th_i}(\g_{\th_i})=\inf_{\cC(\K)}E_{\th_i}(\cdot)$ for each $i\in\N.$ Then 
there exists $\g_0\in H^2(\R/\Z,\R^3)$ and a subsequence 
$(\g_{\th_{i_k}})_k\subset (\g_{\th_i})_{i}$ such that the $\g_{\th_{i_k}}$ 
converge weakly in $H^2$
and strongly in $C^1$ to $\g_0$ as $k\to\infty$. Moreover,
\begin{equation}\label{bendminimizing}
\Eb(\g_0)\le\Eb(\beta)\quad\Foa \beta\in\cC(\K).
\end{equation}
The estimate~\eqref{bendminimizing} is strict unless $\K$
is the unknot class.
\end{theorem}

\begin{definition}[Elastic knots]\label{def:elasticknot}
Any  curve $\g_0$ as in Theorem \ref{thm:existence-limit} is 
called an \emph{elastic knot for $\K$}.
\end{definition}

\begin{proof}[Theorem~\ref{thm:existence-limit}]
For any $\beta\in\cC(\K)$ and any $\th >0$
we can estimate
\begin{equation}\label{bending-estimate}
\Eb(\g_\th)\le\Eth(\g_\th)\le\Eth(\beta),
\end{equation}
where $\g_\th\in\cC(\K)$ is a global  minimizer of $\Eth$ within $\cC(\K)$ whose existence
is guaranteed by Theorem \ref{thm:existence-total}.
Now we restrict to $\beta\in\cC(\K)\cap C^{1,1}(\R/\Z,\R^3)$,
which implies by means of \cite[Theorem 1 (iii)]{SvdM}
that the right-hand side of~\eqref{bending-estimate}
is finite.
Now the right-hand side tends to $\Eb(\beta)<\infty$ as $\th\to 0$ and
we find a constant $C$ independent of $\th$ such that 
\begin{equation}\label{unifH2}
\|\g_\th\|_{H^2}\le  C\quad\Foa \th\in (0,1).
\end{equation}
In particular, this uniform estimate holds for the $\g_{\th_i}
\in\cC(\K)$ so that there
is $\g_0\in H^2(\R/\Z,\R^3)$ and a subsequence $(\g_{\th_{i_k}})_k\subset
(\g_{\th_i})_i$ with $\g_{\th_{i_k}}\rightharpoonup\g_0$ in $H^2$ and
$\g_{\th_{i_k}}\to\g_0$ in $C^1$ as $k\to\infty.$ The bending energy $\Eb$
is lower semicontinuous with respect to weak convergence in $H^2$ which implies
$$
\Eb(\g_0)\le\liminf_{k\to\infty}\Eb(\g_{\th_{i_k}})\overset{\eqref{bending-estimate}}{\le}\liminf_{k\to\infty}E_{\th_{i_k}}(\beta)=\Eb(\beta).
$$
Thus we have established~\eqref{bendminimizing}
for any $\beta\in\cC(\K)\cap C^{1,1}(\R/\Z,\R^3)$.
In order to extend it to the full domain,
we approximate an arbitrary $\beta\in\cC(\K)$ by a sequence of functions
$\seqn\beta\subset C^{\infty}\rzd$ with respect to the $H^{2}$-norm.
As $H^{2}$ embeds into $C^{1,1/2}$, the tangent vectors $\beta_{k}'$ uniformly
converge to $\beta'$, so $\abs{\beta_{k}'}\ge c>0$ for all $k\gg1$.

Furthermore, the $\beta_{k}$ are injective 
since for $\beta$ there are
positive constants $c$ and $\eps$ depending only on $\beta$ such
that 
$$
|\beta(s)-\beta(t)|\ge c|s-t|_{\R/\Z}\quad\Foa |s-t|_{\R/\Z}<\eps,
$$
because $|\beta'|\equiv 1$, and in consequence, there is another constant
$\delta=\delta(\beta)\in (0,c\eps]$ such that
$$
|\beta(s)-\beta(t)|\ge \delta\quad\Foa |s-t|_{\R/\Z}\ge\eps,
$$
for $\beta$ is injective. Consequently, 
for given distinct parameters
$s,t\in [0,1)$ we can estimate 
$$
|\beta_k(s)-\beta_k(t)|  \ge
|\beta(s)-\beta(t)|-2\|\beta_k'-\beta'\|_{L^\infty}|s-t|_{\R/\Z},
$$
which is positive for $k\gg 1$ independent of the intrinsic distance $|s-t|_{\R/\Z}$.
In addition, the $\beta_k$ represent the same
knot class $\K$ as $\beta$ does for $k\gg 1$,
since isotopy is stable under $C^1$-convergence
\cite{dej,blatt:isot,reiter:isot}.
According to~\cite[Thm.~A.1]{reiter:rkepdc}
the sequence $\seqn{\tilde\beta}$ of smooth curves,
where $\tilde\beta_{k}$ is obtained from $\beta_{k}$ (after omitting finitely
many $\beta_k$) by rescaling to
unit length and then reparametrizing to arc-length,
converges to $\tilde\beta=\beta$ with respect to the $H^{2}$-norm,
and, of course, $\tilde\beta_{k}\in\cC(\K)$ for all $k$.
We conclude
\[ \Eb(\g_0)\le\Eb(\tilde\beta_{k}) = \lnorm{\tilde\beta_{k}''}^{2}
\to\lnorm{\beta''}^{2} = \Eb(\beta). \]
Assume now $\Eb(\g_{0})=\Eb(\beta)$ for some $\beta\in\cC(\K)$ where
$\K$ is non-trivial. In this case, $\beta$ would be a local minimizer,
and therefore a stable closed elastica
as there are no restrictions for variations.  According to the
result of J.~Langer and D.~A.~Singer~\cite{LS:cs}\footnote{Recall
that  \emph{elasticae} are the critical curves for the bending
energy. Notice that Langer and
Singer work on the  tangent indicatrices  of arclength parametrized curves with
a fixed point, i.e., on {\it balanced curves} on $\S^2$ through a
fixed point, so that their
variational arguments can be applied to the tangent vectors of
curves in $\cC(\K)$; see \cite[p. 78]{LS:cs}.},
$\beta$ turns out to be the round circle, hence $\K$ is the unknot, contradiction.
\end{proof}

% !TEX root = regelast.tex
\section{The elastic unknot and  tangential pairs of circles}
\label{sec:shape}
The springy knotted wires strongly suggest that we should expect
that the elastic knots generally exhibit self-intersections, and this
is indeed the case unless the knot class is trivial. In fact, the elastic
unknot is the round circle of length one.
\begin{proposition}[Non-trivial elastic knots are not embedded]\label{prop:nonembedded}
The round circle of length one is the unique elastic unknot.
If there exists an embedded elastic knot $\g_0$ for a given  
knot class $\K$ then
$\K$ is the unknot (so that $\g_0$ is the round circle of length one).
In particular, if $\K$ is a non-trivial knot class, every elastic knot
for $\K$ must have double points.
\end{proposition}

\begin{proof}
The round circle of length one uniquely minimizes $\Eb$ and $\Er$ simultaneously
in the class of all arclength parametrized curves in $H^2(\R/\Z,\R^3).$
For the bending energy $\Eb$ this is true since the round once-covered circle
is the only stable closed
elastica in $\R^3$ according to
the work of J.~Langer and D.~A.~Singer \cite{LS:cs},
and for ropelength this follows, e.g., from 
the more general uniqueness result in \cite[Lemma 7]{SvdM07} for the functionals
$$
\cU_p(\gamma):=\br{\int_\gamma\sup_{v,w\in\R/\Z\atop u\not=v\not=w\not=u}R^{-p}(\gamma(u),\gamma(v),\gamma(w))\d u}^{1/p}, \quad p\ge 1.
$$
For closed rectifiable
curves $\gamma$ different from the round circle with $\Er(\gamma)<\infty$  
one has indeed
$$
2\pi=\Er(\textnormal{circle})=\cU_p(\textnormal{circle})\overset{\text{\cite[Lem.~7]{SvdM07}}}{<}\cU_p(\gamma)\le\Er(\gamma).
$$

Hence the round circle uniquely
minimizes also the total energy $\Eth$ 
for each $\th>0$ in $\cC(\K)$ when $\K$ is the unknot. Thus any elastic
unknot as the $C^1$-limit of $E_{\th_i}$-minimizers as $\th_i\to 0$
is also the round circle
of length one.

If $\g_0$ is embedded then according to the stability of isotopy classes
under $C^1$-perturbations (see \cite{dej,blatt:isot,reiter:isot}) we find
$\g_0\in\cC(\K)$ since $\g_0$ was obtained as the weak $H^2$-limit of
$E_{\th_i}$-minimizers
$\g_{\th_i}\in\cC(\K).$ 
Thus, by~\eqref{bendminimizing}, the curve $\g_0$ is a local  minimizer of $\Eb$ in $\cC(\K)$,
and therefore it is a stable elastica. Thus, again by the stability
result of Langer and Singer, $\g_0$ is the round circle of length one.
Consequently, $\K$
is the unknot, which proves the proposition.
\end{proof}

In the appendix we extend the F\'ary--Milnor theorem to the
$C^1$-closure of $\cC(\K)$, where $\K$ is a non-trivial knot class; see
Theorem \ref{thm:fm}. This allows us in a first step
to show that elastic knots for
$\K$ have constant curvature if one can get arbitrarily close with 
the bending energy in $\cC(\K)$  to the
natural lower bound $(4\pi)^2$ induced by F\'ary and Milnor.
\begin{proposition}[Elastic knots of constant curvature]\label{prop:constcurv}
For any knot class $\K$ with
\begin{equation}\label{infimaequal}
\inf_{\cC(\K)}\Eb=(4\pi)^2
\end{equation}
one has $\kappa_{\g_0} = 4\pi$  a.e.\@ on $\R/\Z$ for each elastic knot
$\g_0$ for $\K$. In particular, $\g_0\in C^{1,1}(\R/\Z,\R^3).$
\end{proposition}

\begin{proof}
According to our generalized F\'ary--Milnor Theorem, Theorem \ref{thm:fm}
in the Appendix applied to $\g_0\in\overline{\cC(\K)}^{C^1}$, 
we estimate by means of H\"older's  inequality
\[
(4\pi)^2\overset{\eqref{eq:fm}}{\le} \left(\int_{\g_0}\kappa_{\g_0}\right)^2
\le\Eb(\g_0)\stackrel{\eqref{bendminimizing}}
\le \inf_{\cC(\K)}\Eb\stackrel{\eqref{infimaequal}}=(4\pi)^2,
\]
hence equality  everywhere. In particular, equality in H\"older's inequality
implies a constant integrand, which we calculate to be $\kappa_{\g_0}=4\pi$
a.e.\@ on $\R/\Z$.
\end{proof}

Of course, there are many closed curves of constant curvature,
see e.g.\@ Fenchel~\cite{fenchel1930},
even in every knot class,
see McAtee \cite{mcatee_2007}.
Closed spaces curves of constant curvature may, e.g.,
be constructed by joining suitable arcs of helices,
see Koch and Engelhardt~\cite{koch-engelhardt} for an explicit
construction and examples.
But in order to identify the possible shape of elastic knots  for
knot classes that satisfy assumption~\eqref{infimaequal}
recall that any minimizer in a non-trivial knot class has at least
one double  point; see Proposition
\ref{prop:nonembedded}. In addition, the length is fixed to one, so
that we can reduce the possible shapes of elastic knots considerably
by means of 
a Schur-type argument that connects length and constant curvature; see the following Proposition~\ref{prop:minarc}. This will in particular lead to the proof
of the classification result, Theorem~\ref{thm:elastic-shapes}.

\begin{proposition}[Shortest arc of constant curvature]\label{prop:minarc}
 Let $\g\in H^{2}([0,L],\R^{3})$, $L>0$, be parametrized by
  arc-length with constant curvature $\kappa = \abs{\ddg} =4\pi$ a.e.\@
   and coinciding endpoints $\g(0) = \g(L)$.
    Then $L\ge\frac12$ with equality if and only if $\g$ is a circle
     with radius $\frac1{4\pi}$.
     \end{proposition}

Before proving this rigidity result let us note an immediate consequence.

\begin{corollary}[Tangential pairs of circles]\label{cor:tg8}
Every
closed arclength parametrized curve $\gamma\in H^2(\R/\Z,\R^3)$ with at least 
one double point and with
constant curvature $\kappa_\g=|\g''| = 4\pi$ a.e.\@ on $\R/\Z$,
is a \emph{tangential
pair of circles}. That is, $\gamma $ belongs, up to isometry and parametrization, to
 the one-parameter family of tangentially intersecting circles
 \begin{equation}\label{eq:tg8}
  \tge_\varphi:t\mapsto\tfrac1{4\pi}
  \begin{cases}
   \uni1(1-\cos(4\pi t)) - \uni2\sin(4\pi t), & t\in [0,\tfrac12], \\
   (\uni1\cos\varphi + \uni3\sin\varphi)(1-\cos(4\pi t)) - \uni2\sin(4\pi t), & t\in [\tfrac12,1],
  \end{cases}
 \end{equation}
 where $\varphi\in [0,\pi]$.
 Here $\uni k$ denotes the $k$-th unit vector in $\R^{3}$, $k=1,2,3$.
\end{corollary}
Note that $\tge_0$ is a doubly covered circle and $\tge_\pi$ a tangentially 
intersecting planar figure eight,  both located in the plane spanned by
$\uni1$ and $\uni2$. For intermediate values $\varphi\in 
(0,\pi)$ one obtains a configuration as shown in Figure~\ref{fig:intermediate}.

\begin{figure}
 \centering
  \includegraphics{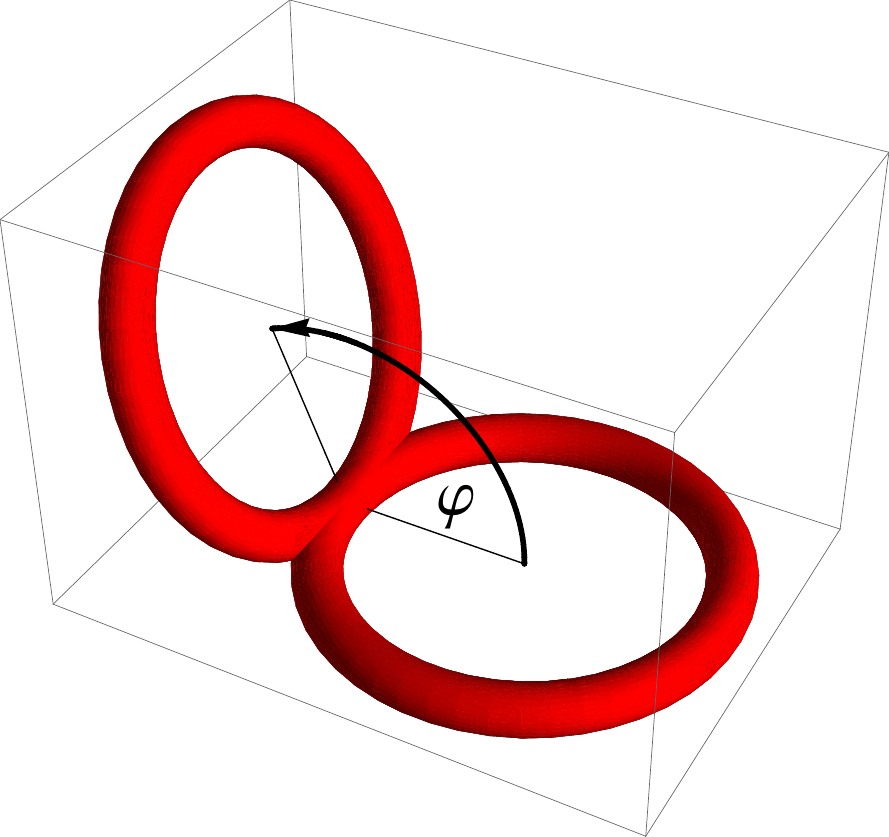}
    \caption{Plot of $\tge_\varphi$ for $\varphi={\nfrac{5\pi}8}$}\label{fig:intermediate}
     \end{figure}

\begin{proof}[Corollary~\ref{cor:tg8}]
Applying Proposition~\ref{prop:minarc}
     we find that the length of any arc starting from a double point 
     amounts to at least $\tfrac12$.
      As $\length(\g)=1$ we have precisely two arcs of 
      length~$\tfrac12$ between each double point
       which again by Proposition~\ref{prop:minarc} implies that both these
       connecting arcs are 
       circles of radius $\tfrac1{4\pi}$.
        They have to meet tangentially due to the embedding 
	$H^{2}\hookrightarrow
	 C^{1,1/2}$.
	  Thus $\g=\tge_{\varphi}$ for some
	   $\varphi\in [0,\pi]$.
\end{proof}

\begin{proof}[Proposition~\ref{prop:minarc}]
Obviously, the statement is equivalent to minimizing $L>0$ over 
$f = \dg \in H^{1}([0,L],\mathbb S^{2})$ with $\abs{f'} = 4\pi$ a.e.\@
and $\int_{0}^{L}f=0$. 
 As $f\not\equiv0$ there is some $T\in(0,L)$ maximizing $t\mapsto \abs{\int_{0}^{t}f(\theta)\d\theta}^{2}$ which leads to
 \[ 0 =  2\sp{\int_{0}^{T}f(\theta)\d\theta,f(T)}
      = -2\sp{\int_{T}^{L}f(\theta)\d\theta,f(T)}. \]
 We consider
 \[ g : t\mapsto\sign(t-T)\sp{\int_{T}^{t}f(\theta)\d\theta,f(T)}
              = \sign(t-T)\int_{T}^{t}\sp{f(\theta),f(T)}\d\theta \]
 for $t\in[0,L]$. By assumption we have $g(0)=g(L)=0$.
 As $f$ takes its values in the sphere $\mathbb S^{2}$
 and moves with constant speed, we can estimate
 \begin{equation}\label{eq:angle}
  \angle\br{f(\theta),f(T)} \le \length(f|_{[\theta,T]})=
  \int_\theta^T|f'(\tau)|\,d\tau = 4\pi\abs{\theta-T}.
 \end{equation}
 As long as $\theta \in [T-\frac14,T+\frac14]$ we obtain by monotonicity of
 the cosine function
 \begin{equation}\label{eq:cosle}
  \cos(4\pi\abs{\theta-T}) \le \cos\angle\br{f(\theta),f(T)}
    = \sp{f(\theta),f(T)}.
 \end{equation}
So, for $t\in [T-\frac14,T+\frac14]$, as cosine is even,
 \begin{equation}\label{eq:gge}
  \begin{split}
   g(t) &\ge \sign(t-T)\int_{T}^{t}\cos(4\pi\br{\theta-T})\d\theta
         =   \left.\tfrac1{4\pi}{\sign(t-T)}\sin(4\pi\br{\theta-T})\right|_{T}^{t} \\
        &=   \tfrac{1}{4\pi}\sin(4\pi\abs{t-T}).
  \end{split}
 \end{equation}
 The right-hand side is positive for $t\in(T-\frac14,T)\cup(T,T+\tfrac14)$
 and vanishes if $t\in\set{0,T\pm\frac14}$.
 Now $g(0)=g(L)=0$ together with $0<T<L$ yields 
 \[ 0,L \notin (T-\tfrac14,T+\tfrac14), \]
 and one has  $0\le T-\frac14<T+\frac14\le L$,
 and therefore $T\ge \tfrac14$ as well as
 \[ L \ge T+\tfrac14 \ge \tfrac12. \]
 The  case $L=\frac12$ enforces $T=\frac14$, and inserting
$t=L=\frac12$ in~\eqref{eq:gge} we arrive at
 \[ 0 = \int_{1/4}^{1/2}\sp{f(\theta),f(\tfrac14)}\d\theta 
 \ge \int_{1/4}^{1/2}\cos\br{4\pi\br{\theta-\tfrac14}}\d\theta = 0, \]
 thus
 \[ \int_{1/4}^{1/2}\sq{\sp{f(\theta),f(\tfrac14)}-
 \cos\br{4\pi\br{\theta-\tfrac14}}}\d\theta = 0. \]
 The integrand is non-negative by~\eqref{eq:cosle}, so it vanishes a.e. on
 $[\frac14,\frac12]$.
 By continuity we obtain equality in~\eqref{eq:cosle} for any $\theta\in[\frac14,
 \frac12]$.
 Similarly, inserting $t=0$ in \eqref{eq:gge} we obtain
 equality in~\eqref{eq:cosle}
 also for $\theta\in[0,\tfrac14]$, whence for all $\theta\in [0,\frac12]$,
 and subsequently equality in~\eqref{eq:angle}
 for all $\theta\in [0,\frac12].$
 We especially face
 $\angle\br{f(0),f(\frac14)} = \angle\br{f(\frac14),f(\frac12)} = \pi$.
 So $f$ joins two antipodal points by an arc of length $\pi$.
 Therefore both $f|_{[0,1/4]}$ and $f|_{[1/4,1/2]}$ are
great semicircles connecting
$f(0)$ and $f(\frac14)$ resp.\@ $f(\frac14)$ and $f(\frac12)$, and
if these two great semicircles did not belong to the same great circle then
 $\int_0^Lf\not=0$, contradiction.
  Note that $t\mapsto \int_{0}^{t}f(\theta)\d\theta$ is also a circle
 as $f$ parameterizes a great circle on $\S^2$ with constant speed $4\pi$ a.e.
\end{proof}

% !TEX root = regelast.tex

\section{Tangential pairs of circles identify $(2,b)$-torus knots}\label{sec:class}

We have seen in the previous section that elastic knots are restricted to
the one-parameter family \eqref{eq:tg8} of tangential pairs of circles, if $(4\pi)^2$
is the sharp lower bound for the bending energy on $\cC(\K)$; see \eqref{infimaequal}.
 And since
elastic knots by definition lie in the $C^1$-closure of $\cC(\K)$ one can
ask the question, whether the existence of tangential pairs of circles in that $C^1$-closure
determines the knot class $\K$ in any way.  This is indeed the case, $\K$ turns
out to be either unknotted or a $(2,b)$-torus knot for odd $b$ with $|b|\ge 3$ as will be
shown in Proposition \ref{prop:braids}. The following preliminary construction of
sufficiently small cylinders containing the self-intersection of a given tangential
pair of circles in Lemmata \ref{lem:lgr} and \ref{lem:zylinder} as well as the explicit isotopy constructed in Lemma \ref{lem:isohandle}
do not only prepare the proof of Proposition \ref{prop:braids} but are also the foundation
for the argument to single out the twice-covered circle as the only possible shape
of an elastic knot in $\cT(2,b)$ in Section \ref{sec:5}.

We intend to characterize the situation of
a curve $\g$ being close to $\tge_{\varphi}$ for $\varphi\in(0,\pi]$
with respect to the $C^{1}$-topology, i.e.,
\[ \norm{\g-\tge_{\varphi}}_{C^{1}}\le\delta \]
where $\delta>0$ will depend on some $\eps>0$.

To this end, we consider a cylinder $\mathcal Z$
around the intersection point of
the two circles of $\tge_{\varphi}$ (which is the origin), see Figure~\ref{fig:zylinder}.
Its axis will be parallel to the tangent line (containing $\mathbf e_{2}$)
and centered at the origin. More precisely,
\begin{equation}\label{grundzylinder}
 \mathcal Z := \sett{(x_{1},x_{2},x_{3})^{\top}\in\R^{3}}{x_{1}^{2}+x_{3}^{2}\le\zeta^{2},\abs{x_{2}}\le\eta},
\end{equation}
where
\[ \eta:=\sqrt{\nfrac\zeta{({8\pi)}}}
 \qquad\text{for some }\zeta\in\left(0,\tfrac1{32\pi}\right]. \]
This will produce a ``braid representation'' shaped form of the knot $\g$.
(See, e.g., Burde and Zieschang~\cite[Chap.~2~D, 10]{BZ}
for information on braids and their closures.)
Outside $\mathcal Z$ the curve $\gamma$ consists of two ``unlinked'' handles
which both connect the opposite caps of $\mathcal Z$.
Inside $\mathcal Z$ it consists of two sub-arcs $\tg_{1}$ and $\tg_{2}$ of $\g$
which can be reparametrized as graphs over the $\mathbf e_{2}$-axis
(thus also entering and leaving at the caps).

As follows, the knot type of $\g$ can be analyzed by studying the
over- and under-crossings inside $\mathcal Z$
viewed under a suitable projection.
Due to the graph representation, each fibre
\[ \mathcal Z_{\xi}:=\mathcal Z\cap\br{\xi+\mathbf e_{2}^{\perp}} = \sett{(x_{1},\xi,x_{3})^{\top}\in\R^{3}}{x_{1}^{2}+x_{3}^{2}\le\zeta^{2}}, \qquad \xi\in\sq{-\tfrac1{16\pi},\tfrac1{16\pi}}, \]
is transversally met by
precisely one point $\tg_{1}(\xi)$, $\tg_{2}(\xi)$ of each arc.
By embeddedness of $\g$, this defines a vector
\begin{equation}\label{eq:a-xi}
 a_{\xi} := \tg_{1}(\xi)-\tg_{2}(\xi)\in\mathbf e_{2}^{\perp} 
\end{equation}
of positive length.
Set 
\begin{equation}\label{eq:nu}
\nu:=\mathbf e_{1}\cos\nfrac\varphi2+\mathbf e_{3}\sin\nfrac\varphi2.
\end{equation}
Since both, $\nu$ and $a_\xi$ for every $\xi\in [-\eta,\eta]$, are
contained in the plane $\mathbf e_2^\perp$, we may write
\begin{equation}\label{betadef}
a_\xi=
\left(\begin{array}{ccc}
\cos\beta(\xi) & 0 & -\sin\beta(\xi)\\
0 & 1 & 0 \\
\sin\beta(\xi) & 0 & \cos\beta(\xi)\end{array}\right)\nu
=\mathbf e_{1}\cos\br{\nfrac\varphi2+\beta(\xi)}+\mathbf e_{3}\sin\br{\nfrac\varphi2+\beta(\xi)},
\end{equation}
which defines
a continuous function $\beta\in C^0([-\eta,\eta])$ measuring
the angle between $\nu$ and $a_\xi$ for $\xi\in [-\eta,\eta].$
This function is uniquely defined if we additionally set
$\beta(-\eta)\in [0,2\pi)$, and it traces possible multiple rotations of the 
vector $a_\xi$ about the $\mathbf e_2$-axis as $\xi$ traverses the
parameter range from $-\eta$ to $+\eta$. So $\beta(\xi)$ should
{\it not} be considered an element of $\R/2\pi\Z$.

Choosing $\delta$ sufficiently small,
it will turn out that the difference of the $\beta$-values
at the caps of the cylinder $\mathcal{Z}$, i.e., 
\[ \Delta_{\beta}\equiv \Delta_{\beta,\eta} := \beta(\eta)-\beta(-\eta) \]
captures the essential topological information of the knot type of $\g$.
(Since $\eta$ will be fixed later on we may as well suppress the dependence
of $\Delta_\beta$ on $\eta.$)

In order to make these ideas more precise, we first introduce
a local graph representation of $\g\cap\mathcal Z$ in Lemma~\ref{lem:lgr} below.
Then we employ an isotopy that maps $\g$ outside $\mathcal Z$
to $\tge_{\varphi}$, see Lemma~\ref{lem:isohandle}.
Proposition~\ref{prop:braids} will characterize the knot type of $\g$ assuming
that the opening angle $\varphi$ of the tangential pair of circles $\tge_\varphi$
in \eqref{eq:tg8} is different from zero, and
in Corollary~\ref{cor:braids} we state the
corresponding result for $\varphi=0$.

The reparametrization can explicitly be written for
$\tge_{\varphi}$, $\varphi\in[0,\pi]$. In fact,
letting
$\Tge_{\varphi,1}(\xi) :=  \tge_{\varphi}\br{\frac{\arcsin(4\pi\xi)}{4\pi}}$
and
$\Tge_{\varphi,2}(\xi) :=  \tge_{\varphi}\br{\frac{\arcsin(4\pi\xi)}{4\pi}+\frac12}$
we arrive at
$\Tge_{\varphi,j}(\xi)\in-\xi\mathbf e_{2}+\mathbf e_{2}^{\perp}$
for all $\xi\in[-\frac{1}{4\pi},\frac{1}{4\pi}]$ and $j=1,2$.

\begin{lemma}[Local graph representation]\label{lem:lgr}
 Let $\varphi\in[0,\pi]$, $\eps>0$, and $\g\in C^{1}\rzd$ with
 \begin{equation}\label{eq:lgr}
  \norm{\g-\tge_{\varphi}}_{C^{1}}
  \le\delta=\delta_{\eps}:=\min\br{\tfrac\eps{42},\tfrac1{16\pi}}.
 \end{equation}
 Then there are two diffeomorphisms $\phi_{1},\phi_{2}\in C^{1}([-\frac{1}{16\pi},\frac{1}{16\pi}])$
 such that, for $j=1,2$,
 \[ \tg_{j}(\xi):=\g(\phi_{j}(\xi))\in-\xi\mathbf e_{2}+\mathbf e_{2}^{\perp}\qquad\text{for all } \xi\in[-\tfrac{1}{16\pi},\tfrac{1}{16\pi}] \]
 and
 \begin{equation*}
  \norm{\tg_{j}-\Tge_{\varphi,j}}_{C^{1}([-\frac{1}{16\pi},\frac{1}{16\pi}])} \le \eps.
 \end{equation*}
\end{lemma}

\begin{proof}
 From~\eqref{eq:tg8} we infer for $t\in[-\tfrac1{24},\tfrac1{24}]
 \cup[\tfrac{11}{24},\tfrac{13}{24}]$
 \[ \sp{\tge_{\varphi}'(t),-\mathbf e_{2}} = \cos(4\pi t) \ge \
 \cos\tfrac\pi6 = \tfrac12\sqrt3, \]
 so, for $\norm{\g'-\tge_{\varphi}'}_{C^{0}}\le\frac1{2}\br{\sqrt3-1}$,
 \[ \sp{\dg(t),-\mathbf e_{2}} \ge
 \sp{\tge_{\varphi}'(t),-\mathbf e_{2}} - \norm{\g'-\tge_{\varphi}'}_{C^{0}}
 \ge \tfrac12. \]

 Thus the first derivative of the $C^{1}$-mapping 
 $t\mapsto\sp{\g(t),-\mathbf e_{2}}$
 is strictly positive on $[-\frac1{24},\frac1{24}]$ and $[\frac{11}{24},\frac{13}{24}]$.
 Consequently it is invertible, and its inverse is also $C^{1}$.
 Claiming $\norm{\g-\tge_{\varphi}}_{C^{0}} \le \delta \le \frac1{16\pi}$,
 its image contains the interval $[-\frac1{16\pi},\frac1{16\pi}]$
 since $\sp{\tge_{\varphi}\br{[-\frac1{24},\frac1{24}]},-\mathbf e_{2}}=[-\frac{1}{8\pi},\frac{1}{8\pi}]$.
 We denote the
 respective inverse functions, restricted to $[-\frac1{16\pi},\frac1{16\pi}]$,
 by $\phi_{1},\phi_{2}$.
 
 In order to estimate the distance between $x:=\phi_{1}(\xi)$ and
 $y:=\frac{\arcsin(4\pi\xi)}{4\pi}$ for $\xi\in[-\frac1{16\pi},\frac1{16\pi}]$
 we first remark that, by construction, $x\in[-\frac1{24},\frac1{24}]$.
 We obtain
 \begin{align*}
  \abs{\g(x)-\tge_{\varphi}(x)}
  &\ge\dist\br{-\xi\mathbf e_{2}+\mathbf e_{2}^{\perp},\tge_{\varphi}(x)} \\
  &=\abs{\xi-\sp{\tge_{\varphi}(x),-\mathbf e_{2}}} \\
  &=\tfrac1{4\pi}\abs{4\pi\xi-\sin(4\pi x)} \\
  &=\tfrac1{4\pi}\abs{\sin(4\pi y)-\sin(4\pi x)} \\
  &=\abs{\int_{x}^{y}\cos(4\pi s)\d s} \\
  &\ge \min\br{\cos(4\pi y),\cos(4\pi x)}\abs{x-y} \\
  &\ge \min\br{\tfrac14\sqrt{15},\cos\tfrac\pi6}\abs{x-y} \\
  &=\tfrac12\sqrt3\abs{x-y},
 \end{align*}
where we used the fact that $x$ and $y$ are so small that we are in the strictly
concave region of the cosine near zero.
 Letting $\norm{\g-\tge_{\varphi}}_{C^{1}}\le\delta\le\tfrac1{16\pi}$,
 we arrive at
 \[ \abs{\phi_{1}(\xi)-\tfrac{\arcsin(4\pi\xi)}{4\pi}} = \abs{x-y}\le\tfrac23\sqrt3\delta
 \qquad\text{for }\xi\in[-\tfrac1{16\pi},\tfrac1{16\pi}]. \]
 For the derivatives, we infer
 \begin{align*}
  &\sp{\tge_{\varphi}\br{\tfrac{\arcsin(4\pi\xi)}{4\pi}},-\mathbf e_{2}} = \xi,
 &&\sp{\tge_{\varphi}'\br{\tfrac{\arcsin(4\pi\xi)}{4\pi}},-\mathbf e_{2}} = \sqrt{1-(4\pi\xi)^{2}},
 \end{align*}
 thus
 \begin{align*}
 &\abs{\phi_{1}'(\xi)-\tfrac{\d}{\d\xi}\tfrac{\arcsin(4\pi\xi)}{4\pi}}
 = \abs{\frac1{{\sp{\dg(\phi_{1}(\xi)),-\mathbf e_{2}}}}-\frac1{\sqrt{1-(4\pi\xi)^{2}}}} \\
 &=\abs{\frac{\sp{\tge_{\varphi}'(\tfrac{\arcsin(4\pi\xi)}{4\pi}),-\mathbf e_{2}}-\sp{\dg(\phi_{1}(\xi)),-\mathbf e_{2}}}{\sqrt{1-(4\pi\xi)^{2}}\sp{\dg(\phi_{1}(\xi)),-\mathbf e_{2}}}} \\
 &\le\abs{\frac{\sp{\tge_{\varphi}'(\tfrac{\arcsin(4\pi\xi)}{4\pi})-\dg(\phi_{1}(\xi)),-\mathbf e_{2}}}{\frac14\sqrt{15}\br{\sqrt{1-(4\pi\phi_{1}(\xi))^{2}}-\sp{\tge_{\varphi}'(\phi_{1}(\xi))-\dg(\phi_{1}(\xi)),-\mathbf e_{2}}}}} \\
 &\le\frac{\abs{\tge_{\varphi}'(\tfrac{\arcsin(4\pi\xi)}{4\pi})-\dg(\phi_{1}(\xi))}}{\frac14\sqrt{15}\br{\sqrt{1-(\tfrac\pi6)^{2}}-\delta}} \\
 &\le\frac{4\pi\abs{\tfrac{\arcsin(4\pi\xi)}{4\pi}-\phi_{1}(\xi)}
 +\abs{\tge_{\varphi}'(\phi_{1}(\xi))-\dg(\phi_{1}(\xi))}}{\frac14\sqrt{15}\br{\sqrt{1-(\tfrac\pi6)^{2}}-\frac1{16\pi}}} \\
 &\le\frac{\frac{8\pi}{3}\sqrt3+1}{\frac14\sqrt{15}\br{\sqrt{1-(\tfrac\pi6)^{2}}-\frac1{16\pi}}}\cdot\delta \\
 &\le 20\delta \qquad\text{for }\xi\in[-\tfrac1{16\pi},\tfrac1{16\pi}].
 \end{align*}
 We arrive at
 \begin{align*}
  \norm{\tg_{1}-\Tge_{\varphi,1}}_{C^{0}([-\frac{1}{16\pi},\frac{1}{16\pi}])}
  &=\norm{\g\circ\phi_{1}-\tge_{\varphi}(\tfrac{\arcsin(4\pi\cdot)}{4\pi})}_{C^{0}} \\
  &\le\norm{\g\circ\phi_{1}-\tge_{\varphi}\circ\phi_{1}}_{C^{0}}
  +\norm{\tge_{\varphi}\circ\phi_{1}-\tge_{\varphi}(\tfrac{\arcsin(4\pi\cdot)}{4\pi})}_{C^{0}} \\
  &\le\br{1+\tfrac23\sqrt3}\delta \\
  &\le3\delta, \\
  \norm{\tg_{1}'-\Tge_{\varphi,1}'}_{C^{0}([-\frac{1}{16\pi},\frac{1}{16\pi}])}
  &=\norm{\br{\dg\circ\phi_{1}}\phi_{1}'-\tge_{\varphi}'\br{\tfrac{\arcsin(4\pi\cdot)}{4\pi}}\frac1{\sqrt{1-(4\pi\cdot)^{2}}}}_{C^{0}} \\
  &\le
  \norm{\br{\dg\circ\phi_{1}}\br{\phi_{1}'-\frac1{\sqrt{1-(4\pi\cdot)^{2}}}}}_{C^{0}}
  +\tfrac43\norm{\dg\circ\phi_{1}-\tge_{\varphi}'\circ\phi_{1}}_{C^{0}} \\
  &\qquad{} +\tfrac43\norm{\tge_{\varphi}'\circ\phi_{1}-\tge_{\varphi}'\br{\tfrac{\arcsin(4\pi\cdot)}{4\pi}}}_{C^{0}} \\
  &\le 20\delta\norm{\dg}_{C^{0}} + \tfrac43\delta + \tfrac43\cdot4\pi\cdot\tfrac23\sqrt3\delta \\
  &\le\br{20+20\delta+21}\delta \\
  &\le 42\delta <\varepsilon
 \end{align*}
 by our choice of $\delta$. The case $j=2$ is symmetric.
\end{proof}

Now we can state that the cylinder has in fact
the form claimed above.

\begin{lemma}[Two strands in a cylinder]\label{lem:zylinder}
 Let $\varphi\in[0,\pi]$, $\zeta\in\left(0,\tfrac1{96\pi}\right]$, and $\g\in C^{1}\rzd$ with
 \begin{equation}
  \norm{\g-\tge_{\varphi}}_{C^{1}} \le \delta\equiv\delta_{\zeta/2}
 \end{equation}
 where $\delta_{\eps}>0$ is the constant from~\eqref{eq:lgr}.
 Then the intersection of $\g$ with the cylinder
 $\mathcal Z$ defined in \eqref{grundzylinder}
 consists of two (connected) arcs $\tg_{1}$,
 $\tg_{2}$ 
 which enter and leave at the caps of $\mathcal Z$. These arcs
 can be written as graphs over the $\mathbf e_{2}$-axis, so
 each fibre
 $\mathcal Z_{\xi}$
 is met by both $\tg_{1}$ and $\tg_{2}$ transversally in precisely one point
 respectively.
\end{lemma}

Note that the images of $\tg_{1}$ and $\tg_{2}$ might possibly intersect.

\begin{proof}
 Applying Lemma~\ref{lem:lgr} we merely have to show that
 $\mathcal Z$ is not too narrow. We compute for $\abs\xi\le\eta$, $j=1,2$,
 \begin{align*}
  \abs{\widetilde{\tge}_{\varphi,j}(\xi)+\xi\mathbf e_{2}}
  &=\tfrac1{4\pi}\br{1-\cos\arcsin(4\pi\xi)} \\
  &=\tfrac1{4\pi} \frac{(4\pi\xi)^{2}}{1+\sqrt{1-(4\pi\xi)^{2}}}\le4\pi\xi^{2}\le\tfrac\zeta2, \\
  \abs{\tilde\g_{j}(\xi)+\xi\mathbf e_{2}}&\le \tfrac\zeta2+\tfrac\zeta2=\zeta.
 \end{align*}
 Furthermore we have to show that there exist no other intersection
 points of $\g$ with $\mathcal Z$.
 In the neighborhood of the caps of $\mathcal Z$
 there are no such points apart from those belonging to $\tilde\g_{1},\tilde\g_{2}$
 since the tangents of $\g$ transversally meet
 the normal disks of $\tge_{\varphi}$ as follows.
 As $\abs{\g(t)-\tge_{\varphi}(t)}\le\delta\le\frac\eps{42}$,
 all points belong to the $\delta$-neighborhood of $\tge_{\varphi}$,
 and any point $\g(t)$ belongs to a normal disk
 centered at $\tge_{\varphi}(\tilde t)$ with
 $\abs{\g(t)-\tge_{\varphi}(\tilde t)}\le\delta$,
 so
 \begin{equation}\label{eq:t-tilde1}
  \abs{\tge_{\varphi}(\tilde t)-\tge_{\varphi}(t)}\le2\delta.
 \end{equation}
 Furthermore, as $\tge_{\varphi}$ parametrizes a circle on $[0,\tfrac12]$
 and $[\tfrac12,1]$, we arrive at
 \begin{equation}\label{eq:t-tilde2}
  \sin\br{2\pi\abs{\tilde t-t}} = 2\pi\abs{\tge_{\varphi}(\tilde t)-\tge_{\varphi}(t)} \qquad\text{for either }t,\tilde t\in[0,\tfrac12]\text{ or }t,\tilde t\in[\tfrac12,1].
 \end{equation}
 Therefore, the angle between
 $\tge_{\varphi}'(\tilde t)$ and $\tge_{\varphi}'(t)$
 amounts to at most
 \begin{align*}
 \arccos\sp{\tge_{\varphi}'(\tilde t),\tge_{\varphi}'(t)}
 &=\arccos\br{1+\sp{\tge_{\varphi}'(\tilde t)-\tge_{\varphi}'(t),\tge_{\varphi}'(t)}} \\
 &\le\arccos\br{1-\abs{\tge_{\varphi}'(\tilde t)-\tge_{\varphi}'(t)}} \\
 &\le\arccos\br{1-4\pi\abs{\tilde t-t}} \\
 &\refeq{t-tilde2}=\arccos\br{1-2\arcsin\br{2\pi\abs{\tge_{\varphi}(\tilde t)-\tge_{\varphi}(t)}}} \\
 &\refeq{t-tilde1}\le\arccos\br{1-2\arcsin\br{4\pi\delta}} \\
 &\le\arccos\br{1-2\arcsin{\tfrac14}}
 <1.1.
 \end{align*}
 From $\cos x\le1-\tfrac{x^{2}}\pi$ for $x\in[-\tfrac\pi2,\tfrac\pi2]$
 we infer
 \begin{equation}\label{eq:arccos-x}
  x\ge\arccos\br{1-\tfrac{x^{2}}\pi},
 \end{equation}
 so the angle between $\g'(t)$ and $\tge_{\varphi}'(t)$
 is bounded above by
 \begin{equation}\label{eq:arccos}
 \begin{split}
 \arccos\frac{\sp{\g'(t),\tge_{\varphi}'(t)}}{\abs{\g'(t)}}
 &\le\arccos\br{\frac{1+\sp{\dg(t)-\tge_{\varphi}'(t),\tge_{\varphi}'(t)}}{1+\delta}} \\
 &\le\arccos\frac{1-\delta}{1+\delta}
 \overset{\eqref{eq:arccos-x}}{\le}\sqrt{\tfrac{2\pi\delta}{1+\delta}}
 \le\sqrt{2\pi\delta}\le\tfrac{\sqrt2}4<0.4.
 \end{split}
 \end{equation}
 Thus transversality is established by
 \begin{equation}\label{eq:tpc-transversal}
 \angle\br{\g'(t),\tge_{\varphi}'(\tilde t)}
 < 1.5<\tfrac\pi2.
 \end{equation}

 Now we want to determine which points of $\gamma$ actually lie
in the cylinder $\mathcal{Z}$. Such points satisfy
 $\abs{\sp{\g(t),-\mathbf e_{2}}}\le\eta$
 which implies
 \[ \abs{\sp{\tge_{\varphi}(t),-\mathbf e_{2}}}\le\eta+\delta
 \le\tfrac1{16\pi}+\tfrac1{16\pi}=\tfrac1{8\pi}. \]
 This, by definition of $\tge_{\varphi}$ in \eqref{eq:tg8},
 leads to $\abs{\sin(4\pi t)}\le\frac12$
 which defines four connected arcs in $\R/\Z$.
 Two of them, namely $[-\tfrac1{24},\tfrac1{24}]$ and $[\tfrac{11}{24},\tfrac{13}{24}]$, (partially) belong to $\mathcal Z$;
 for the other two, $[\tfrac{5}{24},\tfrac{7}{24}]$ and $[\tfrac{17}{24},\tfrac{19}{24}]$, we obtain
 \begin{align*}
  \dist\br{\tge_{\varphi}(t),\R\mathbf e_{2}}
  &=\frac{1-\cos(4\pi t)}{4\pi}\ge\frac{1-\frac12\sqrt2}{4\pi}, \\
\textnormal{hence}\qquad  \dist\br{\g(t),\R\mathbf e_{2}}
  &\ge\frac{1-\frac12\sqrt2}{4\pi}-\delta
  \ge\frac{4-2\sqrt2-1}{16\pi}
  >\frac1{96\pi}\ge\zeta.
 \end{align*}
\end{proof}

We just have seen that $\mathcal Z$ only contains
two sub-arcs of $\g$ which are parametrized over the $\mathbf e_{2}$-axis.
Thus $\g\setminus\mathcal Z$ consists of two arcs
which both join (different) caps of $\mathcal Z$.
In fact, they have the form of two ``handles''.
We intend to map them onto $\tge_{\varphi}$
by a suitable isotopy.

The actual construction is a little bit delicate as
we construct an isotopy on normal disks of
$\tge_{\varphi}$ which is not consistent with the fibres $\mathcal Z_{\xi}$
of the cylinder.
Therefore, we cannot claim to leave the entire cylinder $\mathcal Z$
pointwise invariant.
Instead, we consider the two middle quarters of $\mathcal Z$,
more precisely
 \[ \mathcal Z' := \sett{(x_{1},x_{2},x_{3})^{\top}\in\R^{3}}{x_{1}^{2}+x_{3}^{2}\le\zeta^{2},\abs{x_{2}}\le\nfrac\eta2}, \]
which will contain the entire ``linking'' of the two arcs
provided $\delta$ has been chosen accordingly.
In fact we will construct an isotopy on
$B_{2\eps}\br{\tge_{\varphi}}\setminus\mathcal Z'$
which
leaves $\mathcal Z'$ pointwise invariant
and maps $\g\setminus\mathcal Z$ to $\tge_{\varphi}\setminus\mathcal Z$.

The idea is first to map $\tg_{j}$ to $\widetilde{\tge}_{\varphi,j}$
on $\abs\xi\in[\nfrac\eta2+\eps,\eta-\eps]$
and to straight lines connecting
$\tg_{j}(\nfrac\eta2)$ to $\widetilde{\tge}_{\varphi,j}(\nfrac\eta2+\eps)$
and
$\widetilde{\tge}_{\varphi,j}(\eta-\eps)$ to $\tg_{j}(\eta)$
on $\abs\xi\in[\nfrac\eta2,\nfrac\eta2+\eps]\cup[\eta-\eps,\eta]$.
A second isotopy
on normal disks of $\tge_{\varphi}$
maps $\g\setminus\mathcal Z$
and the straight line inside $\mathcal Z$ on $\abs\xi\in[\eta-\eps,\eta]$ to $\tge_{\varphi}$.

\begin{figure}\centering
 \includegraphics[scale=.3]{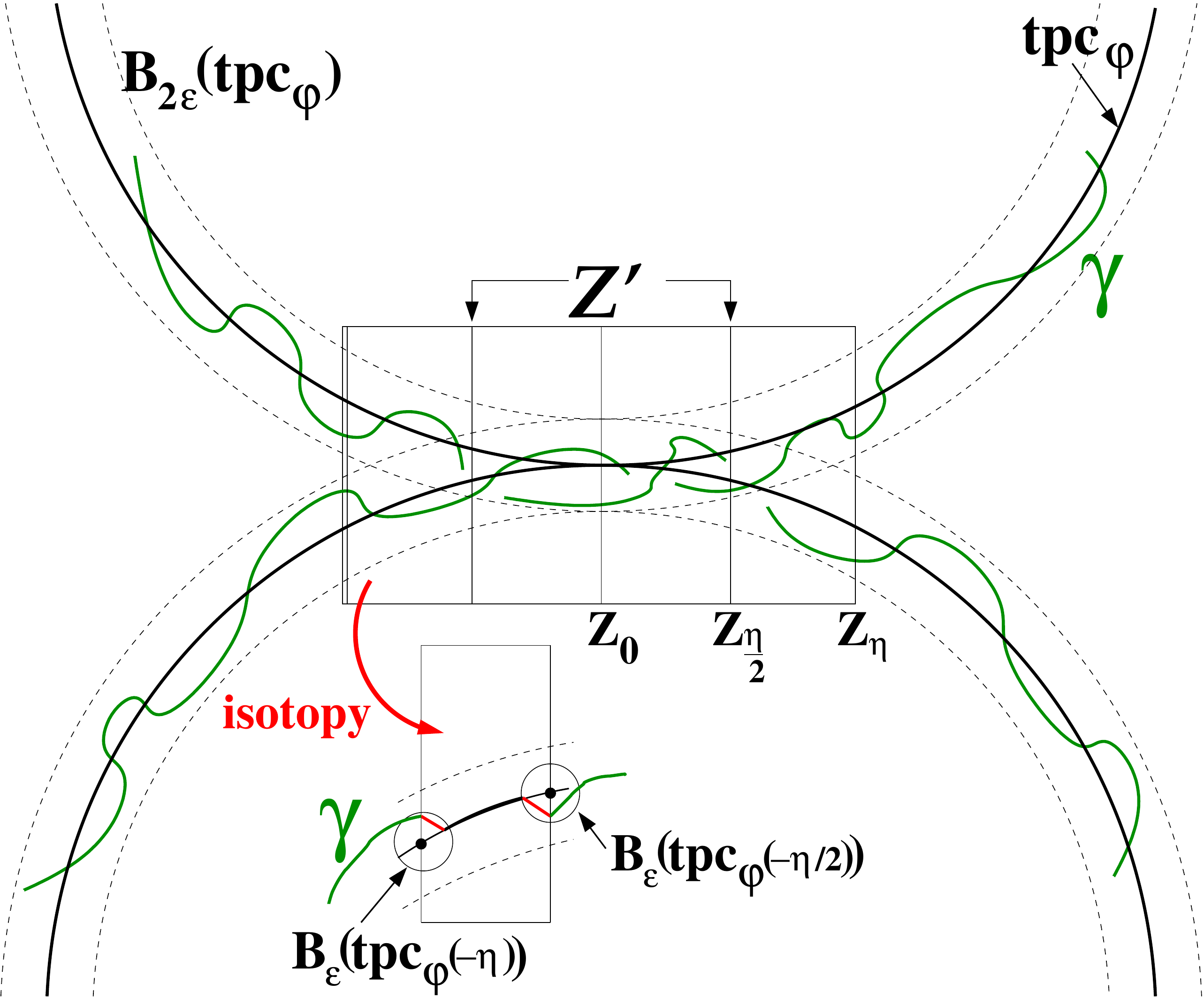}
 \caption{The handle isotopy on $\mathcal Z$: an intermediate stage, where one connects the part of $\tge_{\varphi}$ in $\mathcal{Z}\setminus\mathcal{Z'}$ to $\gamma$ by short straight segments.}\label{fig:zylinder}
\end{figure}

\begin{lemma}[Handle isotopy]\label{lem:isohandle}
 Let $\varphi\in(0,\pi]$, $\zeta\in\left(0,\tfrac1{96\pi}\right]$, and $\g\in C^{1}\rzd$ with
 \begin{equation*}
  \norm{\g-\tge_{\varphi}}_{C^{1}} \le \delta\equiv
  \delta_{\eps}
 \end{equation*}
 where $\delta_{\eps}>0$ is the constant from~\eqref{eq:lgr} and
 \begin{equation}\label{eq:epsilon}
  \eps\equiv\eps_{\zeta}:=\min\br{\nfrac\zeta{64}\sin\nfrac\varphi2,\tfrac1{20}\sqrt{\nfrac{\zeta}{8\pi}}}.
 \end{equation}
 Then there is an isotopy of $B_{2\eps}\br{\tge_{\varphi}}$
 which
 \begin{itemize}
 \item leaves $\mathcal Z'$ and $\partial B_{2\eps}\br{\tge_{\varphi}}$
 pointwise fixed,
 \item  deforms $\g\setminus\mathcal Z$ to $\tge_{\varphi}\setminus\mathcal Z$, and
 \item moves points by at most $4\eps$.
 \end{itemize}
\end{lemma}
\begin{remark}\label{rem:iso}
From the proof it will become clear that the isotopy actually
deforms all of $\gamma$ outside a small $\varepsilon$-neighbourhood $B_\varepsilon(
\mathcal{Z}')$ of the smaller
cylinder $\mathcal{Z}'$ to $\tge_\varphi.$ In addition, in the small region $B_\varepsilon(\mathcal{Z}')\setminus \mathcal{Z}'$ the curve $\gamma$ is
deformed into a straight segment that lies in the $2\varepsilon$-neighbourhood
of $\tge_{\varphi}$.
\end{remark}
 Notice for the following proof that
the $\varepsilon$-neighbourhood of $\tge_\varphi$ coincides with the
union of all $\varepsilon$-normal disks of $\tge_\varphi$.

\begin{proof}
 On circular fibres we may employ an isotopy adapted from
 Crowell and Fox~\cite[App.~I, p.~151]{crowell-fox}.
 For an arbitrary closed circular planar disk $D$
 and given \emph{interior} points $p_{0},p_{1}\in D$ we may define
 an homeomorphism $g_{D,p_{0},p_{1}}:D\to D$
 by mapping any ray joining $p_{0}$ to a point $q$ on the boundary of $D$
 linearly onto the ray joining $p_{1}$ to $q$ so that
 $p_{0}\mapsto p_{1}$ and $q\mapsto q$.
 This leaves the boundary $\partial D$ pointwise invariant.
 Furthermore, $g_{D,p_{0},p_{1}}$ is continuous in $p_{0}$, $p_{1}$ and $D$ (thus especially in
 the center and the radius of $D$).
 Of course, as $g_{D,p_{0},p_{1}}$ maps $D$ onto itself,
 any point is moved by at most the diameter of~$D$.
 The isotopy is now provided by
 the homeomorphism
 \begin{equation}\label{eq:isotopy}
 H:[0,1]\times D\to[0,1]\times D, \qquad
 (\lambda,x)\mapsto(\lambda,g_{D,p_{0},(1-\lambda) p_{0}+\lambda p_{1}}(x))
 \end{equation}
 which analogously works for the ellipsoid.
 
 Now we apply this isotopy to any fibre $\mathcal Z_{\xi}$
 of
 \[ \mathcal Z\setminus\mathcal Z' = \bigcup_{\abs\xi\in[\eta/2,\eta]}\mathcal Z_{\xi}. \]
 Here, for $j=1,2$ and any $\xi$ satisfying
 $\abs\xi\in\sq{\nfrac\eta2,\eta}$,
 we let $\set{p_{0,j}}=\mathcal Z_{\xi}\cap
 \tilde\g_{j}$ and $D_{j}$ be the $2\eps$-ball
 centered at $p_{1,j}:=\widetilde{\tge}_{\varphi,j}(\xi)$.
 As $\eps<\nfrac\zeta4$ we have $D_{j}\subset \mathcal Z_{\xi}$,	
 
 By our choice of $\eps$,
 the disks $D_{1}$ and $D_{2}$ in each fibre are
 disjoint. To see this, we compute using
 $1-\cos\varphi=2\sin^{2}\nfrac\varphi2$
 \begin{align}\label{eq:disjoint-disks}
 \begin{split}
 \abs{\widetilde{\tge}_{\varphi,1}(\xi)-\widetilde{\tge}_{\varphi,2}(\xi)}
 &=\tfrac1{4\pi}\br{1-\cos\arcsin(4\pi\xi)}\sqrt{2-2\cos\varphi} \\
 &=\tfrac1{4\pi}\br{1-\sqrt{1-(4\pi\xi)^{2}}}\cdot2\sin\nfrac\varphi2 \\
 &=\frac{(4\pi\xi)^{2}}{1+\sqrt{1-(4\pi\xi)^{2}}}\cdot\frac{\sin\tfrac\varphi2}{2\pi}
 \ge4\pi\xi^{2}\sin\nfrac\varphi2 \\
 &\ge\pi\eta^{2}\sin\nfrac\varphi2=\nfrac\zeta8\sin\nfrac\varphi2
 \ge8\eps
 \end{split}
 \end{align}
 for any $\xi\in\sq{\nfrac\eta2,\eta}$,
 so $\dist(D_{1},D_{2})
 \ge8\eps-2\cdot2\eps=4\eps>0$.
 
 Now we construct an isotopy on the fibres
 contained in $\mathcal Z\setminus\mathcal Z'$,
 see Figure~\ref{fig:zylinder}.
 We will always employ the isotopy~\eqref{eq:isotopy}
 on the fibres with $p_{0,j}:=\tilde\g_j(\xi)$.

 For $\xi \in \sq{\nfrac\eta2,\nfrac\eta2+\eps}$
 we let $p_{1,j}$ be the intersection
 of $\mathcal Z_{\xi}$ with the straight line
 joining $\tilde\g_{j}(\nfrac\eta2)$ and $\widetilde{\tge}_{\varphi,j}(\nfrac\eta2+\eps)$.
 To this end we have to ensure that
 this line belongs to the $2\eps$-neighborhood
 of $\tge_{\varphi}$.
 (In fact, its interior points belong to $\mathcal Z\setminus\mathcal Z'$ since any cylinder is convex.)
 As $\tilde\g_{j}(\nfrac\eta2)$
 belongs to the $\eps$-neighborhood
 of $\tge_{\varphi}$,
 it is sufficient to apply Lemma~\ref{lem:litn} below.
 To this end, we let $f = \mathbb P\,\widetilde{\tge}_{\varphi,j}$
 where $\mathbb P$ denotes the projection to $\mathbf e_{2}^{\perp}$, and
 \[ f'(\xi) = \frac{\mathbb P\tge_{\varphi}'\br{\frac{\arcsin(4\pi\xi)}{4\pi}}}{\sqrt{1-(4\pi\xi)^{2}}},
 \quad
 f''(\xi) = \frac{\mathbb P\tge_{\varphi}''\br{\frac{\arcsin(4\pi\xi)}{4\pi}}}{{1-(4\pi\xi)^{2}}}
 + (4\pi)^{2}\xi\frac{\mathbb P\tge_{\varphi}'\br{\frac{\arcsin(4\pi\xi)}{4\pi}}}{\br{1-(4\pi\xi)^{2}}^{3/2}}. \]
 As $\abs{\tge_{\varphi}'}\equiv1$, $\abs{\tge_{\varphi}''}\equiv4\pi$
 (up to the points $t=0$, $t=\tfrac12$ of tangential intersection), we arrive at
 $\abs{f''}\le6\pi$
 for any $\xi\in\sq{-\frac1{16\pi},\frac1{16\pi}}$.
 Then, applying Lemma~\ref{lem:litn} to
 $f(\cdot-\nfrac\eta2)$ with $\ell=\eps$, $K=6\pi$, and
 $y=\tilde\g_{j}(\nfrac\eta2)$, the distance of any point of
 the straight line to $f$ is bounded by
 $\sqrt2\cdot 6\pi\eps^{2}+\eps<2\eps$
 for $\eps<\frac1{6\pi\sqrt2}$.
 For future reference we remark that,
 by $\abs{\widetilde{\tge}'_{\varphi,j}}\le\nfrac4{\sqrt{15}}$, the angle between
 the straight line and the $\mathbf e_{2}$-axis
 is bounded above by
 \begin{equation}\label{eq:straight-line-angle}
  \arctan\frac{\abs{\tg_{j}(\nfrac\eta2)-\widetilde{\tge}_{\varphi,j}(\nfrac\eta2)}
  +\abs{\widetilde{\tge}_{\varphi,j}(\nfrac\eta2)-\widetilde{\tge}_{\varphi,j}(\nfrac\eta2+\eps)}}{\br{\nfrac\eta2+\eps}-\nfrac\eta2}
  \le\arctan\tfrac{\eps+4\eps/\sqrt{15}}{\eps}<\tfrac25\pi.
 \end{equation}

 For $\xi \in \sq{\nfrac\eta2+\eps,\eta-\eps}$
 we let $p_{1,j} = \widetilde{\tge}_{\varphi,j}(\xi)$.
 
 For $\xi \in \sq{\eta-\eps,\eta}$
 we let $p_{1,j}$ be the intersection
 of $\mathcal Z_{\xi}$ with the straight line
 joining $\tilde\g_{j}(\eta)$ and $\widetilde{\tge}_{\varphi,j}(\eta-\eps)$.
 We argue as before using Lemma~\ref{lem:litn}.
 This particular straight line ends at one cap of $\mathcal Z$
 and will be moved to $\tge_{\varphi}$ by the second isotopy.
 
 The same construction can be applied for the corresponding negative
 values of $\xi$.
 Now we have obtained the situation sketched in Figure~\ref{fig:zylinder}.
 
 We consider the $2\eps$-neighborhood of $\tge_{\varphi}$.
 It consists of normal disks centered at the points of $\tge_{\varphi}$.
 We restrict to a neighborhood containing those normal disks
 that do not intersect $\mathcal Z'$.
 They cover
 the straight lines in $\mathcal Z$ at $\abs\xi \in \sq{\eta-\eps,\eta}$
 for otherwise there would be some normal disk (of radius $2\eps$)
 intersecting the straight line (at $\abs\xi\in[\eta-\eps,\eta]$)
 and $\mathcal Z'$ (at $\xi\in[-\nfrac\eta2,\nfrac\eta2]$).
 However, by construction, $(\eta-\eps)-\nfrac\eta2=\nfrac\eta2-\eps
 =\tfrac12\sqrt{\nfrac\zeta{8\pi}}-\eps\refeq{epsilon}\ge9\eps$,
 a contradiction.

 In order to apply the isotopy which moves the points of $\g$ to $\tge_{\varphi}$,
 we have to show
 that all points of $\g$ (and the straight lines as well) belong to
 the $2\eps$-neighborhood of $\tge_{\varphi}$ and transversally meet the
 corresponding normal disk.
 
 For the straight lines we have already seen (using
 Lemma~\ref{lem:litn}) that they lie inside
 the $2\eps$-neighborhood of $\tge_{\varphi}$.
 On $\abs\xi\le\eta\le\frac1{\sqrt3\cdot16\pi}$
 the angle $\alpha$ between the $\mathbf e_{2}$-axis
 and the tangent line to $\widetilde{\tge}_{\varphi,j}$ amounts at most
 to
 \begin{equation}\label{eq:axis-angle}
 \arcsin\frac{\frac1{16\pi\sqrt3}}{\frac1{4\pi}}=\arcsin\tfrac1{4\sqrt3}<\tfrac\pi{10};
 \end{equation}
 see Figure~\ref{fig:tangent}.
 \begin{figure}\centering
 \includegraphics[scale=.5]{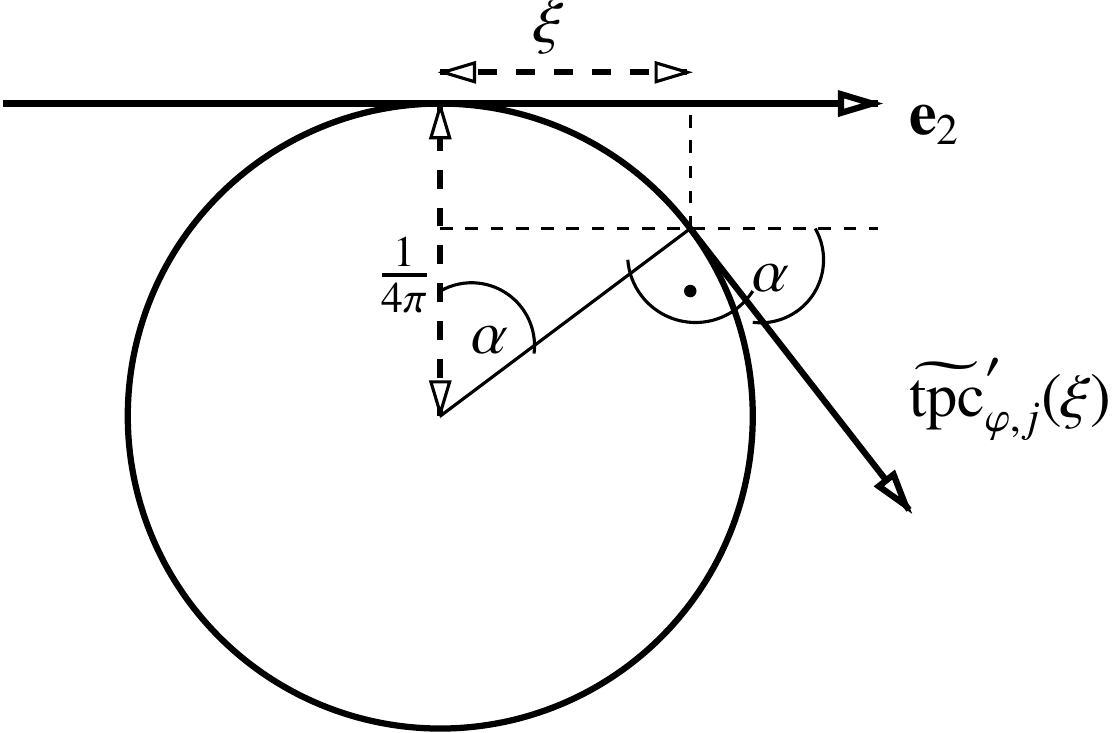}
 \caption{Elementary geometry leading to~\eqref{eq:axis-angle}: $\sin\alpha\le 4\pi\xi\le \tfrac{4\pi}{16\sqrt{3}\pi}=\tfrac{1}{4\sqrt{3}}$.}\label{fig:tangent}
 \end{figure}
 Therefore, using~\eqref{eq:straight-line-angle},
 the angle between the straight line
 and the tangent line to $\widetilde{\tge}_{\varphi,j}$
 is strictly smaller than $\frac\pi2$,
 which implies transversality.

 As calculated above, the distance between $\mathcal Z_{\eta}$
 and $\mathcal Z_{\eta/2}$ amounts to
 at least $10\eps$, so the straight line is covered
 by normal disks which do not intersect $\mathcal Z'$.
  
 For the other points of $\g$ outside $\mathcal Z$ we argue as
 in~\eqref{eq:tpc-transversal}.
\end{proof}

\begin{lemma}[Lines inside a graphical tubular neighborhood]\label{lem:litn}
 Let $f\in C^{2}([0,\ell],\R^{d})$, $\abs{f''}\le K$, $y\in\R^{d}$.
 Then,  $g(x):=f(x) - \br{f(0)+\frac x\ell(y-f(0))}$ satisfies
 \[ \abs{g(x)} \le \sqrt dK\ell^{2} + \abs{y-f(\ell)}
 \qquad\text{for any }x\in[0,\ell]. \]
\end{lemma}

\begin{proof}
 We compute
 \[ g(x) = \int_{0}^{x} \br{f'(\xi)-\frac{f(\ell)-f(0)}{\ell}-\frac{y-f(\ell)}{\ell}}\d\xi. \]
 By the mean value theorem, there are $\sigma_{1},\dots,\sigma_{d}\in(0,\ell)$
 with $f'_{j}(\sigma_{j})=\frac{f_{j}(\ell)-f_{j}(0)}{\ell}$, $j=1,\dots,d$, so
 \begin{align*}
  \abs{g(x)} &\le \int_{0}^{x} \sqrt{\abs{f'_{1}(\xi)-f'_{1}(\sigma_{1})}^{2}
  + \cdots + \abs{f'_{d}(\xi)-f'_{d}(\sigma_{d})}^{2}}\d\xi + \abs{y-f(\ell)} \\
  &\le\sqrt dK\ell^{2}+\abs{y-f(\ell)}.
 \end{align*}
\end{proof}

Recall the definition of the continuous
accumulating angle function $\beta:[-\eta,\eta]\to\R$ in \eqref{betadef},
and
$\Delta_{\beta} = \beta(\eta)-\beta(-\eta)$,
where $a_\xi$ is the segment 
connecting $\tg_1(\xi)$ and $\tg_2(\xi)$ (see \eqref{eq:a-xi}),  
and $\nu$ was given by \eqref{eq:nu}.

\begin{proposition}[Only $(2,b)$-torus knots are $C^1$-close to $\tge_{\varphi}$]\label{prop:braids}
 Let $\varphi\in(0,\pi]$, $\zeta\in\left(0,\tfrac1{96\pi}\right]$, and $\g\in C^{1}\rzd$ be embedded with
 \begin{equation*}
  \norm{\g-\tge_{\varphi}}_{C^{1}} \le \delta
 \end{equation*}
 where $\delta\equiv\delta_{\eps}>0$ and $\eps\equiv\eps_{\zeta}>0$
 are defined in~\eqref{eq:lgr} and~\eqref{eq:epsilon}.
 Let $b\in\Z$ denote the rounded value of
 $\Delta_{\beta}/\pi$.
 Then $b$ is an odd integer, and $\g$
 is unknotted if $b=\pm1$ and belongs to $\tkc$, if $|b|\ge 3.$
\end{proposition}

\begin{figure}\centering
 \includegraphics[scale=.5]{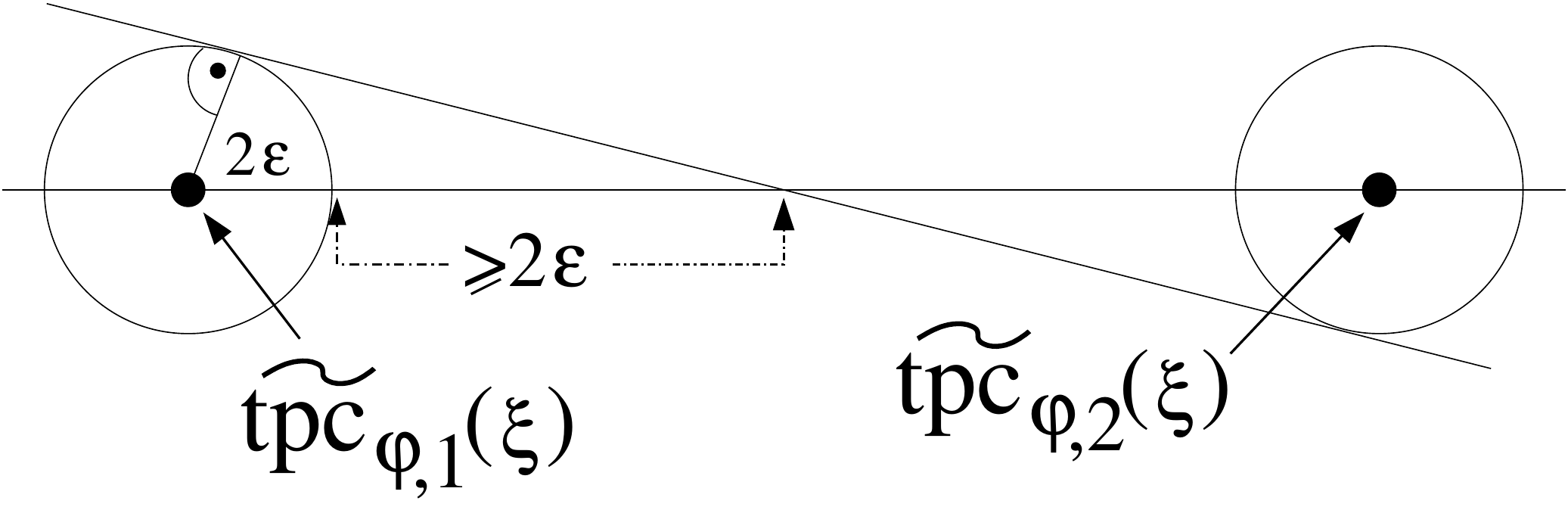}
 \caption{The maximal possible angle between the lines
 defined by $a_{\pm\xi}$ and $\nu$.}\label{fig:a-angle}
\end{figure}

\begin{proof}
 For each $\xi\in [\eta/2,\eta]$ (in fact,
 for any $\xi\in(0,\eta]$) the line $\R\nu$
  is perpendicular to the vectors
  $\chi_{\pm\xi}:=\widetilde{\tge}_{\varphi,1}(\pm\xi)- \widetilde{\tge}_{\varphi,2}(\pm\xi)$, and notice that $\chi_\xi=-\chi_{-\xi}$.

The endpoints $\tg_{1}(\pm\xi)$ and $\tg_{2}(\pm\xi)$
of the vectors $a_{\pm\xi}$ are contained in disks of radius $2\eps$
 inside $\mathcal Z_{\pm\xi}$ centered at $\widetilde{\tge}_{\varphi,j}(\pm\xi)$,
 $j=1,2$, for $\xi\in [\eta/2,\eta]$.
 According to~\eqref{eq:disjoint-disks},
 these disks are disjoint, having distance at least $4\eps$, so that
the usual (unoriented) angle between $\chi_{\pm\xi}$  and $a_{\pm\xi}$
 is bounded by $\arcsin\frac{2\eps}{4\eps}=\frac\pi6$, see Figure~\ref{fig:a-angle}.
Consequently, 
$$
\angle (\nu,a_\xi)\in [\tfrac\pi2-\tfrac\pi6,\tfrac\pi2+\tfrac\pi6]\quad\Foa
|\xi|\in [\tfrac{\eta}2,\eta];
$$
hence, according to~\eqref{betadef}, $\beta(-\eta)\in[\tfrac{3\pi}2-\tfrac\pi6,-\tfrac{3\pi}2+\tfrac\pi6]$.
For the range $\xi\in [\eta/2,\eta]$ such an explicit statement about 
$\beta(\xi)$
(taking values in $\R$) cannot be made,
since the vector $a_\xi$ may have rotated several times while $\xi$ traverses the
interval $[-\eta/2,\eta/2]$, but the vector $a_\xi$ points roughly into
the direction of $\chi_\xi$ for each $\xi\in [\eta/2,\eta]$, which implies 
that there is an integer $m\in\Z$ (counting those rotations) such that
$$
(2m+1)\pi-\tfrac{\pi}3\le \beta(\xi)-\beta(-\xi)
\le (2m+1)\pi+\tfrac{\pi}3\quad
\Foa \xi\in [\eta/2,\eta].
$$
This in turn yields that the rounded value  of $(\beta(\xi)-\beta(-\xi))/\pi$ for all
$\xi\in [\eta/2,\eta)$, and hence also $b$ equals
$(2m+1)$, an odd number.

Therefore, in order to determine the knot type
of $\g$, we may apply Lemma~\ref{lem:isohandle} to deform $\gamma$
into an ambient isotopic curve $\g_{*}$ and analyze that curve instead.
By Remark \ref{rem:iso} the previous arguments apply to $\g_{*}$
as well, in particular the crucial angle-estimate based on 
 Figure~\ref{fig:a-angle}. So the rounded value $b_*$
of $\Delta_{\beta_{*}}$
of the  angle $\beta_{*}(\xi)$
defined by $\mathbf e_{1}\cos\br{\nfrac\varphi2+\beta_{*}(\xi)}+\mathbf e_{3}\sin\br{\nfrac\varphi2+\beta_{*}(\xi)}=a_{*\xi}:=\tilde{\g}_{*1}(\xi)-\tilde{\g}_{*2}(\xi)$
coincides with $b$, since $\gamma_*$ coincides with
$\gamma$ in the subcylinder $\mathcal{Z}'$. 
(By construction we know (see  Lemma~\ref{lem:isohandle} and Remark
\ref{rem:iso})    that $\gamma_*\cap\mathcal{Z}$ consists of
two sub-arcs of $\g_{*}$ 
 transversally meeting each fibre $\mathcal{Z}_\xi$, $\xi\in [-\eta,\eta]$,
of the cylinder $\mathcal Z$, so the vectors $a^*_\xi$ are well-defined.)

 We may now consider the $C^{1}$-mapping
 $\mathfrak n:[-\nfrac\eta2,\nfrac\eta2]\to\S^{1}$,
 $\xi\mapsto\frac{a^*_{\xi}}{\abs{a^*_{\xi}}}\in \Span\{\mathbf e_1,\mathbf e_3\}$.
 By Sard's theorem, almost any direction $\tilde\nu\in\S^{1}$
 is a regular value of $\mathfrak n$, i.e.,
 its preimage consists of isolated (thus, by compactness,
 finitely many) points
 $-\nfrac\eta2\le\xi_{1}<\cdots<\xi_{k}\le\nfrac\eta2$
 (so-called regular points) 
 at which the derivative of $\mathfrak n$ does not vanish.
 At these points we face a self-intersection
 of the two strands inside $\mathcal Z$
 when projecting onto $\tilde\nu^{\perp}$.
 Due to $\mathfrak n'(\xi_{j})\ne0$, $j=1,\dots,k$,
 each of these points can be identified to be either an \emph{overcrossing} ($\overcrossing$) or
 an \emph{undercrossing} ($\undercrossing$).
 We choose some $\psi\in\R/2\pi\Z$ arbitrarily close
 to $\nfrac\varphi2$, such that
 \[ \tilde\nu = \mathbf e_{1}\cos\psi+\mathbf e_{3}\sin\psi \]
 is a regular value of $\mathfrak n$.
 
 As $\g_{*}$ coincides with $\tge_{\varphi}$
 outside $\mathcal Z$ and $a^*_{\xi}/\abs{a^*_{\xi}}$
 is bounded away from the projection line $\R\tilde\nu$, i.e., $\angle(a^*_\xi,
 \nu)\in [\nfrac\pi2-\nfrac\pi6,\nfrac\pi2+\nfrac\pi6]$,
for all $\abs\xi\in[\nfrac\eta2,\eta]$ 
 there are no crossings outside $\mathcal Z'$.
 
 In fact, the projection provides a two-braid presentation of the knot
 as we see two strands transversally passing through the fibres of a narrow
 cylinder in the same direction.
 This is a (two-) \emph{braid}.
 The fact that the strands' end-points
 on one cap of the cylinder are connected to the
 end-points on the opposite cap by two ``unlinked'' arcs (outside $\mathcal Z'$)
 provides a \emph{closure} of the braid.
 
 The isotopy class of any braid
 consisting of $n$ strands is uniquely characterized by
 a \emph{braid word}, i.e., an element of the
 group $\mathcal B_{n}$ which
 is given by $n-1$ generators $\sigma_{1},\dots,\sigma_{n-1}$
 and the relations
 \begin{equation}\label{eq:braid-rel}
 \begin{split}
 \sigma_{j}\sigma_{j+1}\sigma_{j}=\sigma_{j+1}\sigma_{j}\sigma_{j+1}
 &\quad\text{for }j=1,\dots,n-2, \\
 \text{and}\qquad \sigma_{j}\sigma_{k}=\sigma_{k}\sigma_{j}
 &\quad\text{for }1,\dots,j<k-1,\dots,n-2,
 \end{split}
 \end{equation}
 see Burde and Zieschang~\cite[Prop.~10.2, 10.3]{BZ}.
 Closures of braids (resulting in knots or links)
 are ambient isotopic
 if and only if the braids are \emph{Markov equivalent}.
 The latter means that the braids are connected by
 a finite sequence of braids where two consecutive
 braids are either conjugate or
 related by a \emph{Markov move}, see~\cite[Def.~10.21, Thm.~10.22]{BZ}.
 The latter replaces $\mathfrak z\in\mathcal B_{n-1}$
 by $\mathfrak z\sigma_{n-1}^{\pm1}$.
 
 In the special case of two braids this condition simplifies
 as follows.
 As $\mathcal B_{2}$ has only one generator,
 namely $\overcrossing$ with $(\overcrossing)^{-1}=\undercrossing$,
 in fact $\mathcal B_{2}\cong\Z$,
 conjugate braids are identical.
 As $\mathcal B_{1}=\set1$,
 a Markov move can only be applied to the word $1$,
 thus proving that the closed one-braid, i.e., the round circle,
 is ambient isotopic to the closures of both $\overcrossing$
 and $\undercrossing$, which settles the case $b=b_*=\pm 1.$
 
Assume  now $|b|\ge 3$.
 The braid represented by $\tg_{1}$ and $\tg_{2}$
 is characterized by the braid word $(\overcrossing)^{k}$
 for some $k\in\Z$.
 If $k$ were even we would arrive at a two-component link
 which is impossible.
 Any overcrossing $\overcrossing$
 is equivalent to half a rotation of $a_{\xi}$ in positive
 direction (with respect to the $\mathbf e_{1}$-$\mathbf e_{3}$-plane).
 This gives $k= b$.
 On the other hand,
 one easily checks from~\eqref{eq:torus-knot} that
 a $(2,b)$-torus knot
 has the braid word $\sigma_{1}^{-b}$, $b\in1+2\Z$, $b\ne\pm1$
 (cf.~Artin~\cite[p.~56]{artin}). 
\end{proof}

\begin{corollary}[Torus knots]\label{cor:braids}
 Any embedded $\g\in C^{1}\rzd$ with
 $\norm{\g-\tge_{0}}_{C^{1}} \le \tfrac1{100}$
 is either unknotted or belongs to $\tkc$ for some odd $b\ne\pm1$.
\end{corollary}

\begin{sketch}
 We argue similarly to the preceding argument.
 The image of $\tge_{0}$ coincides with the circle of radius $\nfrac1{4\pi}$
 in the $\mathbf e_{1}$-$\mathbf e_{2}$-plane centered at $\nfrac{\mathbf e_{1}}{8\pi}$.
 Consider the $\frac1{100}$-neighborhood of $\tge_{0}$
 fibred by the normal disks of this circle. 
 Any of these normal disks is transversally met by $\g$ in precisely two points.
 Consider the Gau{\ss} map $\mathfrak n:\R/\Z\times\R/\Z\to\S^{2}$,
 $(s,t)\mapsto\frac{\g(s)-\g(t)}{\abs{\g(s)-\g(t)}}$.
 Off the diagonal, this map is well-defined and $C^{1}$.
 Sard's lemma gives the existence of some $\nu_{0}\in\S^{2}$
 arbitrarily close to $\mathbf e_{3}$,
 such that any crossing of $\g$ in the projection onto $\nu_{0}^{\perp}$
 is either an over- or an undercrossing.
 Here we face the situation of a deformed cylinder with its caps
 glued together.
 By stretching and deforming, we arrive at a usual braid representation.
 We conclude as before.
\end{sketch}

% !TEX root = regelast.tex

\section{Comparison  $(2,b)$-torus knots and energy estimates}\label{sect:torus}

Let $a,b\in\Z\setminus\set{-1,0,1}$ be coprime, i.e.,\@ $\gcd(\abs a,\abs b)=1$.
The \emph{$(a,b)$-torus knot class} $\tkc[a,b]$ contains the one-parameter
family of curves
\begin{equation}\label{eq:torus-knot}
 \tau_\rho:t\mapsto
 \begin{pmatrix}
  (1+\rho\cos(bt)) \cos(at) \\
  (1+\rho\cos(bt)) \sin(at) \\
     \rho          \sin(bt)
 \end{pmatrix}, \quad t\in\R/2\pi\Z,
\end{equation}
where the parameter $\rho\in(0,1)$ can be chosen arbitrarily.
For information on torus knots we refer to Burde and Zieschang~\cite[Chapters~3~E, 6]{BZ}.
 As $\tkc[-a,b]=\tkc[a,-b]$ \cite[Prop.~3.27]{BZ},
it suffices to consider $a>1$.
Since $\tkc[a,-b]$ contains the mirror images of $\tkc[a,b]$
we may also, keeping in mind this symmetry, pass to $b>1$.
Note, however, that the latter classes are in fact disjoint,
i.e., torus knots are not amphicheiral~\cite[Thm.~3.29]{BZ}.

We will later restrict to $a=2$; in this case $\gcd(2,b)=1$
holds for any odd $b$ with $|b|\ge 3$.
 The two (mirror-symmetric) trefoil knot classes
coincide with the $(2,\pm3)$-torus knot classes.
\begin{figure}
 \centering
 \includegraphics{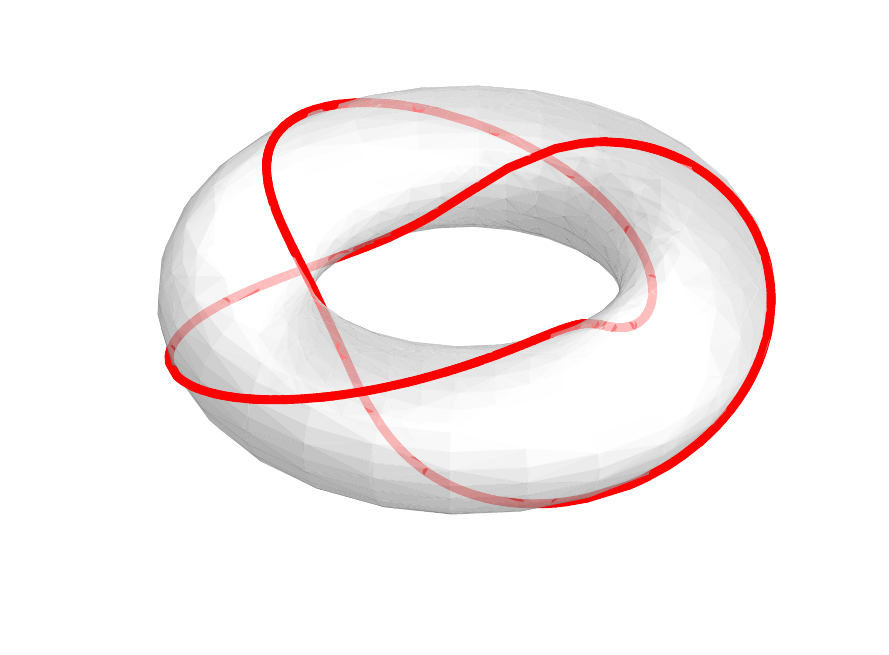}
 \caption{Plot of $\tau_{1/3}$ for the trefoil knot class $\tkc[2,3]$}
\end{figure}

The \emph{total curvature} of a given curve $\g\in H^2(\R/L\Z,\R^3)$, $L>0$, is given by
\begin{equation}\label{eq:tc}
\TC(\g)=\int_{\g}\kappa\d s = \int_{\R/L\Z} {\kappa} \abs{\dg}\d t = \int_0^{L} \frac{\abs{\ddg\wedge\dg}}{\abs{\dg}^2}\d t
 \stackrel{\abs\dg\equiv1}=
 \int_0^{L} {\abs{\ddg}}\d t.
\end{equation}

The torus knots $\tau_\rho$ introduced in~\eqref{eq:torus-knot}
lead to a family of comparison curves that approximate
the $a$-times covered circle $\tau_{0}$ with respect to the $C^{k}$-norm
for any $k\in\N$ as $\rho\searrow0$.
Rescaling and reparametrizing to arc-length
we obtain
$\tilde\tau_\rho\in\cC({{\tkc[a,b]}})$.
As this does not destroy $H^{2}$-convergence~\cite[Thm.~A.1]{reiter:rkepdc},
we find that the arclength parametrization $\tilde\tau_{0}$ of the
$a$-times covered circle lies in the (strong) $H^2$-closure of
$\cC({{\tkc[a,b]}})$.

\begin{lemma}[Bending energy estimate for comparison torus knots]\label{lem:tau-estimate}
 There is a constant $C=C(a,b)$ such that
 \begin{equation}\label{eq:tau-estimate}
  \Eb(\tilde\tau_\rho)\le (2\pi a)^2 + C\rho^2 \qquad\text{for all }\rho\in\sq{0,\tfrac{a}{4\sqrt{a^2+b^2}}}\,.
 \end{equation}
\end{lemma}

\begin{proof}
 We begin with computing the first derivatives of $\tau_\rho$,
 \begin{align}
 \tau_\rho'(t)
 &=
 \begin{pmatrix}
  -b\rho\sin(bt)\cos(at)-a(1+\rho\cos(bt))\sin(at) \\
  -b\rho\sin(bt)\sin(at)+a(1+\rho\cos(bt))\cos(at) \\
  b\rho\cos(bt)
 \end{pmatrix}, \notag\\
 \tau_\rho''(t)
 &=  -
 \begin{pmatrix}
  \br{a^2+(a^2+b^2)\rho\cos(bt)} \cos(at) - 2ab\rho\sin(bt)\sin(at) \\
  \br{a^2+(a^2+b^2)\rho\cos(bt)} \sin(at) + 2ab\rho\sin(bt)\cos(at) \\
  b^2\rho\sin(bt)
 \end{pmatrix}, \notag\\
 \abs{\tau_\rho'(t)}^2
 &=
 b^2\rho^2 + a^2\br{1+\rho\cos(bt)}^2\notag \\
 &= a^2  + 2a^2\cos(bt)\rho + \br{b^2+a^2\cos^2(bt)}\rho^2, \label{first-deriv-taurho}\\
 \abs{\tau_\rho''(t)}^2
 &=
 a^4 + 2a^2(a^2+b^2)\cos(bt)\rho + \sq{(4a^2+b^2)b^{2} + a^{2}(a^2-2b^2)\cos^{2}(bt)}\rho^2,\notag \\
 \sp{\tau_\rho''(t),\tau_\rho'(t)}
 &=
 -a^2b\rho\sin(bt)\br{1+\rho\cos(bt)},\notag
 \end{align}
 which by means of the Lagrange identity implies
 \begin{align*}
 \Eb(\tau_\rho)
 &=
 \int_0^{2\pi}\frac{\abs{\tau_\rho''(t)\wedge\tau_\rho'(t)}^2}{\abs{\tau_\rho'(t)}^5}\d t
 =\int_0^{2\pi}\frac{\abs{\tau_\rho''(t)}^2\abs{\tau_\rho'(t)}^2-\sp{\tau_\rho''(t),\tau_\rho'(t)}^2}{\abs{\tau_\rho'(t)}^5}\d t \\
 &=\int_0^{2\pi}\br{\frac{\abs{\tau_\rho''(t)}^2}{\abs{\tau_\rho'(t)}^3}-\frac{\sp{\tau_\rho''(t),\tau_\rho'(t)}^2}{\abs{\tau_\rho'(t)}^5}}\d t.
 \end{align*}
 As 
\begin{equation}\label{abl-reg}
\abs{\tau_\rho'(t)}^2 \ge \tfrac{7}{16}a^2\quad\textnormal{
because $(a^{2}+b^{2})\rho^{2}\le \tfrac1{16}a^2$ and $2a^{2}\rho\le \tfrac12a^{2}$}
\end{equation}
 the subtrahend in $\Eb(\tau_\rho)$
and the third term in the expression for $|\tau_\rho''|^2$ are bounded by $C\rho^2$ uniformly in~$t$.
 Expanding $z^{-3/2} = 1-\tfrac32z+\mathcal{O}(z^2)$ as $z\to 0$
 we derive
 \[ |\tau_\rho'(t)|^{-3}
 = {\br{b^2\rho^2 + a^2\br{1+\rho\cos(bt)}^2}^{-3/2}}
 = \frac1{a^{3}} - \frac{3}{a^{3}}\cos\br{bt}\rho + \mathcal O(\rho^{2}), \]
 which yields
 \begin{align*}
 \Eb(\tau_\rho) = &\int_0^{2\pi}\br{a^4 + 2a^2(a^2+b^2)\cos(bt)\rho}\br{\frac1{a^{3}} - \frac{3}{a^{3}}\cos\br{bt}\rho + \mathcal O(\rho^{2})}\d t + 
\mathcal{O}(\rho^2)\\
  = &\int_0^{2\pi}\br{a + \tfrac{2b^2-a^2}{a}\cos\br{bt}\rho}\d t + \mathcal{O}(\rho^2)=2\pi a+ \mathcal{O}(\rho^2)\quad\textnormal{as $\rho\searrow 0$.}
 \end{align*}
 Finally, in order to pass to $\tilde\tau_\rho$, we recall that $\Eb$ is invariant under reparametrization
 and $\Eb(r\g)=r^{-1}\Eb(\g)$ for $r>0$. So the claim follows for $r:=\length(
\tau_\rho)^{-1}$ by
 \begin{align*}
  \length(\tau_\rho)
  &= \int_0^{2\pi} \abs{\tau_\rho'(t)}\d t
  = \int_0^{2\pi} \sqrt{b^2\rho^2 + a^2\br{1+\rho\cos(bt)}^2}\d t \\
  &= \int_0^{2\pi} \br{a + a\cos(bt)\rho + \mathcal O(\rho^2)}\d t
  = 2\pi a + \mathcal O(\rho^2)\quad\textnormal{as $\rho\to 0$.}
 \end{align*}
\end{proof}

\begin{proposition}[Ropelength estimate for comparison torus knots]\label{prop:torus-rate}
We have
\begin{equation*}
 \Er(\tilde\tau_\rho) \le \frac C\rho \qquad
 \text{for }0<\rho\ll1,
\end{equation*}
where $C$ is a constant depending only on $a$ and $b$.
\end{proposition}

\begin{proof}
  The invariance of $\Er$ under reparametrization, scaling, and translation
 implies $\Er(\tilde\tau_\rho)=\Er(\tau_\rho)$. 
 According to Lemma~\ref{lem:tau-estimate} above, the
 squared curvature of~$\tau_{\rho}$
 amounts to
 \[ \frac{\abs{\tau_\rho''(t)\wedge\tau_\rho'(t)}^{2}}{\abs{\tau_\rho'(t)}^6}
  \le \frac{\abs{\tau_\rho''(t)}^2}{\abs{\tau_\rho'(t)}^4} \]
 which is uniformly bounded independent of $\rho$ and $t$ as long as
$\rho$ is in the range required in~\eqref{eq:tau-estimate} by some $\kappa_{0}>0$ (combine \eqref{first-deriv-taurho} with \eqref{abl-reg}).
 By $\length(\tau_{\rho}|_{[s,t]})$ we denote the length
 of the shorter subarc of~$\tau_{\rho}$ connecting the points
 $\tau_{\rho}(s)$
 and $\tau_{\rho}(t)$.
 As $\abs{\tau_{\rho}'}$ uniformly converges to $a$ 
as $\rho\to 0$ (see \eqref{first-deriv-taurho}),
 we may find some $\rho_{0}\in\left(0,\tfrac{a}{4\sqrt{a^2+b^2}}\right]$ 
such that
 $a\abs{s-t}_{\R/2\pi\Z}\ge\frac12\length(\tau_{\rho}|_{[s,t]})$
 for any $\rho\in(0,\rho_{0}]$.
 The estimate
 \begin{equation}\label{eq:lower-dist-bound}
  \abs{\tau_{\rho}(s)-\tau_{\rho}(t)} \ge c\rho
  \quad\text{for any }s,t\in\R/2\pi\Z \text{ with }
  \abs{s-t}_{\R/2\pi\Z}\ge\tfrac1{20a\kappa_{0}}
  \text{ and }\rho\in(0,\rho_{0}]
 \end{equation}
 which will be proven below for some uniform $c>0$ now implies
 \[ \abs{\tau_{\rho}(s)-\tau_{\rho}(t)} \ge c\rho
 \quad\text{for any }s,t\in\R/2\pi\Z \text{ with }
 \mathscr L(\tau_\rho|_{[s,t]})\ge\tfrac1{10\kappa_0}
  \text{ and }\rho\in(0,\rho_{0}]. \]
 As $\length(\tau_{\rho})$ is also uniformly bounded, we may apply
 Lemma~\ref{lem:qtb} below, which yields the desired.
 
 It remains to verify~\eqref{eq:lower-dist-bound}.
 Recall that the limit curve $\tau_{0}$ is an $a$-times covered circle
 parametrized by uniform speed $a$. Fix $t\in\R/2\pi\Z$ and
 consider the image of~$\tau_{\rho}(s)$ restricted to
 $s\in\R/2\pi\Z$ with
 $\abs{s-t}_{\R/\frac{2\pi}a\Z}<\frac\pi{2ab}$,
 i.e.\@ 
 $s\in t+\frac{2\pi}a\Z+(-\frac\pi{2ab},\frac\pi{2ab})$,
 see Figure~\ref{fig:torus-reg}.
 It consists of $a$ disjoint arcs (on the $\rho$-torus,
 in some neighborhood of $\tau_{0}(t)$).
 The associated angle function~$s\mapsto bs$ (see~\eqref{eq:torus-knot})
 strides across \emph{disjoint} regions of length $\frac\pi a$ on $\R/2\pi\Z$ for each of
 those arcs.
 These regions have positive distance $\ge\frac\pi{a}$
 on the surface of the $\rho$-torus which
 leads to a uniform positive lower bound on $\frac1\rho\abs{\tau_{\rho}(s)-\tau_{\rho}(t)}$
 for $s\in\R/2\pi\Z$ where
 $\abs{s-t}_{\R/\frac{2\pi}a\Z}<\frac\pi{2ab}$ but
 $\abs{s-t}_{\R/2\pi\Z}>\frac\pi{2ab}$.
 The latter restriction reflects the fact that each arc has zero
 distance to itself.

 \begin{figure}
 \centering
 \includegraphics{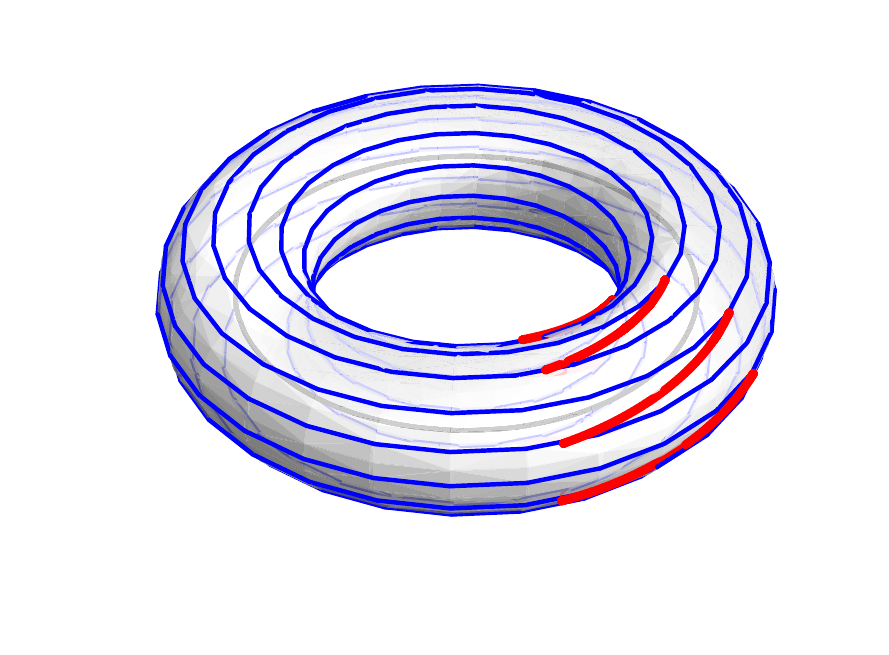}
 \caption{Plot of the regions of $\tau_{1/3}$
 for $a=7$, $b=3$ (red; with respect to a fixed $t\in\R/\Z$)
 mentioned in the proof of Proposition~\ref{prop:torus-rate}.
 The blue lines  visualize the
 (intrinsic) distance between the strands.}
 \label{fig:torus-reg}
\end{figure}
 
 If, on the other hand,
 $\abs{s-t}_{\R/\frac{2\pi}a\Z}\ge\lambda:=\min\br{\frac\pi{2ab},\frac1{20a\kappa_{0}}}$,
 we may derive
 \begin{align*}
  \abs{\tau_{\rho}(s)-\tau_{\rho}(t)}
  \ge\abs{\tau_{0}(s)-\tau_{0}(t)} - 2\rho
  &\ge2\sin\br{\tfrac{a}{2}\abs{s-t}_{\R/2\pi\Z}}-2\rho_{0} \\
  &\ge2\sin\br{\tfrac{a}{2}\abs{s-t}_{\R/\frac{2\pi}a\Z}}-2\rho_{0}
  \ge2\sin{\tfrac{a\lambda}{2}}-2\rho_{0}.
 \end{align*}
 Diminishing $\rho_{0}$ if necessary, the right-hand side is positive for all
 $\rho\in(0,\rho_{0}]$.
\end{proof}

\begin{lemma}[Quantitative thickness bound]\label{lem:qtb}
 Let $\g\in C^{2}\rzd$ be a regular curve with
 uniformly bounded curvature $\kappa\le\kappa_{0}$, $\kappa_{0}>0$, and
 assume that
 \[ \delta := \inf\sett{\abs{\g(s)-\g(t)}}{s,t\in\R/\Z,\mathscr L(\g|_{[s,t]})\ge\tfrac1{10\kappa_0}} \]
 is positive
 where $\length(\g|_{[s,t]})$ denotes the length of the shorter sub-arc of~$\g$
 joining the points $\g(s)$ and $\g(t)$.
 Then $\triangle[\g]\ge\min\br{\tfrac\delta2,\frac1{\kappa_{0}}}>0$, thus
 \[ \Er(\g)\le{\max\br{\tfrac2\delta,\kappa_0}}{\mathscr L(\g)}<\infty. \]
\end{lemma}

\begin{proof}
 As the quantities in the statement do not depend on the actual
 parametrization and distances, thickness, and the reciprocal
 curvature are positively homogeneous of degree one,
 there is no loss of generality in assuming arc-length parametrization.
 According to~\cite[Thm.~1]{lsdr}, the thickness equals
 the minimum of $\nfrac1{\max\kappa}\ge\nfrac1{\kappa_{0}}$
 and one half of the \emph{doubly critical self-distance}, that is,
 the infimum over all distances $\abs{\g(s)-\g(t)}$
 where $s,t\in\R/\Z$ satisfy $\dg(s)\perp\g(s)-\g(t)\perp\dg(t)$.
 By our assumption we only need to show that the doubly critical self-distance
is not attained on the parameter range where
  $\length(\g|_{[s,t]})=\abs{s-t}_{\R/\Z}<\tfrac1{10\kappa_0}$.
 
 To this end we show that any angle between
 $\dg(t)$ and $\g(s)-\g(t)$
 is  smaller than $\frac\pi6<\frac\pi2$ if $\abs{s-t}_{\R/\Z}\le\tfrac1{10\kappa_0}$.
 We obtain for $w:=\abs{s-t}_{\R/\Z}\in\sq{0,\frac1{10\kappa_{0}}}$
 (note that $\frac1{10\kappa_{0}}<\frac12$ by Fenchel's theorem)
 \begin{align*}
  \abs{\sp{\g(s)-\g(t),\dg(t)}}
  &= \abs{(s-t) + (s-t)^{2}\int_{0}^{1}(1-\th)\sp{\ddg(t+\th(s-t)),\dg(t)}\d\th} \\
  &\ge w\br{1-\kappa_{0}w} \ge \tfrac9{10}w.
 \end{align*}
 Using the Lipschitz continuity of~$\g$, this yields
 for the angle $\a\in[0,\frac\pi2]$ between the lines parallel to
 $\g(s)-\g(t)$ and $\dg(t)$
 \[ \cos\a
 = \frac{\abs{\sp{\g(s)-\g(t),\dg(t)}}}{\abs{\g(s)-\g(t)}}
 \ge \tfrac9{10} > \tfrac12\sqrt3=\cos\tfrac\pi6
 \qquad\Longrightarrow\qquad\a<\tfrac\pi6. \]
\end{proof}

Combining  Lemma~\ref{lem:tau-estimate} with the previous ropelength
estimate we can use the comparison torus knots to obtain non-trivial
growth estimates on the total energy $\Eth$ and on the ropelength
$\Er$ of minimizers $\g_\th$.

\begin{proposition}[Total energy growth rate for minimizers]\label{prop:estimate-torus}
 For $a,b\in\Z\setminus\set{-1,0,1}$ with $\gcd(\abs a,\abs b)=1$ there is a positive constant  $C=C(a,b)$ such that
 any sequence $\seq[\th>0]\gth$ of $E_\th$-minimizers in $\cC(\mathcal{T}(a,b))$ satisfies
 \begin{equation}\label{eq:energy-growth-rate}
  \Eth(\gth) \le (2a\pi)^{2}+C{\th^{2/3}}
 \end{equation}
  and if $a=2$
 \begin{equation}\label{eq:ropelength-growth-rate}
  \Er(\gth) \le C{\th^{-1/3}}.
 \end{equation}
 \end{proposition}

\begin{proof}
 The first claim immediately follows by $\Eth(\gth)\le\Eth(\tilde{\tau}_{\th^{1/3}})$ from Lemma~\ref{lem:tau-estimate} and
 Proposition~\ref{prop:torus-rate}.
 The classic F\'ary--Milnor Theorem applied to $\g_\th$ gives
 $(4\pi)^{2} + \th\Er(\gth) \le \Eth(\gth)$.
 Now the first estimate for $a=2$ implies the second one.
\end{proof}
% !TEX root = regelast.tex
\section{Crookedness estimate and the elastic $(2,b)$-torus knot\\
Proofs of the main theorems}\label{sec:5}

In his seminal article~\cite{milnor} on the F\'ary--Milnor theorem,
Milnor
derived
the lower bound for the total curvature of knotted arcs
by studying the \emph{crookedness} of a curve and relating it to
the total curvature.
For some \emph{regular} curve $\g\in C^{0,1}\rzd$,
i.e., a Lipschitz-continuous mapping $\R/\Z\to\R^{d}$
which is not constant on any open subset of $\R/\Z$,
the crookedness of~$\g$ is the infimum over all $\nu\in\S^{2}$ of
\[ \mu(\g,\nu) := \#\sett{t_{0}\in\R/\Z}{t_{0} \textnormal{ is a local maximizer of }
t\mapsto\sp{\g(t),\nu}_{\R^{3}}}. \]
We briefly cite the main properties,
proofs can be found in~\cite[Sect.~3]{milnor}.

\begin{proposition}[Crookedness]\label{prop:crook}
 Let $\g\in C^{0,1}\rzd$ be a regular curve. Then
 \[ \TC(\g) = \tfrac12\int_{\S^{2}}\mu(\g,\nu)\d\area(\nu). \]
 Any partition of $\R/\Z$ gives rise to an inscribed regular polygon $p\in C^{0,1}\rzd$
 with $\mu(p,\nu)\le\mu(\g,\nu)$ for all directions $\nu\in\S^{2}$ that are
 not perpendicular to some edge of~$p$.
\end{proposition}

The following statement is the heart of the argument  for Theorem~\ref{thm:main}.

\begin{lemma}[Crookedness estimate]\label{lem:crook}
 Let $\varphi\in(0,\pi]$,
 \[ \zeta := \tfrac12\min\br{\tfrac{1}{8\pi}\sin\nfrac\varphi4,\tfrac1{96\pi}}, \]
 and $\g\in C^{1,1}\rzd$ be
  a non-trivially knotted, arclength parametrized curve
 with  
 \begin{equation}\label{eq:g-phi}
  \norm{\g-\tge_{\varphi}}_{C^{1}}\le\delta
 \end{equation}
  where $\delta\equiv\delta_{\eps}>0$ and $\eps\equiv\eps_{\zeta}>0$
 are defined in~\eqref{eq:lgr} and~\eqref{eq:epsilon}.
 Then the set
 \[ \mathcal B(\g) := \sett{\nu\in\S^{2}}{\mu(\g,\nu)\ge3} \]
 is measurable with respect to the two-dimensional Hausdorff measure $\sH^2$ on~$\S^{2}$ and satisfies
 \[ \area{(\mathcal B(\g))} \ge
 \tfrac\pi{16}\cdot\frac{\varphi}{\Er(\g)}. \]
\end{lemma}

\begin{proof}
 As to the measurability of $\mathcal B(\g)$,
 we may consider the two-dimensional Hausdorff measure $\area$
 on $\S^{2}\setminus\mathscr N$
 where $\mathscr N$ denotes the set of measure zero where $\mu(\g,\cdot)$
 is infinite.
 From the fact 
 that $\mu(\g,\cdot)$ is $\area$-a.e.\@ lower semi-continuous on $\S^{2}
 \setminus\mathscr N$
 if $\g\in C^{1}$,
 we infer that $\sett{\nu\in\S^{2}\setminus\mathscr N}{\mu(\g,\nu)\le2}$ is closed,
 thus measurable.
 Therefore, its complement (which coincides with $\mathcal B(\g)$
 up to a set of measure zero) is also measurable.

 As $\g$ is embedded and $C^{1,1}$, its thickness
 $\triangle[\g]=\nfrac1{\Er(\g)}$ is positive.

 We consider the diffeomorphism
 $\Phi:(0,\pi)\times\R/2\pi\Z\to\S^{2}\setminus\set{\pm\mathbf e_{2}}$,
 \[ (\th,\psi)\mapsto
 \begin{pmatrix}-\sin\th\sin\psi\\\cos\th\\\sin\th\cos\psi \end{pmatrix}. \]

 We aim at showing that, for a.e.\@ $\psi\in[\nfrac\varphi4,\nfrac{3\varphi}4]$,
 there is some sub-interval $J_{\psi}\subset(\tfrac\pi6,\tfrac{5\pi}6)$
 with $\abs{J_{\psi}}\ge\tfrac\pi4\triangle[\g]$ and
 $\Phi(\th,\psi)\in\mathcal B(\g)$ for any $\th\in J_{\psi}$.
 In this case we have
 \[ \area(\mathcal B(\g))
 = \iint_{\S^{2}\setminus\set{\pm\mathbf e_{2}}}\chi_{\mathcal B(\g)}\d\area
 =\int_{0}^{\pi}\int_{\R/2\pi\Z}\chi_{\Phi^{-1}(\mathcal B(\g))}\sin\th\d\psi\d\th
 \ge\tfrac\pi{16}\varphi\triangle[\g]. \]
 
 For given $\varphi\in (0,\pi]$ let
 \[ \tilde{\zeta} := 2\zeta = \min\br{\tfrac{1}{8\pi}\sin\nfrac\varphi4,\tfrac1{96\pi}} \]
 and denote the corresponding cylinder by $\widetilde{\mathcal Z}$,
 cf.\@ Section~\ref{sec:class}.
 For an arbitrary $\psi\in[\nfrac\varphi4,\nfrac{3\varphi}4]$.
 The orthogonal projection onto $\nu_{\psi}^{\perp}$ for
 \[ \nu_{\psi} := \cos\psi\mathbf e_{1}+\sin\psi\mathbf e_{3}
 = \Phi(\tfrac\pi2,\psi-\tfrac\pi2) \]
 will be denoted by $\P$.
 Note that $\nu_{\psi}\perp\Phi(\th,\psi)$ for any $\th\in(0,\pi)$.

 Firstly we note that the half circles $h$
 of $\tge_{\varphi}$ apart from the cylinder $\widetilde{\mathcal Z}$,
 namely $\tge_{\varphi}$
 restricted to $[\tfrac18,\tfrac38]$ and $[\tfrac58,\tfrac78]$ respectively,
 do not interfere with $\widetilde{\mathcal Z}$ when projected onto $\nu_{\psi}^{\perp}$,
  more precisely, $\P h\cap\P\widetilde{\mathcal Z}=\emptyset$.
 To see this, consider the projection $\Pi$ onto $\mathbf e_{2}^{\perp}$,
 see Figure~\ref{fig:e2}.
 The two segments connecting $\tge_{\varphi}(\tfrac18)$
 to $\tge_{\varphi}(\tfrac38)$, and $\tge_{\varphi}(\tfrac58)$
 to $\tge_{\varphi}(\tfrac78)$, respectively,  are mapped onto
 points  $P,P'\in\mathbf e_2^\perp$ under the projection $\Pi$,
 and $P$ and $P'$ separate $\Pi(h)$ and $\Pi(\widetilde{\mathcal Z})$.
 From Figure~\ref{fig:e2} we read off that
 no projection line of $\P$ meets both
 $\mathcal Z$ and $h$,
since otherwise, such a projection line (parallel to $\nu_\psi$
and orthogonal to $\mathbf e_{2}$ by definition)
projected onto $\mathbf e_2^\perp$
under $\Pi$, would intersect both, $\Pi(\widetilde{\mathcal Z})$ and $\Pi(h)$,
which is impossible.

 \begin{figure}\centering
 \includegraphics[scale=.5]{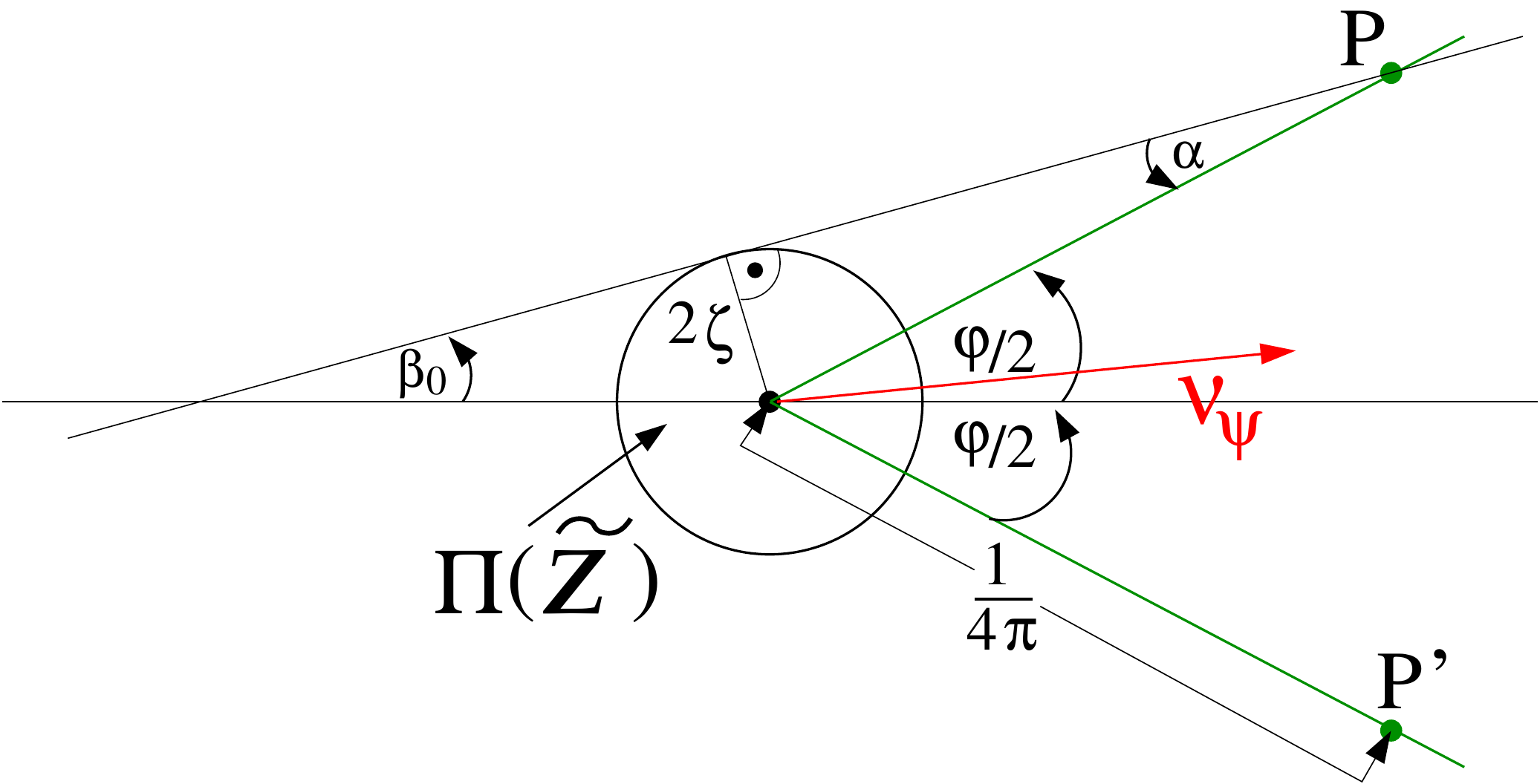}
 \caption{Under the projection $\Pi$ onto $\mathbf e_{2}^{\perp}$,
 the cylinder $\widetilde{\mathcal Z}$ is mapped to a disk.
 Each line enclosing an angle smaller than $\beta_{0}$
 with $\nu$ cannot meet both this circle and (the line) $\Pi(h)$.
 Note that $\beta_{0}=\nfrac\varphi2-\alpha$
 and $\sin\a=8\pi\zeta\le\tfrac12\sin\nfrac\varphi4<\sin\nfrac\varphi4$,
 so $\a<\nfrac\varphi4$ and $\beta_{0}>\nfrac\varphi4$.}\label{fig:e2}
 \end{figure}
 
 \begin{figure}\centering
 \includegraphics[scale=.5]{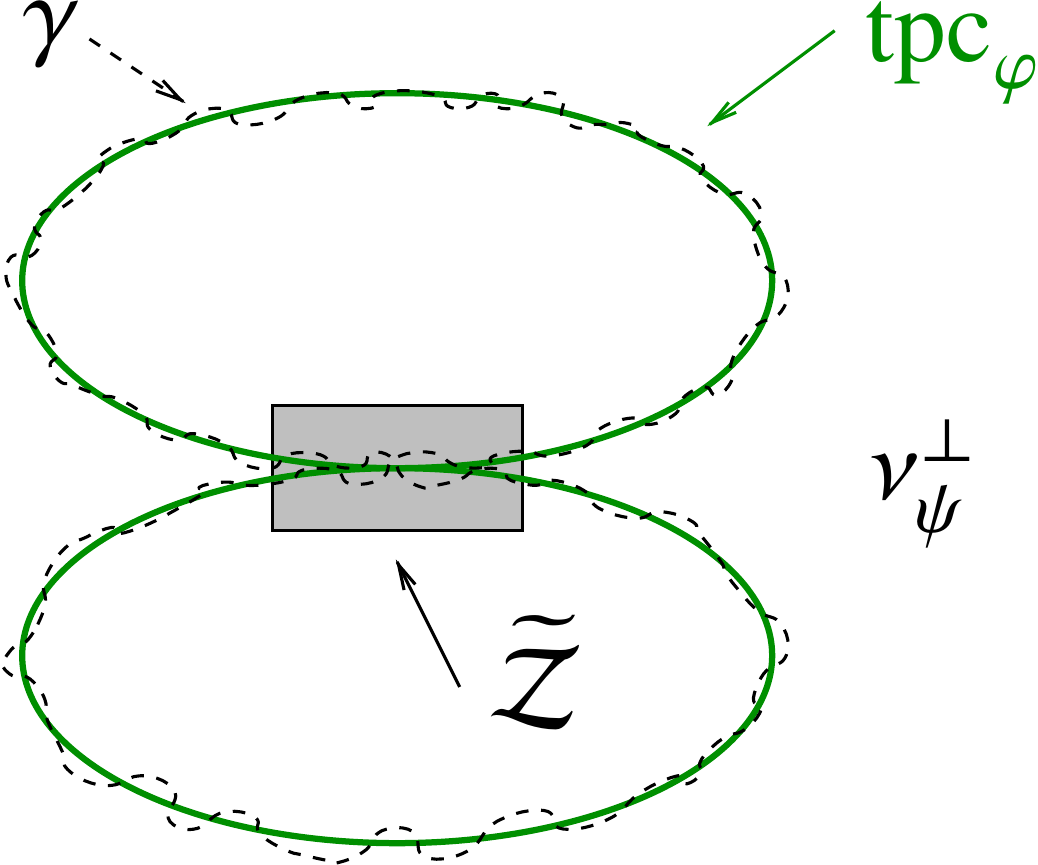}
 \caption{Projection onto $\nu_{\psi}^{\perp}$.}\label{fig:ellipses}
 \end{figure}

 Now we consider the projection $\P$, see Figure~\ref{fig:ellipses}.
 By construction there is (in the projection) one ellipse-shaped component
 on both sides of the line $\R\mathbf e_{2}$.
 As shown in~\eqref{eq:axis-angle},
 the angle between $\tge_{\varphi}'$ and $\mathbf e_{2}$
 is bounded by~$\tfrac\pi{10}$ on $|\xi|\le \tfrac1{16\sqrt{3}\pi}$.

 Proceeding as in~\eqref{eq:arccos}
 and using~\eqref{eq:g-phi}, we arrive at
 \begin{align*}
 \angle\br{\tg_{j}',\widetilde{\tge}_{\varphi,j}'}
 &\le \arccos\frac{\sp{\tg_{j}',\widetilde{\tge}_{\varphi,j}'}}{\abs{\tg_{j}'}\abs{\widetilde{\tge}_{\varphi,j}'}}
 \le \arccos\frac{\abs{\widetilde{\tge}_{\varphi,j}'}-\eps}{\abs{\widetilde{\tge}_{\varphi,j}'}+\eps}
 \;\refeq{arccos-x}\le\sqrt{\frac{2\pi\eps}{\abs{\widetilde{\tge}_{\varphi,j}'}+\eps}} \\
 &\le\sqrt{2\pi\eps}\;\refeq{epsilon}
 \le\sqrt{\tfrac1{10}\sqrt{\tfrac1{8\cdot96}}}
 < 0.1
 < \tfrac\pi{10}, \qquad j=1,2.
 \end{align*}
 Therefore the secant defined by two points of $\g$ 
 inside $\widetilde{\mathcal Z}$,
 either  both on $\tg_{1}$, or  both on $\tg_{2}$,
 always encloses with $\mathbf e_{2}$ an angle of at most $\tfrac\pi5$,
  and the same holds true for the projection onto $\nu_{\psi}^{\perp}$.
 
 Now we pass to the cylinder $\mathcal Z$
 corresponding to $\zeta = \nfrac{\tilde{\zeta}}2$.
 Note that the distance of $\partial\mathcal Z$ to $\partial\widetilde{\mathcal Z}$
 is bounded below by $\min\br{\zeta,\br{\sqrt2-1}\sqrt{\nfrac\zeta{8\pi}}}
 \ge\zeta$
 and that
 \begin{equation}\label{eq:zeta-thick}
  \zeta\ge\triangle[\g]
 \end{equation}
 for otherwise the strands of $\g$ would not fit into $\mathcal Z$
  which is guaranteed by Lemma~\ref{lem:zylinder}.
 
 Applying Proposition~\ref{prop:braids} (using Sard's theorem) and assuming $b\ge3$ (the case $b\le-3$
 being symmetric;  recall that $b=\pm1$ leads to the unknot)
 there are for a.e.\@ $\psi\in[\nfrac\varphi4,\nfrac{3\varphi}4]$ points
 \[ -\eta<\xi_{A}<\xi_{E}<\xi_{B}<\xi_{F}<\xi_{C}<\eta \]
 such that
 $a_{\xi_X}$ is parallel to $\nu_{\psi}$ (thus $\P a_{\xi_X}=0$)
 for $X\in\set{A,B,C}$ while $a_{\xi_X}$ is perpendicular
 to $\nu_{\psi}$ for $X\in\set{E,F}$ and
 \[ \beta(\xi_{A}) + \pi = \beta(\xi_{B}) = \beta(\xi_{C}) - \pi. \]
 We claim
 \begin{equation}\label{eq:thickness-estimate}
  \abs{a_{\xi_{E}}},\abs{a_{\xi_{F}}}\ge2\triangle[\g].
 \end{equation}
 To see this, consider the map $[-\eta,\eta]^{2}\to(0,1)$,
 $(\xi_{1},\xi_{2})\mapsto\abs{\tg_{1}(\xi_{1})-\tg_{2}(\xi_{2})}$.
 There is at least one global minimizer $(\tilde\xi_{1},\tilde\xi_{2})$,
 and consequently we have $\tg_{1}'(\tilde\xi_{1})
 \perp \tg_{1}(\tilde\xi_{1})-\tg_{2}(\tilde\xi_{2}) \perp \tg_{2}'(\tilde\xi_{2})$.
 For any $\tilde\eps>0$ we may choose some $\tilde\xi_{\tilde\eps}\in[-\eta,\eta]$
 close to $\tilde\xi_{1}$ such that the radius $\tilde\rho_{\tilde\eps}$ of the circle
 passing through $\tilde\g_{1}(\tilde\xi_{1})$,
 $\tilde\g_{2}(\tilde\xi_{2})$, $\tilde\g_{1}(\tilde\xi_{\tilde\eps})$
 satisfies
 $\abs{\tg_{1}(\tilde\xi_{1})-\tg_{2}(\tilde\xi_{2})}
 \ge2\tilde\rho_{\tilde\eps}-\tilde\eps\ge2\triangle[\g]-\tilde\eps$
 (cf.\@ Litherland et al.~\cite[Proof of Thm.~3]{lsdr}).
 
 We let
 \begin{align*}
 A&:=\P\tg_{1}(\xi_{A})=\P\tg_{2}(\xi_{A}), &
 E&:=\P\tg_{1}(\xi_{E}), &
 E'&:=\P\tg_{2}(\xi_{E}), \\
 B&:=\P\tg_{1}(\xi_{B})=\P\tg_{2}(\xi_{B}), &
 F&:=\P\tg_{1}(\xi_{F}), &
 F'&:=\P\tg_{2}(\xi_{F}), \\
 C&:=\P\tg_{1}(\xi_{C})=\P\tg_{2}(\xi_{C}),
 \end{align*}
 and denote the secant through $A$ and $C$ by $g$
 which defines two half planes, $G$ and $G'$.
 As shown before, this line encloses an angle of
 at most $\nfrac\pi5$ with the $\mathbf e_{2}$-axis.
 Therefore, the line orthogonal to and bisecting $\overline{AC}$,
 is spanned by $\Phi(\th_{0},\psi)\in\nu_{\psi}^{\perp}\cap\S^{2}$ for
 some $\th_{0}\in[\nfrac{3\pi}{10},\nfrac{7\pi}{10}]$ and meets $\P\g$
 outside $\P\mathcal Z$ in two points, $D\in G$, $D'\in G'$.

 We arrange the labels of these half planes so that $AEFCDAE'F'CD'A$ is a closed
 polygon inscribed in $\P\g$.
 We aim at showing that, for some inscribed (sub-) polygon~$p$,
 there is some sub-interval $J_{\psi}\subset(\tfrac\pi6,\tfrac{5\pi}6)$
 with $\abs{J_{\psi}}\ge\frac{\triangle[\g]}{10}$ and
 \[ 3\le\mu(p,\Phi(\th,\psi))<\infty \quad\text{for all }
 \th\in J_{\psi}. \]
 The definition of $\mu(\cdot,\cdot)$ implies
 that the corresponding polygon inscribed in the (non-projected)
 curve $\g_{*}$ also satisfies the latter estimate.
 We construct~$p$ as follows.
 We distinguish three cases which are depicted in Figure~\ref{fig:cases}.

 \begin{figure}\centering
 \includegraphics[scale=.4]{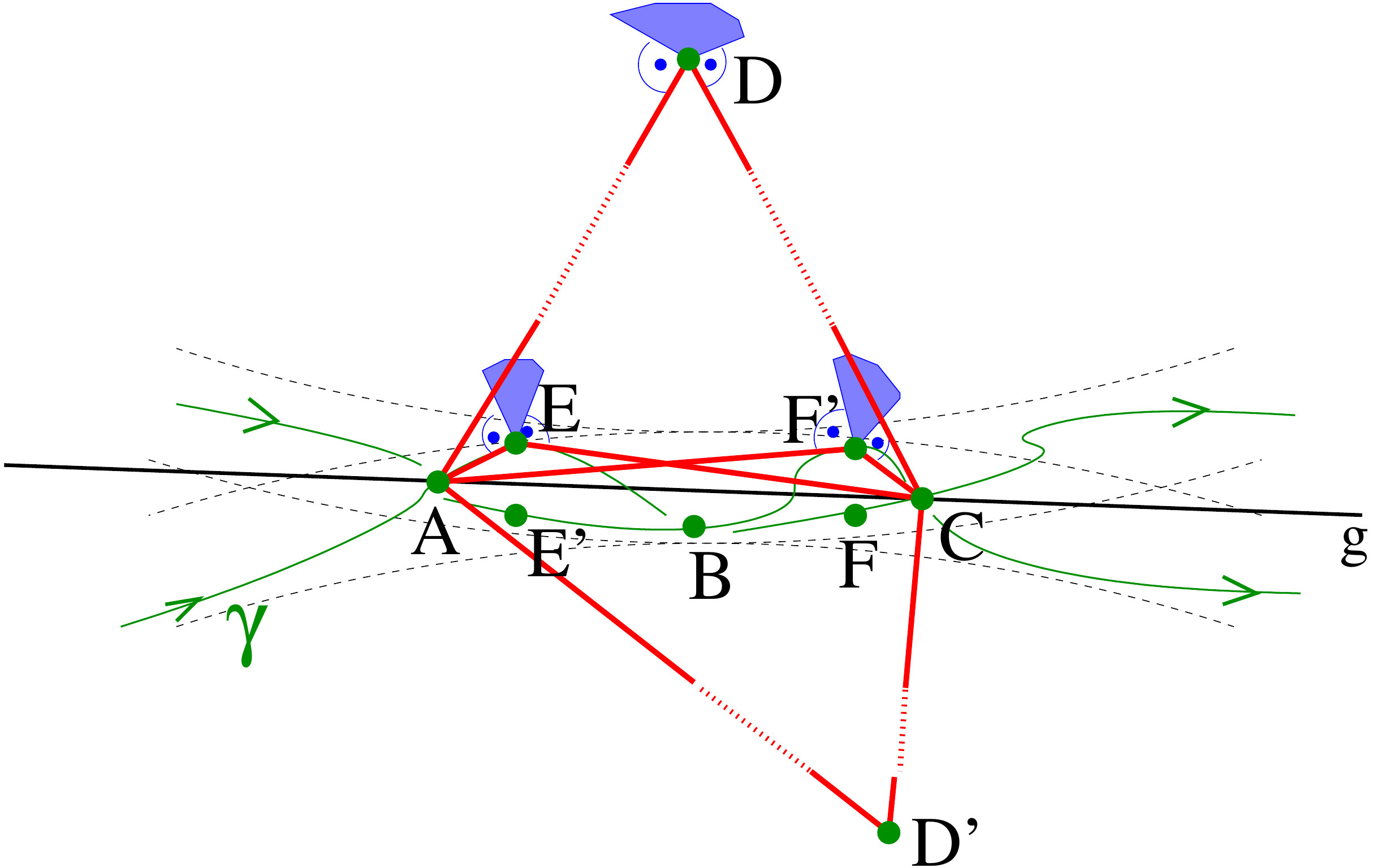}
 \includegraphics[scale=.4]{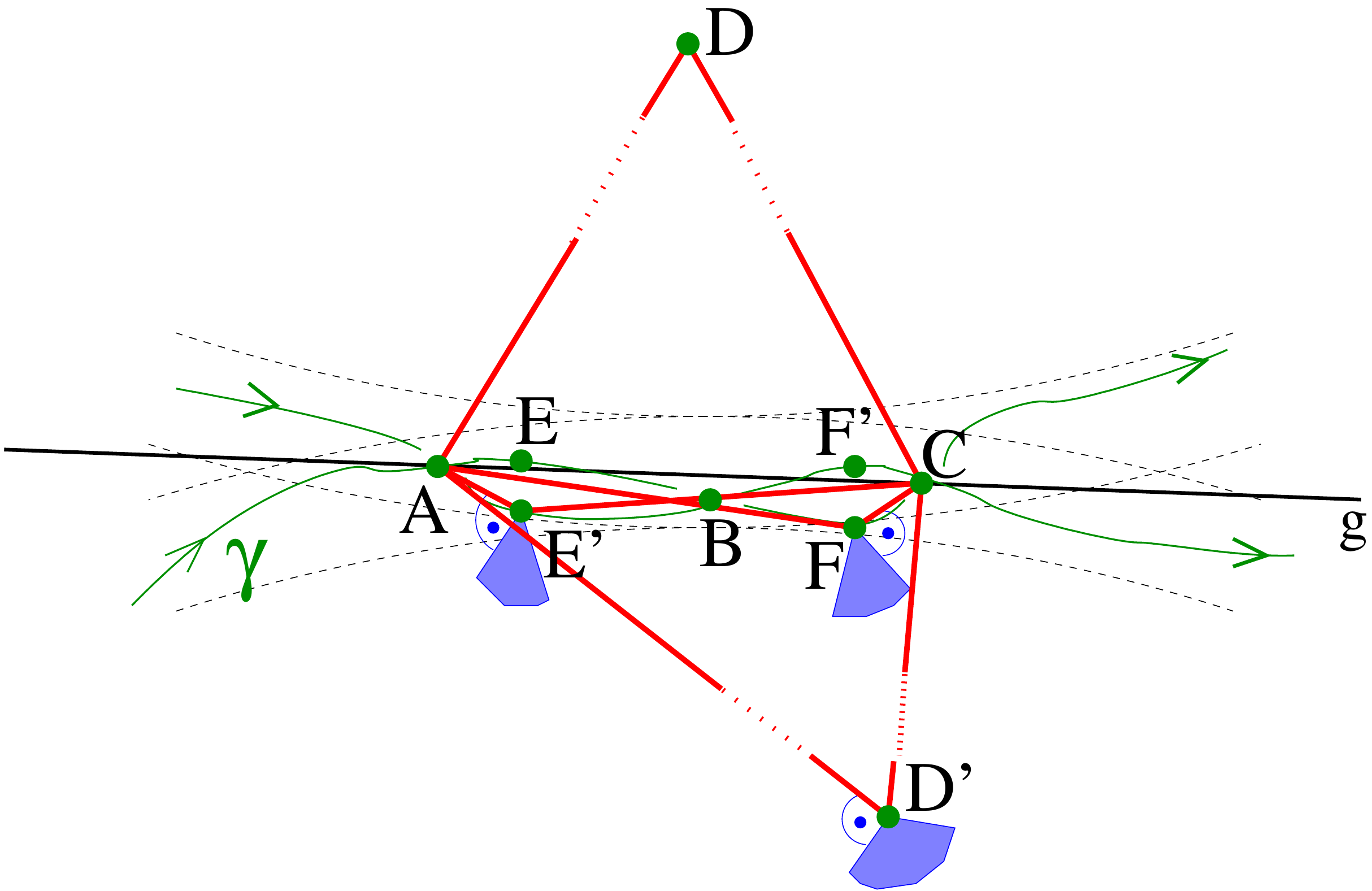}
 \includegraphics[scale=.4]{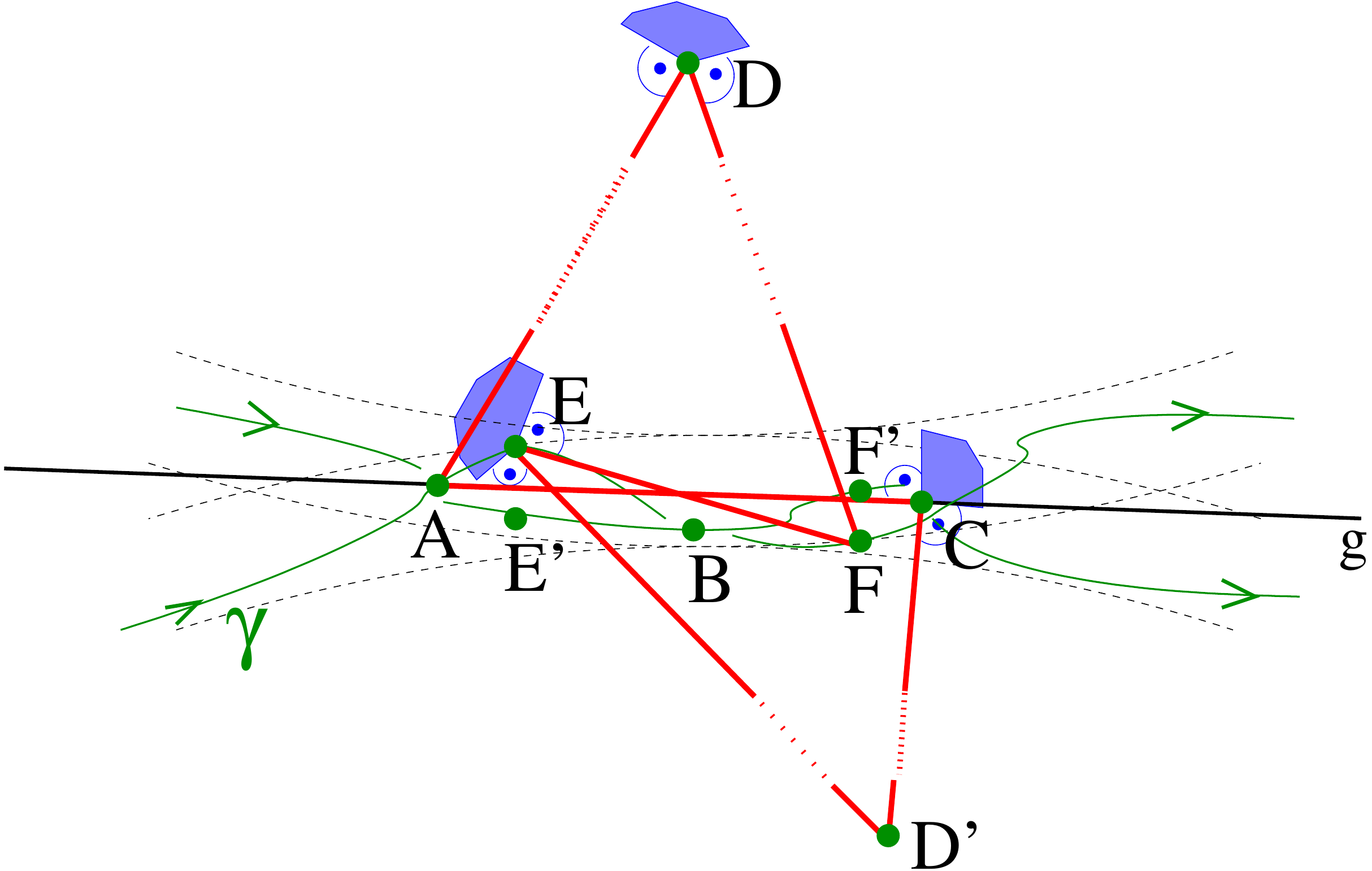}
 \caption{Three cases in the drawing plane $\nu_\psi^\perp$
 discussed in the proof of Lemma~\ref{lem:crook}. Intersecting the blue
 shaded regions leads to the angular region $J_\psi$ in each case.}\label{fig:cases}
 \end{figure}
 If
 \begin{equation}\label{eq:crookcase}
  \dist(E,g),\dist(F',g)\ge\tfrac{\triangle[\g]}2\qquad\text{and}\qquad E,F'\in G
 \end{equation}
 then, using~\eqref{eq:zeta-thick}, the polygon
 $AECDAF'CD'A$ has three local maxima at $E$, $D$, and $F'$
 when projected onto any $\Phi(\th,\psi)\in\S^{2}$ with
 $\abs{\th-\th_{0}}\le\arctan {\triangle[\g]}$.
As $\arctan x\ge\tfrac\pi4x$ for $x\in[0,1]$,
 the set $J_{\psi}:=(\th_{0}-\tfrac\pi4\triangle[\g],\th_{0}+\tfrac\pi4\triangle[\g])$ is contained in $(\tfrac\pi6,\tfrac{5\pi}6)$
 due to $\tfrac\pi4\triangle[\g]\refeq{zeta-thick}\le\tfrac\pi4\zeta
 \le\tfrac1{8\cdot96}<0.0014$
 and satisfies $\abs{J_{\psi}}=\tfrac\pi2\triangle[\g]$.
 
 Now assume that \emph{both} $E$ and $F'$ do not meet the (first) conditions in~\eqref{eq:crookcase}, so,  by~\eqref{eq:thickness-estimate},
 \[ \dist(F,g),\dist(E',g)\ge\tfrac{\triangle[\g]}2\qquad\text{and}\qquad E',F\in G'. \]
 But this is symmetric to~\eqref{eq:crookcase}
 since the polygon $AFCDAE'CD'A$
 has three local \emph{minima} at $E'$, $D'$, and $F$
 when projected onto any $\Phi(\th,\psi)\in\S^{2}$ with
 $\abs{\th-\th_{0}}\le\arctan\triangle[\g]$.
 Employing the same argument as for~\eqref{eq:crookcase},
 this yields the desired as the number of local maxima and minima agrees.
 
 Finally we have to deal with the ``mixed case''.
 Without loss of generality we may assume
 \[ \dist(E,g),\dist(F,g)\ge\tfrac{\triangle[\g]}2\qquad\text{and}\qquad E\in G,F\in G'. \]
 But now
 the polygon
 $EFDACD'E$ has three local maxima at $E$, $D$, and $C$
 when projected onto any $\Phi(\th,\psi)\in\S^{2}$ with
 $\abs{\th-\th_{0}}\le\arctan\triangle[\g]$
 and $\sp{\Phi(\th,\psi),C-A}>0$
 which completes the proof with $\abs{J_{\psi}}=\tfrac\pi4\triangle[\g]$.
\end{proof}

\begin{proof}[Theorem~\ref{thm:main}]
Let us now assume that $\tge_{\varphi}$, $\varphi\in (0,\pi]$,
is the $C^1$-limit  of $\Eth$-minimizers ~$\br{\gth}_{\th>0}$ as 
$\th\to 0$.
In the light of Proposition~\ref{prop:estimate-torus} for $a=2$,
\ref{prop:crook} and Lemma~\ref{lem:crook} we arrive at
\begin{align*}
 C\th^{2/3}\;&\refeq{energy-growth-rate}{\ge}\; \Eb(\gth)-(4\pi)^{2} = \int_{\R/\Z}\abs{\ddgth}^{2}-(4\pi)^{2}
 \ge \br{\int_{\R/\Z}\abs{\ddgth}}^{2}-(4\pi)^{2} \\
 &=\TC(\gth)^{2}-(4\pi)^{2}\ge 4\pi\br{\TC(\gth)-4\pi} = 2
 \pi\int_{\S^{2}}\br{\mu\br{\gth,\nu}-2}\d A(\nu) \\
 &\ge 2\pi\area\br{\mathcal B(\gth)}\ge \frac {\varphi\pi^2}{8\Er(\gth)} \overset{\eqref{eq:ropelength-growth-rate}}{\ge} c\th^{1/3}
\end{align*}
for some positive constant $c$ depending on $\varphi$
and $\th>0$ sufficiently small,
which is a contradiction as~$\th\searrow0$. Thus we have proven that
the limit curve $\g_0\in\scl[\tkc]$ is isometric to the twice covered circle~$\tge_0$.
\end{proof}

Our Main Theorem now reads as follows.

\begin{theorem}[Two bridge torus knot classes]\label{thm:elastic-shapes}
For any knot class $\K$ the following statements are equivalent.
\begin{enumerate}
\item\label{item:class} $\K$ is the $(2,b)$-torus knot class for some odd integer $b\in\Z$ where $\abs b\ge3$;
\item\label{item:inf}
  $\inf_{\cC(\K)}\Eb=(4\pi)^2$;
\item\label{item:C1} $\tge_{\varphi}$ belongs to the $C^{1}$-closure of
$\cC(\K)$ for some $\varphi\in[0,\pi]$ and $\K$ is not trivial;
\item\label{item:some-elastic} for some $\varphi\in[0,\pi]$
the pair of tangentially intersecting circles $\tge_{\varphi}$
is an elastic knot for $\cC(\K)$;
\item\label{item:0-elastic} the unique elastic knot for $\cC(\K)$ is the twice covered circle $\tge_{0}$.
\end{enumerate}
\end{theorem}

\begin{proof}[Theorem~\ref{thm:elastic-shapes}]
\ref{item:class} $\Rightarrow$ \ref{item:inf} follows from Theorem~\ref{thm:fm}
and the estimate in Lemma~\ref{lem:tau-estimate}
for the comparison torus curves defined in~\eqref{eq:torus-knot}.

\ref{item:inf} $\Rightarrow$ \ref{item:C1}: As the bending energy of
the circle amounts to $(2\pi)^{2}$, the knot class cannot be trivial.
From Proposition~\ref{prop:constcurv} we infer that any elastic knot
(which exists and belongs to the $C^1$-closure of $\cC(\K)$ according 
to Theorem~\ref{thm:existence-limit}) has constant
curvature $4\pi$ a.e.
Proposition \ref{prop:nonembedded} guarantees that any such elastic
knot must have double points, which
permits to apply
Corollary~\ref{cor:tg8}.
Thus, such an  elastic knot coincides (up to isometry)
with a tangential pair of circles $\tge_\varphi$ for some $\varphi\in [0,\pi]$, 
which
therefore lies in the $C^1$-closure of $\cC(\K)$ as desired.

\ref{item:C1} $\Rightarrow$ \ref{item:class} is an immediate consequence of
Proposition~\ref{prop:braids} and Corollary~\ref{cor:braids}. 

Hence, so far we have shown the equivalence of the first three items.

\begin{figure}\centering
 \includegraphics[scale=.5]{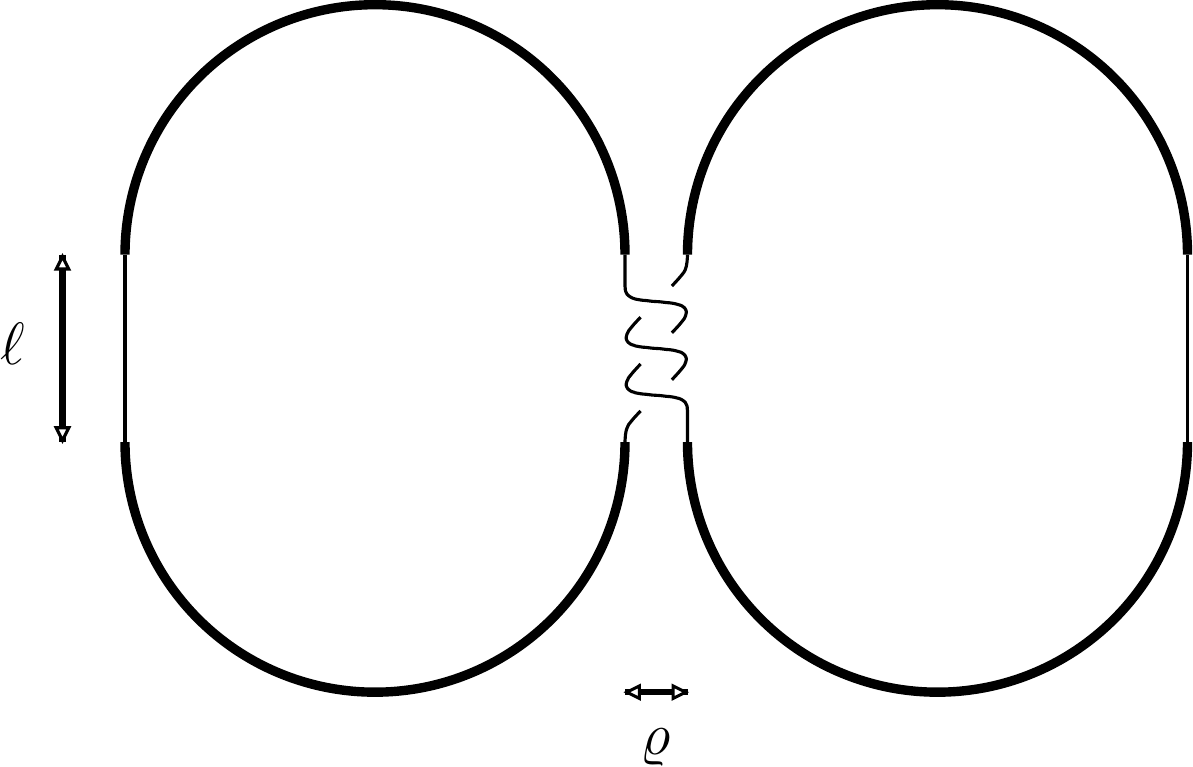}
 \caption{Explicit construction proving that $\tge_{\varphi}$ is in the
 strong $C^k$-closure of $\cC(\K)$
 for any $\varphi\in(0,\pi]$ and $k\in\N$.}\label{fig:hinge}
\end{figure}

\ref{item:class} $\Rightarrow$ \ref{item:0-elastic}
is just the assertion of Theorem~\ref{thm:main}.

\ref{item:0-elastic} $\Rightarrow$ \ref{item:some-elastic} is immediate.

\ref{item:some-elastic} $\Rightarrow$ \ref{item:class}:
Since elastic knots for $\cC(\K)$ lie in the $C^1$-closure of $\cC(\K)$
by Theorem \ref{thm:existence-limit}
we find by Proposition~\ref{prop:braids} for $\varphi\in (0,\pi]$,  and by
Corollary~\ref{cor:braids} for $\varphi =0$ that $\K=\cT(2,b)$
for some odd $b\in\Z$, $\abs b\ge3$, or that $\K$
is trivial. The latter can be excluded since the unique elastic unknot is
the once-covered circle by
Proposition~\ref{prop:nonembedded}.
\end{proof}
\begin{remark}[Strong $C^k$-closure]
Using an explicit construction like the one
indicated in Figure \ref{fig:hinge} one
can show as well that each item above is also equivalent to the following:

\emph{For every  $\varphi\in[0,\pi]$
the corresponding tangential pair of circles $\tge_{\varphi}$ belongs to the $C^{k}$-closure of $\cC(\K)$ for any  $k\in\N$, and $\K$ is not trivial.}
\end{remark}

\begin{appendix}
% !TEX root = regelast.tex

\section{The F\'ary--Milnor theorem for the $C^1$-closure of the knot class}

\begin{theorem}[Extending the F\'ary--Milnor theorem to the $C^1$-closure of $\mathcal C(\K)$]\label{thm:fm}
 Let $\K$ be a non-trivial knot class and let
 $\g$ belong to the closure of $\cC(\K)$ with respect to the $C^1$-norm.
 Then  the total curvature~\eqref{eq:tc} satisfies
 \begin{equation}\label{eq:fm}
  T(\g)\ge4\pi.
 \end{equation}
\end{theorem}

We begin with two auxiliary tools and abbreviate,
for any two vectors $X,Y\in\R^{3}$,
\[ \overrightarrow{XY} := Y-X. \]

\begin{lemma}[Approximating tangents]\label{lem:apprtang}
Let 
 $\seqn{\g}\subset C^1(\R/\Z,\R^3)$ be a sequence of 
 embedded arclength parametrized closed curves
 converging with respect to the $C^1$-norm to some limit curve~$\g $.
 Assume that there are  sequences 
 $\seqn x,\seqn y\subset [0,1)$ of parameters
 satisfying $x_k<y_k$ for all $k\in\N$ and
 $x_k,y_k\to z\in [0,1)$ as $k\to\infty$. Then one has for
  $X_k:=\g_k(x_k)$, $Y_k:=\g_k(y_k)$ 
 \begin{equation}\label{eq:direconv}
  \uvector{X_kY_k}=\frac{Y_k-X_k}{|Y_k-X_k|}
  \xrightarrow{k\to\infty} \g'(z)\in\mathbb S^2.
 \end{equation}
\end{lemma}

\begin{proof}

Since all $\g_k$ are injective we find that $X_k\not= Y_k$ for $k\in\N$, so
that the unit vectors $\uvector{X_kY_k}$ are well-defined, and a subsequence
(still denoted by $\uvector{X_kY_k}$) converges to  some limit unit vector
$d\in\S^2$ because $\S^2$ is compact. It suffices to show that $d=\g'(z)$
to conclude  that the \emph{whole} sequence
converges to $\g'(z),$ since any other subsequence of the $\uvector{X_kY_k}$
with limit $\tilde{d}\in\S^2$ satisfies $\tilde{d}=\g'(z)$ as well. 

 We compute
 \begin{align}\label{deviation}
 \begin{split}
  &{\uvector{X_kY_k}-\frac{\g(y_k)-\g(x_k)}{y_k-x_k}}
  ={\frac{\g_k(y_k)-\g_k(x_k)}{\abs{\g_k(y_k)-\g_k(x_k)}}-\frac{\g(y_k)-\g(x_k)}{y_k-x_k}} \\
  &=\frac{\int_0^1\dg_k(x_k+\theta(y_k-x_k))\d\theta}{\abs{\int_0^1\dg_k(x_k+\theta(y_k-x_k))\d\theta}}-\int_0^1\dg(x_k+\theta(y_k-x_k))\d\theta \\
  &=\br{\frac1{\abs{\int_0^1\dg_k(x_k+\theta(y_k-x_k))\d\theta}}-1}{\int_0^1\dg_k(x_k+\theta(y_k-x_k))\d\theta} + {} \\
  &\qquad{}+\int_0^1(\dg_k-\dg)(x_k+\theta(y_k-x_k))\d\theta \\
  &=\frac{1-\abs{\int_0^1\dg_k(x_k+\theta(y_k-x_k))\d\theta}}{\abs{\int_0^1\dg_k(x_k+\theta(y_k-x_k))\d\theta}}{\int_0^1\dg_k(x_k+\theta(y_k-x_k))\d\theta} + {} \\
  &\qquad{}+\int_0^1(\dg_k-\dg)(x_k+\theta(y_k-x_k))\d\theta.
 \end{split}
 \end{align}
 From the assumptions we infer
 
\begin{align*} 
 1&\ge\abs{\int_0^1\dg_k(x_k+\theta(y_k-x_k))\d\theta} \\
 &\ge
 1-2\|\g_k'-\g'\|_{L^\infty}-\abs{\int_0^1(\g'(x_k+\theta(y_k-x_k))-\g'(x_k))
 \d\theta}\\
 &\ge  1-o(1)\quad\textnormal{as $k\to\infty$,}
 \end{align*}
 so that in particular $\big|\int_0^1\dg_k(x_k+\theta(y_k-x_k))\d\theta\big|\ge 1/2$
 for $k\gg 1$. Inserting these two facts into \eqref{deviation}
 yields
 \begin{align*}
  &\abs{\uvector{X_kY_k}-\frac{\g(y_k)-\g(x_k)}{y_k-x_k}}
  \le 2\abs{1-\Big|\int_0^1\dg_k(x_k+\theta(y_k-x_k))\d\theta\Big|} + \lnorm[\infty]{\dg_k-\dg}=o(1)
 \end{align*}
 as $k\to\infty.$
\end{proof}

\begin{lemma}[Regular curves cannot stop]\label{lem:stop}\hfill\\
 Let $\g\in C^1\rzd$ be regular, i.e., $|\g'|>0$ on $\R/\Z$, and let
 $0\le x<y<1$ such that $\g(x)=\g(y)$.
 Then there is some $z\in(x,y)$ such that $\g(x)\ne\g(z)\ne\g(y)$.
\end{lemma}

\begin{proof}
 Assuming the contrary leads to a situation where $\g$ is constant on an interval of positive measure,
 i.e.\@ $\dg$ vanishes contradicting $\abs\dg >0$.
\end{proof}

\begin{proof}[Theorem~\ref{thm:fm}]
 Let $\seqn{\g}$ be a sequence of knots in~$\cC(\K) $ converging with respect to the $C^1$-norm to some limit curve~$\g$.
 By Denne's theorem~\cite[Main Theorem, p. 6]{denne}, 
 each knot $\g_k$, $k\in\N$, has an alternating quadrisecant, 
 i.e., there are numbers
 $0\le a_k<b_k<c_k<d_k<1$ such that the points $A_k:=\g_k(a_k)$, $B_k:=\g_k(b_k)$, $C_k:=\g_k(c_k)$, $D_k:=\g_k(d_k)$ are collinear
 and appear in the order $A$, $C$, $B$, $D$ on the quadrisecant line, see Figure~\ref{fig:fm1}.
 (Note that our labeling of these points differs from that in~\cite{denne}.)
 \begin{figure}
  \includegraphics[scale=.5]{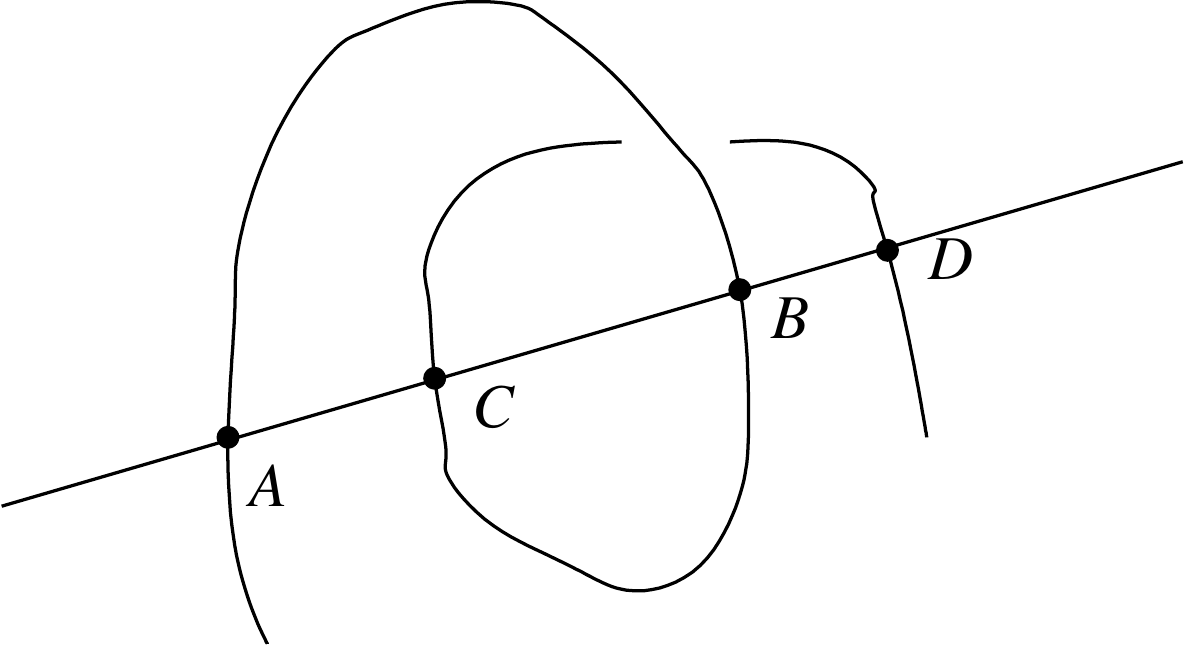}
  \centering
  \caption{An alternating quadrisecant}
  \label{fig:fm1}
 \end{figure}
 Without loss of generality we may assume $a_k\equiv0$. By compactness, we may pass to a subsequence (without relabelling)
 such that $(a_k,b_k,c_k,d_k)$ converges to $(a,b,c,d)$ as $k\to\infty$. Of course,
 \begin{equation*}
  0=a\le b\le c\le d\le 1.
 \end{equation*}
 Note that some or all of the corresponding points $A,B,C,D\in\R^3$, which are still collinear, may coincide and that $1\equiv0$ in $\R/\Z$.
 By Milnor~\cite[Theorem 2.2]{milnor} the total curvature of~$\g$ is bounded from below by the total curvature of any 
inscribed (closed) polygon.
 For each possible location of the points $A,B,C,D$ we estimate the total curvature of $\gamma$ by means
of suitably chosen inscribed polygons.
 \renewcommand{\labelenumi}{\arabic{enumi}.}
 \renewcommand{\labelenumii}{\arabic{enumi}.\arabic{enumii}.}
 \renewcommand{\labelenumiii}{\arabic{enumi}.\arabic{enumii}.\arabic{enumiii}.}
 \renewcommand{\labelenumiv}{\arabic{enumi}.\arabic{enumii}.\arabic{enumiii}.\arabic{enumiv}.}
 \renewcommand{\theenumi}{\arabic{enumi}}
 \renewcommand{\theenumii}{\arabic{enumii}}
 \renewcommand{\theenumiii}{\arabic{enumiii}}
 \renewcommand{\theenumiv}{\arabic{enumiv}}
 \makeatletter\renewcommand\p@enumiii{\theenumi.\theenumii.}\makeatother
 \begin{enumerate}[font=\bfseries,
align=left, leftmargin=0pt, labelindent=\parindent,
listparindent=\parindent, labelwidth=0pt, itemindent=1 ex]
  \item If $\abs{BC}>0$ the polygon $ABCDA$ inscribed in $\gamma$ is non-degenerate in the sense that $0=a<b<c<d<1$.
  All exterior angles angles of $ABCDA$ equal $\pi$ and sum up to~$4\pi$; hence 
 \eqref{eq:fm} holds.
  \item $\abs{BC}=0$, i.e.\@ $B=C$
  \begin{enumerate}[font=\bfseries,
align=left, leftmargin=0pt, labelindent=\parindent,
listparindent=\parindent, labelwidth=0pt, itemindent=1 ex]
   \item $\min\br{\abs{AC},\abs{BD}}>0$
   \begin{enumerate}[font=\bfseries,
align=left, leftmargin=0pt, labelindent=\parindent,
listparindent=\parindent, labelwidth=0pt, itemindent=1 ex]
    \item\label{item:approx} If $b=c$, i.e.\@ there is no loop of $\g$ between~$B$ and~$C$, we may consider the polygon
    $AB_\eps C_\eps DA$ with $B_\eps:=\g(b-\eps)$ and $C_\eps:=\g(c+\eps)$.
    Note that, in general, the point $B=C$ does not belong to the polygon, see Figure~\ref{fig:fm2}.
    By arc-length parametrization, $B_\eps$ and $C_\eps$ cannot coincide for $0<\eps\ll1$.
    \begin{figure}
     \centering
     \includegraphics[scale=.5]{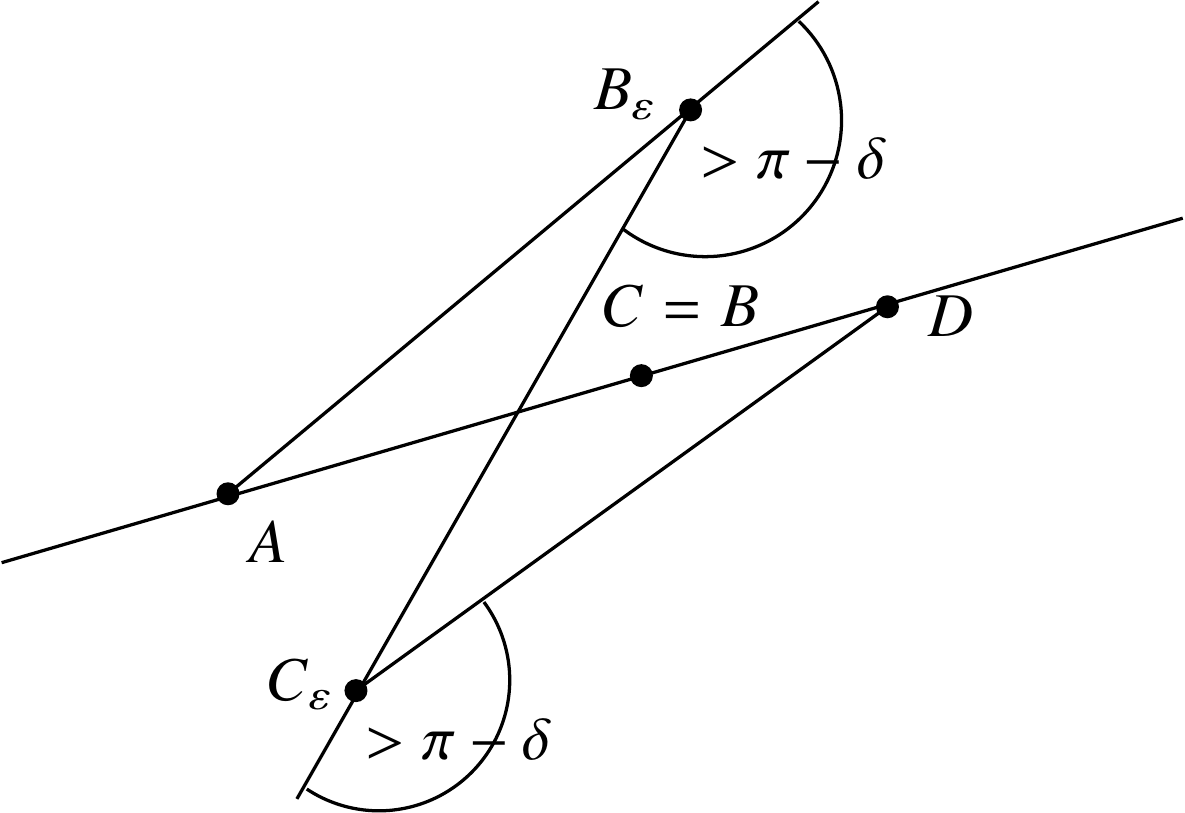}%
     \includegraphics[scale=.5]{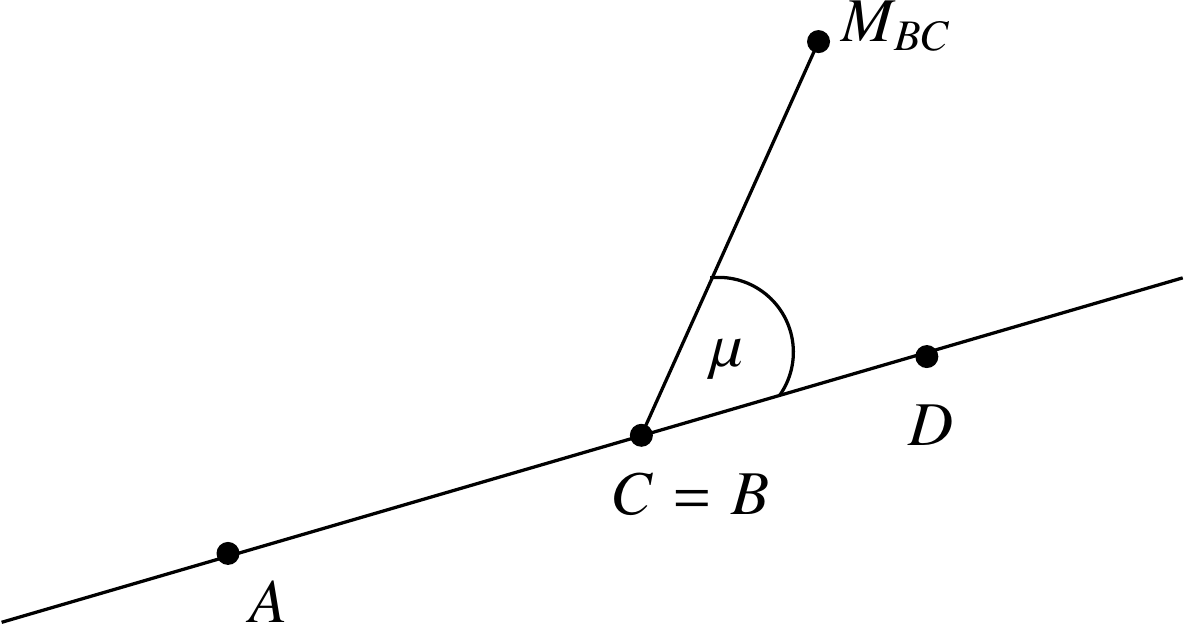}
     \caption{(left) Situation~\ref{item:approx},
     (right) Situation~\ref{item:auxpoint}.}
     \label{fig:fm2}
     \label{fig:fm3}
    \end{figure}
    We deduce that $\uvector{B_\eps C_\eps}$
    converges to $\dg(b)$ as $\eps\to 0$. From Lemma~\ref{lem:apprtang} we infer that $\dg(b)$ points towards~$A$
    (due to $\uvector{B_k A_k}=\uvector{B_k C_k}\to\gamma'(b)$ and the order $A_kC_kB_kD_k$ of the approximating points on the respective quadrisecant).
    Therefore, since $B_\eps\to B$ and $C_\eps\to C$ as $\eps\to 0$, for given $\delta>0$ we obtain some $\eps_\delta>0$ such that
    $\angle\br{\overrightarrow{AB_\eps},\overrightarrow{B_\eps C_\eps}}>\pi-\delta$, 
    $\angle\br{\overrightarrow{B_\eps C_\eps},\overrightarrow{C_\eps D}}>\pi-\delta$,
$\angle\br{\overrightarrow{C_\eps D},\overrightarrow{D A}}>\pi-\delta$, and
$\angle\br{\overrightarrow{D A},\overrightarrow{A B_\eps}}>\pi-\delta$
    for all $\eps\in(0,\eps_\delta]$. We arrive at a lower estimate of $4\pi-4\delta$ for the total curvature of the polygon $AB_\eps C_\eps DA$
    which is a lower bound for the total curvature of~$\g$. Letting $\delta\searrow0$ yields the
    desired.
    \item\label{item:auxpoint} If $b<c$ there is a loop of $\g$ between~$B$ and~$C$ according to 
Lemma \ref{lem:stop} such that we may choose some $m_{bc}\in(b,c)$ with
    $B=C\ne M_{BC}:=\g(m_{bc})$. The total curvature of $\g$ is bounded below by the total curvature of the polygon $ABM_{BC}CDA$.
    Let $\mu:=\angle\br{\overrightarrow{AB},\overrightarrow{BM_{BC}}}$.
    We obtain
    \begin{align*}
     \angle\br{\overrightarrow{AB},\overrightarrow{BM_{BC}}} &= \mu, &
     \angle\br{\overrightarrow{BM_{BC}},\overrightarrow{M_{BC}C}} &= \pi, &
     \angle\br{\overrightarrow{M_{BC}C},\overrightarrow{CD}} &= \pi-\mu, \\
     \angle\br{\overrightarrow{CD},\overrightarrow{DA}} &= \pi, &
     \angle\br{\overrightarrow{DA},\overrightarrow{AB}} &= \pi.
    \end{align*}
   \end{enumerate}
   \item $\abs{AC}=0$, $\abs{BD}>0$, i.e.\@ $A=B=C\ne D$
    \begin{enumerate}[font=\bfseries,
align=left, leftmargin=0pt, labelindent=\parindent,
listparindent=\parindent, labelwidth=0pt, itemindent=1 ex]
     \item\label{item:contrad}The situation $a=b=c$ cannot occur as applying Lemma~\ref{lem:apprtang} twice
      (and recalling the order $A_kC_kB_kD_k$ of the approximating points on the respective quadrisecant) yields $\dg(b)=-\dg(b)$
      contradicting $\abs{\dg(b)}=1$.
     \item\label{item:approx-auxpoint} If $a=b<c$ we may simultaneously apply the techniques from~\ref{item:approx}
      and~\ref{item:auxpoint} considering the polygon $A_\eps BB_\eps M_{BC}CDA_\eps$ with
      $A_\eps:=\g(a-\eps)$ and $B_\eps:=\g(b+\eps)$.
      Note that, in contrast to~\ref{item:approx}, we inserted the point $B=C=A$ between $A_\eps$ and $B_\eps$
      which is admissible due to $a-\eps<a=b<b+\eps$.
      We obtain the triangles $BB_\eps M_{BC}$ and $CDA_\eps$. We label the \emph{inner} angles (starting from $B=C$
in each of the two triangles)
      by $\zeta_\eps$, $\eta_\eps$, $\theta_\eps$ and $\lambda_\eps$, $\mu_\eps$, $\nu_\eps$ and define
      \begin{align*}
       \sigma_\eps &:= \angle\br{\overrightarrow{A_\eps B},\overrightarrow{BB_\eps}}, &
       \tau &:= \angle\br{\overrightarrow{M_{BC}C},\overrightarrow{CD}},
      \end{align*}
      see Figure~\ref{fig:fm4}.
      \begin{figure}
       \centering
       \includegraphics[scale=.5]{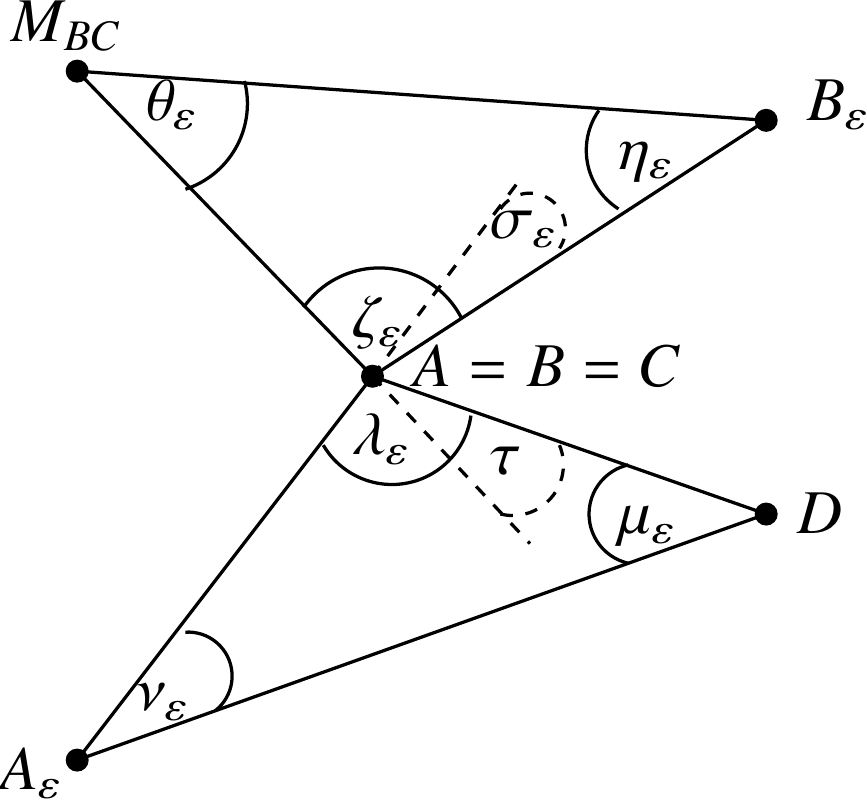}\qquad
       \includegraphics[scale=.5]{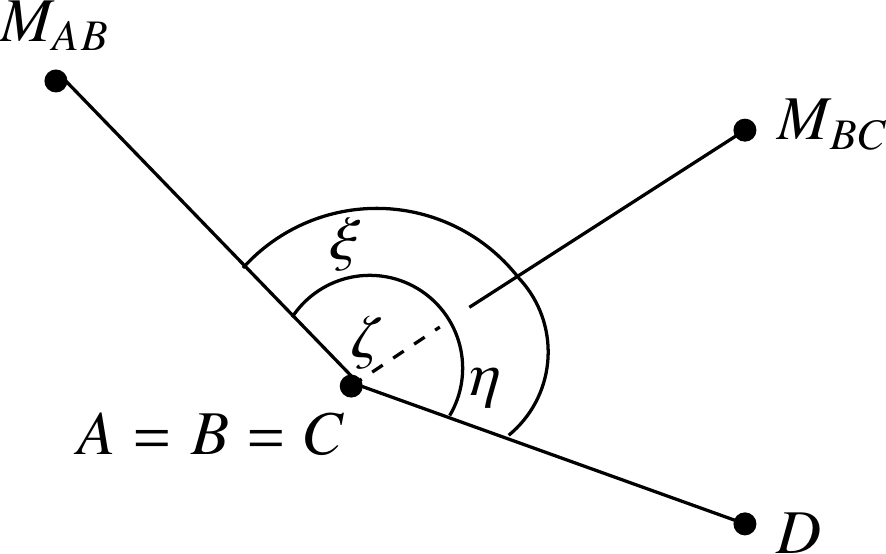}
       \caption{(left) Situation~\ref{item:approx-auxpoint}. Note that, in general, the triangles $BB_\eps M_{BC}$ and $CDA_\eps$
       are not coplanar and the angles $\sigma_\eps$ and $\tau$ do not belong to any of the planes defined by the triangles.
       (right) Situations~\ref{item:spherical} and~\ref{item:3loops}. For the latter one has to replace~$D$ by~$M_{DA}$.}
       \label{fig:fm4}
       \label{fig:fm5}
      \end{figure}
      Therefore, the total curvature of the polygon amounts to the sum of \emph{exterior} angles
      \begin{equation*}
       \sigma_\eps + \br{\pi-\eta_\eps} + \br{\pi-\theta_\eps} + \tau + \br{\pi-\mu_\eps} + \br{\pi-\nu_\eps}
       = \sigma_\eps + \tau + 2\pi + \zeta_\eps + \lambda_\eps.
      \end{equation*}
      Since $\uvector{A_\eps B}$ and $\uvector{BB_\eps}$ \/\/ approximate $\dg(b)$,
      which is a positive multiple of $\overrightarrow{CD}$ (because $c<d$ for $C\not= D$ which implies
$\uvector{C_k D_k}=\uvector{A_k C_k}\to\gamma'(b)=\uvector{C D}=\uvector{B D}$ by Lemma \ref{lem:apprtang}),
      we deduce $\sigma_\eps\to 0$, $\zeta_\eps\to\pi-\tau$ and $\lambda_\eps\nearrow\pi$ as $\eps\searrow0$
      from Lemma~\ref{lem:apprtang}.
     \item For the case $a<b=c$ we consider the polygon $AM_{AB}B_\eps CC_\eps DA$ for $B_\eps:=\gamma(b-\eps)$ and
$C_\eps:=\gamma(c+\eps)$
      which may be treated similarly to~\ref{item:approx-auxpoint}.
     \item\label{item:spherical} If $a<b<c$ we consider the polygon $AM_{AB}BM_{BC}CDA$ applying the technique
      from~\ref{item:auxpoint} twice. Defining
      \begin{align*}
       \xi&:=\angle\br{\overrightarrow{BM_{AB}},\overrightarrow{BM_{BC}}}, &
       \eta&:=\angle\br{\overrightarrow{CM_{BC}},\overrightarrow{CD}}, &
       \zeta&:=\angle\br{\overrightarrow{AD},\overrightarrow{AM_{AB}}}
      \end{align*}
      as indicated in Figure~\ref{fig:fm5} we arrive at
      \begin{align*}
       \angle\br{\overrightarrow{AM_{AB}},\overrightarrow{M_{AB}B}} &= \pi, &
       \angle\br{\overrightarrow{M_{AB}B},\overrightarrow{BM_{BC}}} &= \pi-\xi, &
       \angle\br{\overrightarrow{BM_{BC}},\overrightarrow{M_{BC}C}} &= \pi, \\
       \angle\br{\overrightarrow{M_{BC}C},\overrightarrow{CD}} &= \pi-\eta, &
       \angle\br{\overrightarrow{CD},\overrightarrow{DA}} &= \pi, &
       \angle\br{\overrightarrow{DA},\overrightarrow{AM_{AB}}} &= \pi-\zeta
      \end{align*}
      so the total curvature of the polygon amounts to~$6\pi-\xi-\eta-\zeta$.
      Consider the unit vectors $\uvector{AM_{AB}}$, $\uvector{AM_{BC}}$, and $\uvector{AD}$.
      Unless they are coplanar they define a unique triangle on~$\mathbb S^2$ of area $<2\pi$ which is equal to the
      sum of the three angles between them, namely $\xi$, $\eta$, and $\zeta$.
      Therefore, we obtain the estimate $\xi+\eta+\zeta\le 2\pi$.
    \end{enumerate}
   \item $\abs{AC}>0$, $\abs{BD}=0$, i.e.\@ $A\ne B=C=D$. This case is symmetric to the preceding one.
   \item $\abs{AC}=\abs{BD}=0$, i.e.\@ $A=B=C=D$.
    For $\Box_1,\Box_2,\Box_3,\Box_4\in\set{{=},{<}}$ we abbreviate
    \begin{equation*}
     (\Box_1,\Box_2,\Box_3,\Box_4) := \br{0\equiv a\Box_1b\Box_2c\Box_3d\Box_41\equiv0}.
    \end{equation*}
    Obviously there are $16$ cases.
   \begin{enumerate}[font=\bfseries,
align=left, leftmargin=0pt, labelindent=\parindent,
listparindent=\parindent, labelwidth=0pt, itemindent=1 ex]
    \item Impossible cases: as shown in~\ref{item:contrad}, there cannot arise two neighboring equality signs.
     This excludes the following nine situations:
     $({=},{=},{=},{=})$,
     $({=},{=},{=},{<})$,
     $({=},{=},{<},{=})$,
     $({=},{=},{<},{<})$,
     $({=},{<},{=},{=})$,
     $({=},{<},{<},{=})$,
     $({<},{=},{=},{=})$,
     $({<},{=},{=},{<})$,
     $({<},{<},{=},{=})$.
    \item\label{item:2loops} Two loops: for the situation $({=},{<},{=},{<})$ we consider the polygon
     \begin{equation*}
      A_\eps BB_\eps M_{BC}C_\eps DD_\eps M_{DA}A_\eps
     \end{equation*}
     with $A_\eps:=\g(a-\eps)$, $B_\eps:=\g(b+\eps)$, $C_\eps:=\g(c-\eps)$, and $D_\eps:=\g(d+\eps)$.
     We obtain the (in general non-planar) quadrilaterals $BB_\eps M_{BC}C_\eps$
     and $DD_\eps M_{DA}A_\eps$. Proceeding similarly to~\ref{item:approx-auxpoint},
     we label the \emph{inner} angles (starting from $B$)
     by $\zeta_\eps$, $\eta_\eps$, $\theta_\eps$, $\iota_\eps$ and $\lambda_\eps$, $\mu_\eps$, $\nu_\eps$, $\xi_\eps$ and define
     \begin{align*}
      \sigma_\eps &:= \angle\br{\overrightarrow{A_\eps B},\overrightarrow{BB_\eps}}, &
      \tau_\eps &:= \angle\br{\overrightarrow{C_\eps D},\overrightarrow{DD_\eps}},
     \end{align*}
     see Figure~\ref{fig:fm6}.
     \begin{figure}
      \centering
      \includegraphics[scale=.5]{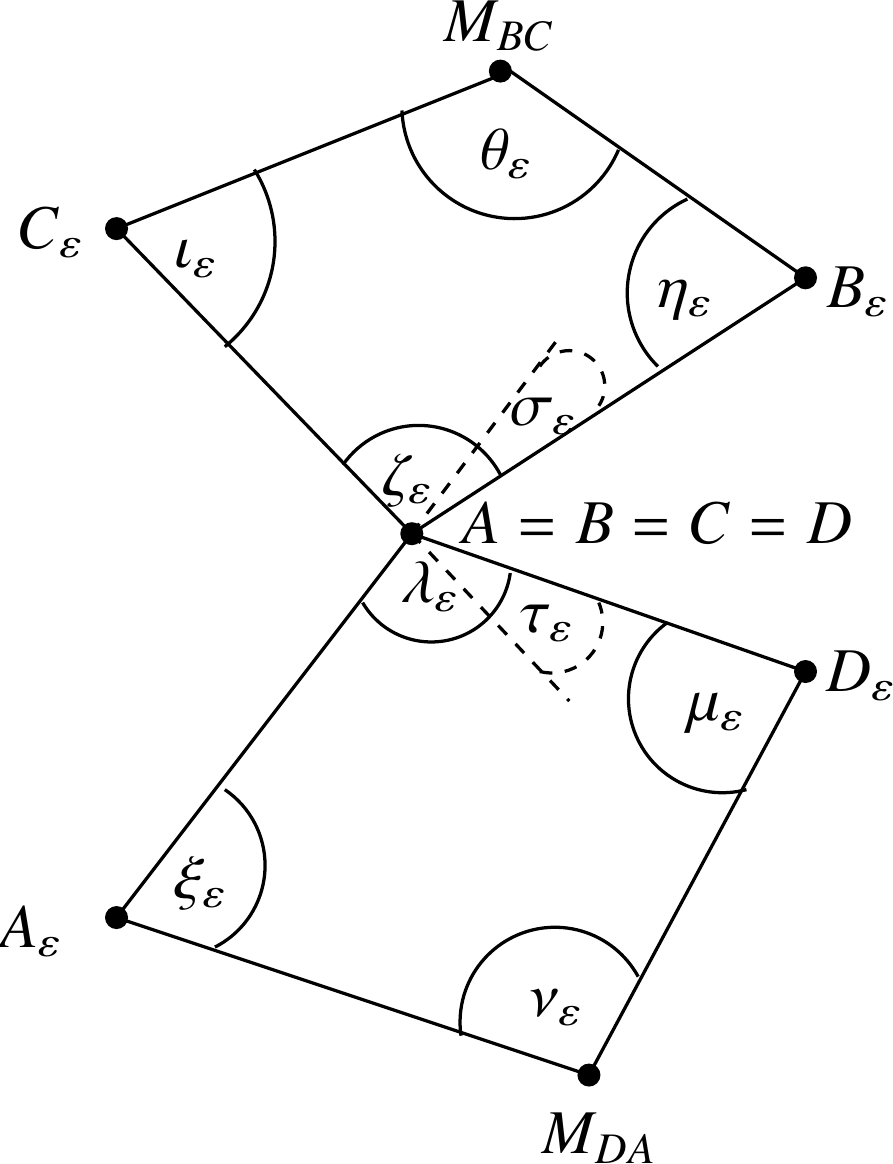}\qquad
      \includegraphics[scale=.5]{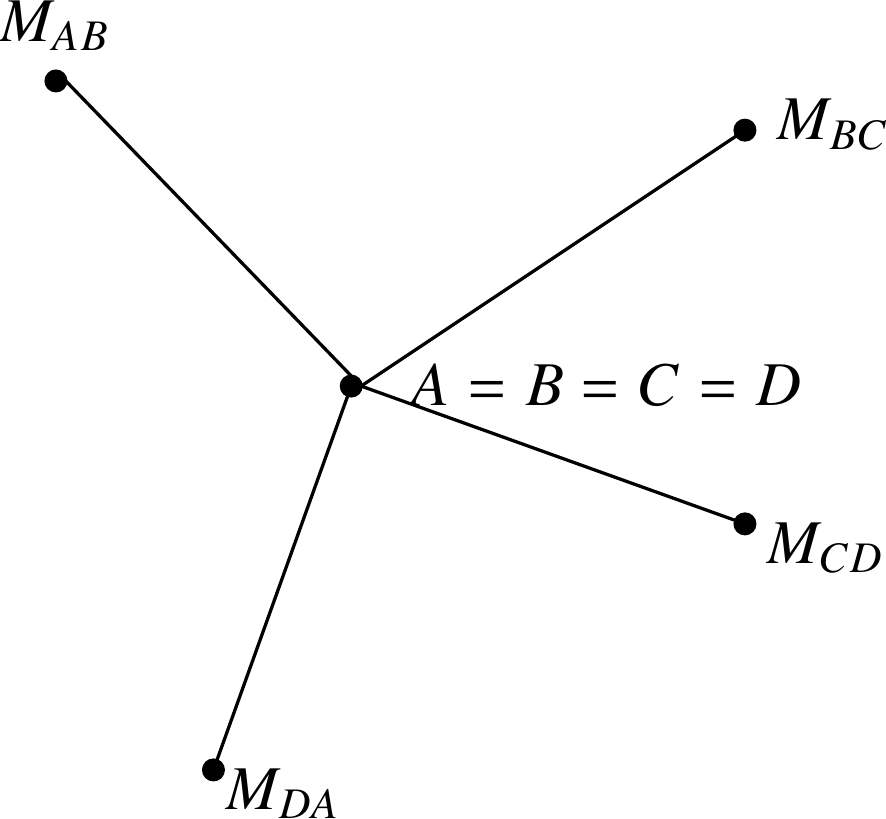}
      \caption{(left) Situation~\ref{item:2loops}. Note that both quadrilaterals $BB_\eps M_{BC}C_\eps$
      and $DD_\eps M_{DA}A_\eps$ are, in general, non-planar.
      (right) Situation~\ref{item:4loops}.}
      \label{fig:fm6}
      \label{fig:fm7}
     \end{figure}
     For the angular sum in a non-planar quadrilateral we obtain
     \begin{align*}
      \zeta_\eps+\eta_\eps+\theta_\eps+\iota_\eps&\le2\pi & \lambda_\eps+\mu_\eps+\nu_\eps+\xi_\eps\le2\pi.
     \end{align*}
     The total curvature of the polygon amounts to the sum of \emph{exterior} angles
     \begin{align*}
      &\sigma_\eps + \br{\pi-\eta_\eps} + \br{\pi-\theta_\eps} + \br{\pi-\iota_\eps}
      + \tau_\eps + \br{\pi-\mu_\eps} + \br{\pi-\nu_\eps} + \br{\pi-\xi_\eps} \\
      &\ge \sigma_\eps + \tau_\eps + 2\pi + \zeta_\eps + \lambda_\eps.
     \end{align*}
     From Lemma~\ref{lem:apprtang} we infer $\zeta_\eps,\lambda_\eps\nearrow\pi$.
     The case $({<},{=},{<},{=})$ is shifted by one position.
     \item\label{item:3loops} Three loops: the case $({<},{<},{<},{=})$ leads to the polygon
     \begin{equation*}
      AM_{AB}BM_{BC}CM_{CD}D
     \end{equation*}
     which is treated similarly
     to~\ref{item:spherical}; here $\overrightarrow{CM_{CD}}$ plays the r\^ole of $\overrightarrow{CD}$
in~\ref{item:spherical}.
     The shifted cases $({=},{<},{<},{<})$, $({<},{=},{<},{<})$, $({<},{<},{=},{<})$ are symmetric.
     \item\label{item:4loops} Four loops: As~\ref{item:3loops} in fact works for $({<},{<},{<},{\le})$
     it also covers the situation $({<},{<},{<},{<})$. Alternatively we consider the polygon
     \begin{equation*}
      AM_{AB}BM_{BC}CM_{CD}DM_{DA}A
     \end{equation*}
     as drawn in Figure~\ref{fig:fm7} with
     \begin{align*}
      &\angle\br{\overrightarrow{AM_{AB}},\overrightarrow{M_{AB}B}}
      =\angle\br{\overrightarrow{BM_{BC}},\overrightarrow{M_{BC}C}}
      =\angle\br{\overrightarrow{CM_{CD}},\overrightarrow{M_{CD}D}}\\
      &=\angle\br{\overrightarrow{DM_{DA}},\overrightarrow{M_{DA}A}}
      =\pi.
     \end{align*}
   \end{enumerate}
  \end{enumerate}
 \end{enumerate}
\end{proof}

\end{appendix}

%%%%%%%%%%%%%%%%%%%%%%%%%
% Bibliography

\end{document}